\documentclass[11pt]{amsart}
\usepackage{setspace}
\usepackage{epigamath}

\usepackage[english]{babel}

\numberwithin{equation}{section}

\usepackage[shortlabels]{enumitem}

\usepackage{amsmath,amsthm,graphicx,mathtools,multicol}

\input{insbox}
\usepackage{float}
\usepackage{subcaption}
\usepackage{amssymb}

\usepackage{tikz}
\usetikzlibrary{patterns,patterns.meta,arrows.meta,calc,3d,perspective,fadings,shadows,cd}

\theoremstyle{plain}
	\newtheorem{thm}{Theorem}[section]
	\newtheorem{cor}[thm]{Corollary}
	\newtheorem{lem}[thm]{Lemma}
	\newtheorem{prop}[thm]{Proposition}

	\newtheorem*{lem_casc_poinca_dual}{Lemma~\ref{lm:casc_poinca_dual}}
	\newtheorem*{thm_symmetry}{Theorem~\ref{thm:symmetry}}
	\newtheorem*{thm_charact_vanishing}{Theorem~\ref{thm:charact_vanishing}}
	\newtheorem*{cor_charact_degeneracy}{Corollary~\ref{cor:charact_degeneracy}}
	\newtheorem*{cor_inequalities}{Corollary~\ref{cor:inequalities}}
	\newtheorem*{cor_odd_deg}{Corollary~\ref{cor:odd_deg}}
	\newtheorem*{cor_dege_viro_even}{Corollary~\ref{cor:dege_viro_even}}
	\newtheorem*{thm_rank_max}{Theorem~\ref{thm:rank_max}}
	\newtheorem*{prop_ell_iota}{Proposition~\ref{prop:ell_iota}}
    \newtheorem*{conj}{Conjecture}
        
\theoremstyle{definition}
	\newtheorem{dfn}[thm]{Definition}

	\newtheorem*{ntns}{Notation}
	\newtheorem*{dfn_dege_index}{Definition~\ref{dfn:dege_index}}
	\newtheorem*{dfn_rank_X}{Definition~\ref{dfn:rank_X}}
        
\theoremstyle{remark}        
	\newtheorem{rem}[thm]{Remark}
	\newtheorem{rems}[thm]{Remarks}
	\newtheorem{ex}[thm]{Example}
	
	\newtheorem{exs}[thm]{Examples}

\newcommand{\bigslant}[2]{{\raisebox{.2em}{$#1$}\left/\raisebox{-.2em}{$#2$}\right.}}
\newcommand{\smallslant}[2]{{\raisebox{.1em}{$#1$}\left/\raisebox{-.1em}{$#2$}\right.}}
\newcommand{\C}{\mathbb{C}}
\newcommand{\R}{\mathbb{R}}

\newcommand{\Z}{\mathbb{Z}}
\newcommand{\N}{\mathbb{N}}
\newcommand{\T}{\mathbb{T}}
\newcommand{\F}{\mathbb{F}}
\renewcommand{\P}{\mathrm{P}}

\newcommand{\Sd}{\mathrm{Sd}\,}
\renewcommand{\O}{\mathcal{O}}
\newcommand{\K}{\mathcal{K}}
\newcommand{\df}{\mathrm{d}}
\renewcommand{\D}{\mathrm{D}}
\newcommand{\X}{{\R X_\varepsilon}}
\newcommand{\m}{\mathfrak{m}}
\newcommand{\x}{\mathrm{x}}
\newcommand{\vect}{\mathbf{Vect.}^\mathbf{f}_{\F_2}}
\newcommand{\hopf}{\mathbf{Hopf}^\mathbf{f}_{\F_2}}
\newcommand{\gralg}{\mathbf{gr.Alg.}^\mathbf{f}_{\F_2}}
\newcommand{\grvect}{\mathbf{gr.Vect.}^\mathbf{f}_{\F_2}}
\newcommand{\gr}{\mathrm{gr}} 

\newcommand{\op}{\mathrm{op}}
\newcommand{\bv}{\mathrm{bv}}
\newcommand{\aug}{\mathrm{aug}}
\newcommand{\ext}{\mathrm{ext}}
\newcommand{\Sph}{\mathrm{S}}

\DeclareMathOperator{\Sed}{Sed}
\DeclareMathOperator{\Hom}{Hom}
\DeclareMathOperator{\Spec}{Spec}
\DeclareMathOperator{\Arg}{Arg}
\DeclareMathOperator{\rk}{rk}

\DeclareMathOperator{\im}{im}

\setlist[enumerate,1]{label={\rm(\arabic*)}, ref={\rm\arabic*}}
\newcommand{\relmiddle}[1]{\mathrel{}\middle#1\mathrel{}}

\newcommand{\supth}[1]{\ensuremath{#1^{\mathrm{th}}}}

\EpigaVolumeYear{10}{2026} \EpigaArticleNr{9} \ReceivedOn{December 19, 2023}
\InFinalFormOn{October 22, 2025}
\AcceptedOn{December 19, 2025}

\title{Poincar\'e duality, degeneracy, and real Lefschetz property for T-hypersurfaces}
\titlemark{Poincar\'e duality, degeneracy, and real Lefschetz property for T-hypersurfaces}

\author{Jules Chenal}
\address{Universitetet i Oslo, Moltke Moes vei 35, Niels Henrik Abels hus, 0851 Oslo, Norway}
\email{julesche@math.uio.no}

\authormark{J.~Chenal}

\AbstractInEnglish{In this article, we present two structural results about the Renaudineau--Shaw spectral sequence that computes the cohomology of T-hypersurfaces. The first is a Poincar\'e duality satisfied by all its pages of positive index. The second is a vanishing criterion. It reformulates the vanishing of the boundary operators of the spectral sequence as the injectivity of some morphisms induced in cohomology by the inclusion of the T-hypersurface in its surrounding toric variety. It implies that the Renaudineau--Shaw spectral sequence of a T-hypersurface degenerates at the second page if and only if the T-hypersurface satisfies a real version of the Lefschetz hyperplane section theorem.}


\acknowledgement{This research was funded, in whole or in part, by l'Agence Nationale de la Recherche (ANR), project ANR-22-CE40-0014. The author was also supported by the Trond Mohn Foundation project ``Algebraic and Topological Cycles in Complex and Tropical Geometries''. }

\begin{document}



\maketitle

\begin{prelims}

\DisplayAbstractInEnglish

\bigskip

\DisplayKeyWords

\medskip

\DisplayMSCclass

\end{prelims}


\newpage

\setcounter{tocdepth}{1}

\tableofcontents


\section{Introduction}

Let $M$ be a free Abelian group of rank $n$. A smooth polytope of $M\otimes\R$ is a full-dimensional polytope $P$ whose vertices lie in the lattice $M$ and whose associated toric variety $Y$ is non-singular. Let $N:=\Hom(M;\Z)$ be the dual lattice. Its reduction modulo $2$ acts on $Y(\R)$ as the $2$-torsion of the real locus of the torus $\Spec(\R[M])$. The moment map 
\begin{equation*}
	\mu\colon Y(\R)\longrightarrow P
\end{equation*}
induces a homeomorphism between the quotient $Y(\R)/(N/2N)$ and $P$. Since $P$ is contractible, $\mu$ admits a section $s$. The surjective map  
\begin{equation*}
	\begin{array}{rcl}
		\left(\smallslant{N}{2N}\right)\times P& \longrightarrow & Y(\R)\\
		(v;p) & \longmapsto & v\cdot s(p),
	\end{array}
\end{equation*}
allows us to see $Y(\R)$ as the gluing of $2^n$ copies of $P$ along their faces;  \textit{cf.} \cite[Theorem 5.4]{Gel-Kap-Zel_dis_res} and Definition~\ref{dfn:RP_CW}. It endows $Y(\R)$ with a regular CW-complex structure $\R P$. It does not depend on the choice of section $s$ as two such sections differ by the action of an element of $N/2N$. Likewise, any triangulation~$K$ of $P$ can be uniquely lifted as an invariant subdivision $\R K$ of $\R P$. This subdivision might fail to be a triangulation but is always a $\Delta$-complex.\footnote{A  $\Delta$-complex is a regular CW-complex in which every closed cell is isomorphic to a simplex as CW-complexes; \textit{cf.} \cite[Chapter 2]{Hat_alg_top}.} By construction, the moment map $\mu\colon \R K\rightarrow K$  is cellular. A cellular formulation of Poincar\'e duality associates with any closed cochain $\alpha\in Z^1(\R K;\F_2)$ a subcomplex of pure codimension $1$ of the barycentric subdivision of $\R K$; \textit{cf.} \cite[Section 8.68]{Munk_ele_alg}. This ``hypersurface'' determines an $(n-1)$-cycle whose homology class is Poincar\'e dual to the cohomology class of $\alpha$. For every edge $\sigma^1=[p;q]$ of $K$, we denote by $\omega(\sigma^1)$ the image of $p-q\in M$ in $M/2M$. If, for every $v\in N/2N$, $\alpha$ satisfies
\begin{equation*}
	v^*\alpha \colon \sigma^1 \in \R K \longmapsto \alpha(\sigma^1)+\omega\left(\mu_*(\sigma^1)\right)(v),
\end{equation*}
we say that $\alpha$ is \emph{symmetric}. The hypersurfaces associated to symmetric cocycles are called T-hypersurfaces. The difference of two symmetric cocycles is always of the form $\df(\mu^*\varepsilon)$ for some $\varepsilon\in C^0(K;\F_2)$. Thus, two T-hypersurfaces are always homologous in $\R P$. There is a canonical symmetric cocycle $\omega_{\R X}$ associated with every primitive triangulation $K$. Its T-hypersurface is sometimes referred to as the T-hypersurface with constant signs. As such, T-hypersurfaces are parametrised by cochains $\varepsilon\in C^0(K;\F_2)$ called sign distributions. We denote by $\X$ the T-hypersurface associated with the cocycle $\omega_{\R X}+ \df(\mu^*\varepsilon)$. They are always PL-smooth; \textit{cf.} \cite[Proposition~4.11]{Brug-LdM-Rau_Comb_pac}. By construction, the image of their fundamental class in the homology of $Y(\R)$ is Poincar\'e dual to the class $[\omega_{\R X}]$, regardless of $\varepsilon$. This property is a way of saying that the T-hypersurfaces $(\X)_\varepsilon$ have ``the same degree''. The image of the hypersurface $\X$ by the moment map $\mu$ does not depend on $\varepsilon$ and is denoted by $X$. When $K$ is convex,\footnote{The triangulation $K$ is said to be convex when it is obtained as the projection of the compact faces of the epigraph of a convex and piecewise-affine function $\nu\colon P\rightarrow \R$.} O.~Viro's patchworking theorem, \textit{cf.} \cite[Theorem 4.3.A]{Vir_pat_rea} and \cite[Th\'eor\`eme 4.2]{Ris_con_hyp}, asserts that $\X$ is isotopic inside $Y(\R)$ to the real locus of a generic member of a family of non-singular algebraic hypersurfaces $(X_t)_{t\geq 1}$ of $Y$. In this case, $X$ represents the tropical limit of the family $(X_t)_{t\geq 1}$. A.~Renaudineau and K.~Shaw used a version of I.\,O.~Kalinin's spectral sequence, \textit{cf.} \cite{Kal_coh_cha} or \cite{Kal_coh_rea}, to derive upper bounds on the Betti numbers of $\X$ in terms of quantities associated with the triangulation; \textit{cf.} \cite[Theorem 1.4]{Ren-Sha_bou_bet}. We denote by $(E^r_{p,q}(\X))_{p,q,r\in\N}$ the Renaudineau--Shaw spectral sequence of $\X$. When $K$ is convex, these quantities are the tropical Hodge numbers of the tropical hypersurface $X$ that is dual to $K$; see \cite{Ite-Kat-Mik-Zha_tro_hom}. They correspond, regardless of the convexity of $K$, to the Hodge numbers of the zero locus of a non-singular section of the line bundle of $Y$ associated with the polytope $P$; \textit{cf.} \cite[Corollary 1.9]{Arn-Ren-Sha_Lef_sec} and \cite[Theorem 1.6]{Brug-LdM-Rau_Comb_pac}. These inequalities, when added together, specialise to the Smith--Thom inequality; \textit{cf.} \cite[Theorem 2.3.1]{Deg-Kha_top_pro}.

As pointed out by E.~Brugall\'e, L.~Lopez de Medrano, and J.~Rau in \cite{Brug-LdM-Rau_Comb_pac}, almost none of these results about $\X$ depend on the convexity of the triangulation $K$. Here we study the spectral sequence $(E^r_{p,q}(\X))_{p,q,r\in\N}$ and its dual $(E_r^{p,q}(\X))_{p,q,r\in\N}$ for an arbitrary primitive triangulation $K$, and tackle the following conjecture of A.~Renaudineau and K.~Shaw.

\begin{conj}[\textit{cf.} {\cite[Conjecture 1.10]{Ren-Sha_bou_bet}}]\label{conj:degenerate}
	The spectral sequence $(E^r_{p,q}(\X))_{p,q,r\in\N}$ degenerates at the second page.
\end{conj}

By establishing a Poincar\'e duality for the spectral sequence $(E_r^{p,q}(\X))_{p,q,r\in\N}$, we are able to show that the degeneracy of the spectral sequence is linked to a property that could be interpreted as a real analogue of the Lefschetz hyperplane section theorem.

\subsubsection*{Poincar\'e duality} First of all, we prove a general lemma about Poincar\'e duality in spectral sequences.

\begin{lem_casc_poinca_dual}[Propagation of Poincar\'e duality]
	Let $(A;d)$ be an increasingly filtered, graded, differential algebra of finite dimension over a field $\F$, and let $(E_r^{p,q}(A))_{p,q,r\geq 0}$ denote its associated spectral sequence. If there are integers $r_0\geq 1$, $m,n\geq0$ such that
	\begin{enumerate}
		\item the vector space $E^{p,q}_{r_0}(A)$ vanishes whenever $p>m$ or $q>n$,
		\item the vector space $E^{m,n}_{r_0}(A)$ has dimension~$1$,
		\item the bilinear pairing $E^{p,q}_{r_0}(A)\otimes_\F E^{m-p,n-q}_{r_0}(A)\rightarrow E^{m,n}_{r_0}(A)$ is non-degenerate,
	\end{enumerate}
	then for all $r\geq r_0$, 
	\begin{enumerate}[label=\rm{(\alph*)}]
		\item the vector space $E^{p,q}_{r}(A)$ vanishes whenever $p>m$ or $q>n$,
		\item the vector space $E^{m,n}_{r}(A)$ has dimension~$1$,
		\item the bilinear pairing $E^{p,q}_{r}(A)\otimes_\F E^{m-p,n-q}_{r}(A)\rightarrow E^{m,n}_{r}(A)$ is non-degenerate.
	\end{enumerate}
\end{lem_casc_poinca_dual}

We now describe the spectral ring $(E_r^{p,q}(\X))_{p,q,r\in\N}$. Using its structure of algebra as well as the tropical Poincar\'e duality, see \cite[Theorem 5.3]{Jel-Rau-Sha_Lef_11} and \cite[Theorem 3.3]{Brug-LdM-Rau_Comb_pac}, we show that it satisfies Poincar\'e duality.

\begin{thm_symmetry}[Poincar\'e duality]
	For all $r\geq 1$, the $E_r$-pages of the Renaudineau--Shaw spectral sequence computing the cohomology of\, $\X$ satisfy the Poincar\'e duality. That is to say, 
	\begin{enumerate}
		\item the vector space $E^{p,q}_{r}(\X)$ vanishes whenever $p>n-1$ or $q>n-1$,
		\item the vector space $E^{n-1,n-1}_{r}(\X)$ has dimension~$1$,
		\item the bilinear pairing $\cup\colon E^{p,q}_{r}(\X)\otimes E^{n-1-p,n-1-q}_{r}(\X)\rightarrow E^{n-1,n-1}_{r}(\X)$ is non-degenerate.
	\end{enumerate}%
	In particular, $E^{p,q}_{r}(\X)$ is isomorphic to $E^{n-1-p,n-1-q}_{r}(\X)$, and $\df^{p,q}_r$ and $\df^{n-1-p+r,n-2-q}_r$ have the same rank and kernel dimension.
\end{thm_symmetry}

\subsubsection*{Degeneracy and real Lefschetz property} Recall the Lefschetz hyperplane section theorem:\textit{ if $X$ is a generic section of the line bundle of $Y$ associated with the moment polytope $P$, then the restriction
\begin{equation*}
	H^q(Y(\C);\Z)\longrightarrow H^q(X(\C);\Z)
\end{equation*} 
is an isomorphism for all $q<n-1$ and is injective for $q=n-1$.} There is no real equivalent of this theorem. But if we want to define a real analogue of this property, asking that
\begin{equation*}
	i^q\colon H^q(Y(\R);\F_2)\longrightarrow H^q(X(\R);\F_2)
\end{equation*} 
is an isomorphism for all $q<\frac{n-1}{2}$ and is injective for $q=\frac{n-1}{2}$ would be extremely restrictive on the cohomology of $X(\R)$. For instance, if $X$ is obtained as a patchwork, which implies that $X(\R)$ is isotopic to a $T$-hypersurface $\X$, then it can almost never satisfy the statement and be \emph{maximal}. Maximality here means that the total Betti number of $X(\R)$ equals the total Betti number of $X(\C)$; the Smith--Thom inequality states that the former is always at most equal to the latter. Instead, we say that $X$ satisfies the \emph{real Lefschetz property}~if 
\begin{equation*}
	i^q\colon H^q(Y(\R);\F_2)\longrightarrow H^q(X(\R);\F_2) 
\end{equation*} 
is injective for all $q\leq \left\lfloor\frac{n-1}{2}\right\rfloor$. We can note that hypersurfaces of odd degree in projective spaces always satisfy the real Lefschetz property. If $r\geq 2$, the differentials of the $\supth{r}$ page of the spectral sequence $(E_r^{p,q}(\X))_{p,q,r\in\N}$ are all zero if $r$ is not congruent to $n$ modulo $2$. This is a consequence of the shape of its first page. We call these pages of index not congruent to $n$ modulo $2$ \emph{irrelevant}. Using Theorem~\ref{thm:symmetry} and the tropical Lefschetz hyperplane section theorem, \textit{cf.} \cite[Theorem 1.1]{Arn-Ren-Sha_Lef_sec} and \cite[Proposition 3.2]{Bru-Man_she_dec}, we are able link the vanishing of the differentials of relevant pages of the spectral sequence to the injectivity of the maps $i^q$.

\begin{thm_charact_vanishing}[Vanishing criterion]
	Let $r\geq 2$ be an integer congruent to $n$ modulo $2$. The differentials of the page $E_r(\X)$ vanish if and only if\, ${i^q\colon H^q(\R P;\F_2)\rightarrow H^q(\X;\F_2)}$ is injective when $q$ equals $\frac{n-r}{2}$. 
\end{thm_charact_vanishing}

And we derive the following corollary.

\begin{cor_charact_degeneracy}[Degeneracy criterion]
  Let $r\geq 2$ be an integer. The Renaudineau--Shaw spectral sequence of\, $\X$ degenerates at the $\supth{r}$ page if and only if the maps ${i^q\colon H^q(\R P;\F_2)\rightarrow H^q(\X;\F_2)}$ are injective for all ${q\leq \left\lfloor\frac{n-r}{2}\right\rfloor}$.
\end{cor_charact_degeneracy}

Corollary~\ref{cor:charact_degeneracy} establishes the equivalence between degeneracy at the second page and a property slightly weaker than the real Lefschetz property. It can be reformulated as the comparison of two quantities associated with the pair ${\X\subset \R P}$.

\begin{dfn_dege_index}
	We define the \emph{degeneracy index} of $\X$ as 
	\begin{equation*}
		r(\X):=\min\left\{r_0\geq 0\relmiddle| \df^{p,q}_r=0,\,\forall p,q\in\N,\,\forall r\geq r_0\right\}.
	\end{equation*}
\end{dfn_dege_index}

\begin{dfn_rank_X}
	The \emph{rank} of $\X$ is defined as
	\begin{equation*}
		\ell(\X):=\max\left\{q_0\geq 0\mid i^q\colon H^q(\R P;\F_2)\rightarrow H^q(\X;\F_2) \textnormal{ is injective for all }q\leq q_0\right\}.
	\end{equation*}
\end{dfn_rank_X}

Using these definitions, Corollary~\ref{cor:charact_degeneracy} can be expressed by two inequalities.

\begin{cor_inequalities}
	We have the inequalities 
	\begin{equation*}
		\ell(\X) \geq \left\lfloor \frac{n-r(\X)}{2}\right\rfloor,
	\end{equation*}
	 with equality if $r(\X)\geq 3+\frac{1-(-1)^n}{2}$, and
	\begin{equation*}
		r(\X) \leq \max\left(2;n-2\ell(\X)-1\right),
	\end{equation*}
	 with equality if\, $\ell(\X)\leq\frac{n-5}{2}$.
\end{cor_inequalities}
We should point out that Corollary~\ref{cor:inequalities} is slightly stronger than Corollary~\ref{cor:charact_degeneracy} as it incorporates the case of degeneracy at the first page in the first inequality. We can remark that the T-hypersurfaces $\X$ whose spectral sequence degenerates at the first page are precisely those that are maximal. Therefore, every maximal T-hypersurface satisfies the real Lefschetz property.

\subsubsection*{Rank and degree} The quantity $\ell$ was previously introduced for subsets of real projective spaces by I.\,O.~Kalinin. ``The rank is involved in many restrictions on topology of real algebraic hypersurfaces of a given degree''; see \cite{Vir_mut_pos}. It was studied by I.\,O.~Kalinin \cite{Kal_coh_cha,Kal_coh_rea}, V.~Nikulin \cite{Nik_int_sym}, and V.~Kharlamov \cite{Kha_add_cong} for real projective hypersurfaces. Let $V$ be a real algebraic hypersurface of the projective space~$\mathbb{P}^n$. Then, $\ell(V(\R))=n-1$ if $V$ is of odd degree, and $\ell(V(\R))\leq\left\lfloor \frac{n-1}{2} \right\rfloor$ if $V$ is of even degree; \textit{cf.} \cite[Corollary 4.2]{Kal_coh_cha}. The same holds for $\ell(\X)$ even when $K$ is non-convex. As a consequence, the spectral sequences $(E^r_{p,q}(\X))_{p,q,r\in\N}$ and $(E_r^{p,q}(\X))_{p,q,r\in\N}$ both degenerate at the second page when $P$ is an odd dilatation of a primitive simplex.

\begin{cor_odd_deg}
	The Renaudineau--Shaw spectral sequence of a T-hypersurface of odd degree in a projective space degenerates at the second page.
\end{cor_odd_deg}

 We generalise this property to a broader variety of polytopes. We define the degree of $\X$ as the cohomology class $[\omega_{\R X}]\in H^1(\R P;\F_2)$ (see Definition~\ref{dfn:T-hypersurface}) and a number $\iota[\omega_{\R X}]$ (see Definition~\ref{dfn:iota}).

\begin{prop_ell_iota}
	We have the inequality $\ell(\X)\geq \iota[\omega_{\R X}]$. 
\end{prop_ell_iota}

In particular, it follows from Corollary~\ref{cor:charact_degeneracy2} that, if $\iota[\omega_{\R X}]\geq \left\lfloor\frac{n}{2}\right\rfloor-1$, then for all $r\geq 2$, the spectral sequence degenerates at the second page. However, this proposition has some limitations since, in a cube, $\iota[\omega_{\R X}]$ is always $0$, regardless of the primitive triangulations $K$.

\subsubsection*{Viro triangulations} In Section~\ref{sec7}, we give the construction of particular triangulations $(V_d^n)_{n,d\geq 1}$ of the $d$-dilatations of primitive $n$-simplices for which we prove the following theorem. 

\begin{thm_rank_max}[Real Lefschetz property]
	If\, $K$ is a Viro triangulation $(V_d^n)_{n,d\geq 1}$ of the corresponding dilatation of the simplex $\P^n$, then the homological inclusion 
	\begin{equation*}
		i_q\colon H_q(\X;\F_2) \longrightarrow H_q(\R\P^n;\F_2) 
	\end{equation*}
	 is surjective for all $q\leq\left\lfloor\frac{n-1}{2}\right\rfloor$ and $\varepsilon\in C^0(K;\F_2)$. That is to say, a projective hypersurface obtained from a primitive patchwork on the Viro triangulation always satisfies the real Lefschetz property.
\end{thm_rank_max}

This implies the following corollary. 

\begin{cor_dege_viro_even}
	The Renaudineau--Shaw spectral sequences computing the homology and the cohomology of the hypersurface $\X\subset\R\P^n$ constructed from a Viro triangulation $K\in(V_d^n)_{n,d\geq 1}$ and a sign distribution $\varepsilon\in C^0(K;\F_2)$ degenerate at the second page. 
\end{cor_dege_viro_even}

A.~Renaudineau and K.~Shaw's conjecture can be rephrased as $r(\X)\leq 2$ or equivalently $\ell(\X)\geq\left\lfloor \frac{n}{2}\right\rfloor-1$. However, as in the case of Theorem~\ref{thm:rank_max}, we believe the stronger statement $\ell(\X)\geq\left\lfloor \frac{n-1}{2}\right\rfloor$ might even be true in full generality, that is to say that every T-hypersurface satisfies the real Lefschetz property.

\section{Duality in spectral sequences}

Let $\F$ be a field.

\begin{dfn}
	An \emph{increasingly filtered, graded, differential algebra} of finite dimension over the field $\F$ is an associative $\F$-algebra $A$ of finite dimension over $\F$ together with%
	\begin{enumerate}
		\item a graduation $A=\bigoplus_{q\geq 0} A^q$ compatible with the product in the following way: for all $\alpha\in A^q$ and all $\alpha'\in A^{q'}$, $\alpha\alpha'$ equals $(-1)^{qq'}\alpha'\alpha$ and belongs to $A^{q+q'}$;
		\item an increasing filtration $F^0A\subset F^1A \subset \cdots \subset A$ compatible with the product: for all $p,p'\in \N$, the product  $F^pAF^{p'}A$ is included in $F^{p+p'}A$; 
		\item a differential $\df\colon A\rightarrow A$ which is a morphism of graded vector spaces of degree $1$, whose square vanishes, that is compatible with the filtration, \textit{i.e.}~$\df F^pA\subset F^pA$ for all $p\in\N$, and for which the Leibniz rule holds: for all $\alpha\in A^q$ and $\alpha'\in A^{q'}$, $\df(\alpha\alpha')=(\df \alpha) \alpha'+(-1)^q \alpha\df \alpha'$.
	\end{enumerate}
\end{dfn}	
	
Such an object is in particular a finite-dimensional cochain complex. The compatibility of the filtration with the differential allows us to use the techniques of spectral sequences to compute the cohomology of $A$. We denote the spectral sequence associated with the filtration $(F^pA)_{p\in\N}$ by $(E^{p,q}_r(A))_{p,q,r\in\N}$. We adopt the following index convention. For all integers $p,q,r\geq 0$, 
\begin{equation*}
	E^{p,q}_r(A)\coloneqq\frac{Z^{p,q}_r(A)+F^{p-1}A^q}{\df Z^{p+r-1,q-1}_{r-1}(A)+F^{p-1}A^q},
\end{equation*}
where 
\begin{equation*}
	Z^{p,q}_r(A)\coloneqq\left\{\alpha \in F^pA^q\relmiddle| \df\alpha\in F^{p-r}A^{q+1}\right\} 
\end{equation*}
and $Z^{p,q}_{-1}(A)\coloneqq0$. In this setting, the differential $\df_r$ of the $\supth{r}$ page has bidegree $(-r;+1)$ and is the factorisation of the restriction of $\df$. By the definition of the objects we consider here, we deal with first-quadrant spectral sequences arising from bounded filtrations, hence, the spectral sequence $(E^{p,q}_r(A))_{p,q,r\geq 0}$ converges towards the cohomology of $A$; \textit{cf.} \cite[Theorem 2.6]{McC_use_gui}. The compatibility of the filtration with the product ensures that $(E^{p,q}_r(A))_{p,q,r\in\N}$ is even a \emph{spectral ring}\index{Spectral Ring}. This means that there is a well-defined product 
\begin{equation*}
	E^{p,q}_r(A)\otimes_\F E^{p',q'}_r(A) \longrightarrow E^{p+p',q+q'}_r(A)
\end{equation*}
which satisfies, for all $\alpha\in E^{p,q}_r(A)$ and $\alpha'\in E^{p',q'}_r(A)$, the graded commutativity%
\begin{equation*}
	\alpha\alpha'=(-1)^{qq'}\alpha'\alpha
\end{equation*}%
 and  the Leibniz rule%
\begin{equation*}
	\df_r^{p+p',q+q'}(\alpha\alpha')=	\df_r^{p,q}(\alpha)\alpha'+(-1)^q\alpha\df_r^{p',q'}(\alpha').
\end{equation*}%
 These constructions are explained in \cite[Section 2.3]{McC_use_gui},  although we should emphasise that we did not adopt the same index convention as J.~McCleary. We justify our choice by the degree of generality of the objects we discuss here. Since we do not have much context about the arising of the spectral sequence, we do not find meaningful to perform the change of index often used with double complexes for which it is particularly adapted.
\begin{dfn}
	Let A be an increasingly filtered, graded, differential algebra of finite dimension over the field $\F$. We denote by $A^*$ the dual chain complex. By the universal coefficient theorem, \textit{cf.} \cite[Theorem~3.3a]{Car-Eil_hom_alg}, the homology of $A^*$ is dual to the cohomology of $A$. 
\end{dfn}

P.~Deligne gave a definition of the dual filtration of a filtered object of an Abelian category. It is a filtration of the same object in the opposite category. It is defined in such a way that the graded objects of the dual filtration are dual to the graded objects of the initial filtration. This construction is designed to be used with contravariant functors. 

\begin{dfn}[\textit{cf.} {\cite[Equation~(1.1.6)]{Del_the_hod}}]\label{dfn:dual_filtration}
	Let $\mathbf{C}$ be an Abelian category. We denote by $\mathbf{C}^\op$ the opposite category and by $(-)^\op\colon (\mathbf{C}^\op)^\op\rightarrow\mathbf{C}$ the ``identical contravariant functor''. Let $V$ (resp.\ $W$) be an object of~$\mathbf{C}$ endowed with an increasing (resp.\ decreasing) filtration $(V^{(k)})_{k\in \Z}$ (resp.\ $(W_{(k)})_{k\in \Z}$). The dual decreasing (resp.\ increasing) filtration of $V^\op$ (resp.\ $W^\op$) is defined, for all $k\in\Z$, by the formula%
	\begin{equation*}
		(V^\op)_{(k)}\coloneqq\im\left(\left( V \twoheadrightarrow \bigslant{V}{V^{(k-1)}_{\;}} \right)^\op\right)\quad\left(\text{resp.}\; (W^\op)^{(k)}\coloneqq\im\left(\left( W \twoheadrightarrow \bigslant{W}{W_{(k+1)}} \right)^\op\right) \right),
	\end{equation*}%
	 so that $(V^{(k)}/V^{(k-1)})^\op$ is naturally isomorphic to $(V^\op)_{(k)}/(V^\op)_{(k+1)}$; \textit{cf.} \cite[Equation~(1.1.7)]{Del_the_hod}. If $\mathbf{C}$ is the category of vector spaces for instance (or even of cochain complexes) over a field $\F$, applying the usual duality functor $\Hom_\F(-;\F)$ to this abstract construction yields a filtration, in the usual sense, of opposite growth of the dual vector space (or chain complex).  
\end{dfn}

Following this definition, the dual chain complex $A^*$ is naturally endowed with a dual decreasing filtration of chain complexes $0\subset \cdots \subset F_1A^*\subset F_0A^*=A^*$. We denote by $(E^r_{p,q}(A^*))_{r,p,q\geq 0}$ the associated spectral sequence. It is defined, for all integers $p,q,r\geq 0$, by%
\begin{equation*}
	E_{p,q}^r(A^*)\coloneqq\frac{Z_{p,q}^r(A^*)+F_{p+1}A^*_q}{\partial Z_{p-r+1,q+1}^{r-1}(A^*)+F_{p+1}A^*_q},
\end{equation*}%
where $\partial=\Hom_\F(\df;\F)$ and%
\begin{equation*}
	Z_{p,q}^r(A^*)\coloneqq\left\{a \in F_pA^*_q\mid \partial a\in F_{p+r}A^*_{q-1}\right\},
\end{equation*}%
and we have the convention $Z_{p,q}^{-1}(A^*)\coloneqq 0$. The induced boundary operators $\partial^r$ have bidegree $(+r;-1)$.

\begin{prop}\label{prop:dual_spec_seq}
	The spectral sequences $(E^r_{p,q}(A^*))_{p,q,r\geq 0}$ and $(E_r^{p,q}(A))_{p,q,r\geq0}$ are dual to each other through a collection of duality pairings%
	\begin{equation*}
		\langle-\,;-\rangle\colon E_r^{p,q}(A)\otimes_\F E^r_{p,q}(A^*) \longrightarrow \F 
	\end{equation*}%
	 defined for all integers $p,q,r\geq 0$. Moreover, $\df_r^{p,q}$ is the adjoint of\, $\partial^r_{p-r,q+1}$ relatively to the pairing.
\end{prop} 

\begin{proof}
  Let $q\in\N$, and denote by $\langle-\,;-\rangle\colon A^q\otimes A^*_q\rightarrow \F$ the usual duality pairing. By definition $\df$ and $\partial$ are adjoints of each other; \textit{i.e.} ${\langle \df \alpha; a\rangle=\langle \alpha; \partial a\rangle}$ for all ${\alpha\in A^q}$ and all ${a\in A^*_{q+1}}$. Since $F_{p+1}A^*_q$ is the image in $A_q^*$ of $(A^q/F^pA^q)^*$, it is precisely the vector space ${\{a\in A_q^*\mid\langle \alpha\,;a\rangle=0,\,\forall \alpha\in F^pA^q\}}$. It follows that the restriction of ${\langle-\,;-\rangle}$ to the sum of the subspaces%
	\begin{equation*}
		\left(Z^{p,q}_r(A)+F^{p-1}A^q\right)\otimes_\F\left(\partial Z_{p-r+1,q+1}^{r-1}(A^*)+F_{p+1}A^*_q\right)
	\end{equation*}%
	 and%
	\begin{equation*}
		\left(\df Z^{p+r-1,q}_{r-1}(A)+F^{p-1}A^q\right)\otimes_\F\left(Z_{p,q}^r(A^*)+F_{p+1}A^*_q\right) 
	\end{equation*}	%
	 vanishes. Thus $\langle-;-\rangle$ factors through the quotient map to give rise to a well-define bilinear product%
	\begin{equation*}
		\langle-\,;-\rangle\colon E_r^{p,q}(A)\otimes_\F E^r_{p,q}(A^*) \longrightarrow \F.
	\end{equation*}%
	The morphisms $(\df^{p,q}_r)_{p,q,r\in\N}$ and $(\partial_{p,q}^r)_{p,q,r\in\N}$ are factorisations of $\df$ and $\partial$, respectively. Hence, they are adjoints of each others for the new pairings. By the definition of the dual filtration, the pairings ${\langle-\,;-\rangle}$ defined on ${E_0^{p,q}(A)\otimes_\F E^0_{p,q}(A^*)}$ are non-degenerate, for all $p,q\in\N$. Moreover, the fundamental property of spectral sequences, namely the one that enables us to compute a page as the cohomology (or homology) of the preceding one, allows us to repetitively use the universal coefficients theorem, \textit{cf.}  \cite[Theorem 3.3a]{Car-Eil_hom_alg}, and deduce that every pairing is non-degenerate. We should emphasise that we only considered finite-dimensional chain and cochain complexes of vector spaces. 
\end{proof}

\begin{lem}[Propagation of Poincar\'e duality]\label{lm:casc_poinca_dual}
	Let $(A;d)$ be an increasingly filtered, graded, differential algebra of finite dimension over the field $\F$, and let $(E_r^{p,q}(A))_{p,q,r\geq 0}$ denote its associated spectral sequence. If there are integers $r_0\geq 1$, $m,n\geq0$ such that%
	\begin{enumerate}
		\item the vector space $E^{p,q}_{r_0}(A)$ vanishes whenever $p>m$ or $q>n$, 
		\item the vector space $E^{m,n}_{r_0}(A)$ has dimension~$1$, 
		\item the bilinear pairing $E^{p,q}_{r_0}(A)\otimes_\F E^{m-p,n-q}_{r_0}(A)\rightarrow E^{m,n}_{r_0}(A)$ is non-degenerate,
	\end{enumerate}
	then for all $r\geq r_0$, 
	\begin{enumerate}[label={\rm (\alph*)}]
		\item the vector space $E^{p,q}_{r}(A)$ vanishes whenever $p>m$ or $q>n$, 
		\item the vector space $E^{m,n}_{r}(A)$ has dimension~$1$, 
		\item the bilinear pairing $E^{p,q}_{r}(A)\otimes_\F E^{m-p,n-q}_{r}(A)\rightarrow E^{m,n}_{r}(A)$ is non-degenerate.
	\end{enumerate}
\end{lem}

\begin{proof}
	By recursion, we only need to prove the statement for $r_0+1$. The first assertion is a consequence of $E^{p,q}_{r_0+1}(A)$ being a sub-quotient of $E^{p,q}_{r_0}(A)$. The second follows from the isomorphism between $E^{m,n}_{r_0+1}(A)$ and the $(m,n)$-cohomology group of $\df_{r_0}$. Indeed, around $E_{r_0}^{n,m}$ the cochain complex is the following:%
	\begin{equation*}
		\begin{tikzcd}[column sep=2cm]
			E^{m+r_0,n-1}_{r_0}(A) \ar[r,"d^{m+r_0,n-1}_{r_0}"] & E^{m,n}_{r_0}(A) \ar[r,"d^{m,n}_{r_0}"] & E^{m-r_0,n+1}_{r_0}(A).
		\end{tikzcd}
	\end{equation*}%
	 By assumption, both the leftmost group and the rightmost group vanish. Hence, $E^{m,n}_{r_0+1}(A)$ is isomorphic to $E^{m,n}_{r_0}(A)$ and has dimension~$1$. For the last part, we choose an isomorphism $\int\colon E^{m,n}_{r_0}(A)\rightarrow \F$ and consider the unique isomorphism%
	\begin{equation*}
	\D_{r_0}^{\,p,q}\colon E^{p,q}_{r_0}(A)\longrightarrow \left(E^{m-p,n-q}_{r_0}(A)\right)^* \cong E^{r_0}_{m-p,n-q}(A^*) 
	\end{equation*}%
satisfying $\alpha\left(\D_{r_0}^{\,p,q}(\beta)\right)=\int \alpha \beta$ for all $\alpha\in E^{m-p,n-q}_{r_0}(A)$ and $\beta\in E^{p,q}_{r_0}(A)$. This is, up to sign, a (co-)chain complex isomorphism since%
	\begin{equation*}
		\begin{split}
			\left\langle \alpha\,;\D_{r_0}^{\,p-r_0,q+1}\left(\df^{p,q}_{r_0}\beta\right)\right\rangle&=\int \alpha \,\df^{p,q}_{r_0}\beta\\
			&= (-1)^{q}\int \left(\df^{m-p+r_0,n-q-1}_{r_0}\alpha\right)\beta\\
			&= (-1)^{q} \left\langle\alpha\,;\partial^{r_0}_{m-p,n-q}\D_{r_0}^{\,p,q}(\beta)\right\rangle.
		\end{split}
	\end{equation*}%
	 It induces, nonetheless, an isomorphism $(\D^{p,q}_{r_0})^*$ between $E^{p,q}_{r_0+1}(A)$ and $E^{r_0+1}_{m-p,n-q}(A^*)$. We now observe  that the isomorphism between $E^{m,n}_{r_0+1}(A)$ and $E^{m,n}_{r_0}(A)$ induces an isomorphism $\int\colon E^{m,n}_{r_0+1}(A)\rightarrow \F$ such that for any two $\alpha\in E^{m-p,n-q}_{r_0+1}(A)$ and ${\beta\in E^{p,q}_{r_0+1}(A)}$, respectively represented by the two cocyles $\alpha'\in E^{m-p,n-q}_{r_0}(A)$ and ${\beta'\in E^{p,q}_{r_0}(A)}$, we have%
	\begin{equation*}
		\int\alpha\beta = \int\alpha'\beta'.
	\end{equation*}%
	 Let us consider the bilinear pairing constructed from the product and the induced integration
	\begin{equation*}
		{E^{m-p,n-q}_{r_0+1}(A)\otimes_\F E^{p,q}_{r_0+1}(A)\longrightarrow\F}.
	\end{equation*}
	Its associated morphism%
	\begin{equation*}
		\D^{p,q}_{r_0+1}\colon E^{p,q}_{r_0+1}(A)\longrightarrow E_{m-p,n-q}^{r_0+1}(A^*)
	\end{equation*}
	is induced in cohomology by $\D^{p,q}_{r_0}$. It is an isomorphism of chain complexes, so the third assertion follows.
\end{proof}

A direct consequence of Lemma~\ref{lm:casc_poinca_dual} is the symmetry of the pages of the spectral sequence following the $\supth{r_0}$ page. 

\begin{prop}\label{prop:symm_spec_seq}
	Let $(A;d)$ be a increasingly filtered, graded, differential algebra of finite dimension over the field $\F$, and let $(E_r^{p,q}(A);\df^{p,q}_r)_{p,q,r\geq 0}$ denote its associated spectral sequence. If the spectral sequence of $A$ satisfies the Poincar\'e duality at the $\supth{r_0}$ page, then for all $r\geq r_0$ and all $p,q\geq 0$, %
	\begin{equation*}
		\D_r^{\,p,q}\colon E^{p,q}_r(A) \overset{\cong}{\longrightarrow} E_{m-p,n-q}^r(A^*) 
	\end{equation*}%
	 and%
	\begin{equation*} 
		\df^{p,q}_{r}=(-1)^q\left(\D_{r}^{p-r,q+1}\right)^{-1}\circ\left(\df^{m-p+r,n-q-1}_{r}\right)^*\circ\D^{p,q}_{r}.
	\end{equation*}%
	 In particular, $E^{p,q}_r(A)$ and $E^{m-p,n-q}_r(A)$ have the same dimension, and $\df^{p,q}_{r}$ and $\df^{m-p+r,n-q-1}_{r}$ have same rank and kernel dimension.
\end{prop}

\section{Basic objects and notation}

\begin{dfn}
	Let $M$ denote a free Abelian group of finite rank $n\geq1$, and let $N$ denote its dual $\Hom_\Z(M;\Z)$. Let $P$ be a polytope of $M\otimes\R$.\begin{enumerate}
	\item For a face $Q$ of $P$, denoted by $Q\leq P$, we denote its tangent space by $TQ$.
	\item We say that $P$ is \emph{full-dimensional} if $TP=M\otimes \R$. 
	\item We say that $P$ is \emph{integral} if its vertices lie in the lattice $M$. In this case, the tangent vector space $TQ$ of every face $Q$ of $P$ is rational.
	\item If $P$ is full-dimensional and integral, and $Q$ is a $1$-codimensional face of $P$, we denote by $v_Q$ the generator of the orthogonal lattice $\{v\in N\mid \alpha(v)=0,\,\forall \alpha\in TQ\}\cong \Z$ that is non-negative over $P-p$, for any point $p\in Q$.
 	\item We say that $P$ is \emph{smooth} if it is full-dimensional, integral, and if, for every vertex $p\in P$, the set $\{v_Q\colon \dim Q=n-1 \textnormal{ and } p\in Q\}$ is a basis of $N$.
	\end{enumerate}
\end{dfn}
 The polytope $P$ is smooth if and only if it is the moment polytope of a non-singular projective toric variety; \textit{cf.} \cite[Chapter 2]{Ful_tor_var}. We note that for algebraic moment maps, the image polytope is necessarily integral; \textit{cf.} \cite[Section 4.2]{Ful_tor_var}.

\begin{figure}[!ht]
	\centering
	\begin{tikzpicture}[scale=2]
		\coordinate (a) at (0,0);
		\coordinate (b) at (2,0);
		\coordinate (c) at ($(a)+(60:2)$);
		\coordinate (a1) at ($(a)+(4,0)$);
		\coordinate (b1) at ($(b)+(4,0)$);
		\coordinate (c1) at ($(c)+(4,0)$);
		\coordinate (center) at ($(a)!0.33333333!(a1)+(a)!0.33333333!(b1)+(a)!0.33333333!(c1)$);
		\filldraw[pattern=dots, thick] (a) -- (b)-- (c) -- cycle;
		\draw (a1) -- ($(a1)!(b1)!(c1)$) -- (center) -- ($(b1)!(c1)!(a1)$)-- cycle;
		\draw (b1) -- ($(b1)!(c1)!(a1)$) -- (center) -- ($(b1)!(a1)!(c1)$) -- cycle;
		\filldraw[pattern=dots,thick] (c1) -- ($(b1)!(a1)!(c1)$) -- (center) -- ($(c1)!(b1)!(a1)$) -- cycle;
		\fill (a) circle (.05);
		\fill (a1) circle (.05);
		\fill (b) circle (.05);
		\fill (b1) circle (.05);
		\fill (c) circle (.05);
		\fill (c1) circle (.05);
		\fill ($(a1)!.5!(b1)$) circle (.05);
		\fill ($(c1)!.5!(b1)$) circle (.05);
		\fill ($(a1)!.5!(c1)$) circle (.05);
		\fill (center) circle (0.05);
		\draw[->,thick] ($(b)+(0,1)$) -- ($(a1)+(0,1)$);
		\draw (a) node[anchor=north east]{$0$};
		\draw (b) node[anchor=north west]{$1$};
		\draw (c) node[anchor=south]{$2$};
		\draw ($(a)!.5!(b)$) node[anchor=north]{$[0;1]$};
		\draw ($(b)!.5!(c)$) node[anchor=south west]{$[1;2]$};
		\draw ($(a)!.5!(c)$) node[anchor=south east]{$[0;2]$};
		\draw (1,.666) -- (-.25,.5) node[anchor=east]{$[0;1;2]$};
		\draw (a1) node[anchor=north east]{$0\leq 0$};
		\draw ($(a1)+(1.5,0)$) node[anchor=north]{$1\leq [0;1]$};
		\draw ($(a1)+(1,1)$) -- ($(a1)+(1.5,1.5)$) node[anchor=west]{$2\leq [0;1;2]$};
		\draw[thick] ($(a1)!.5!(b1)$) -- (b1);
	\end{tikzpicture}
	\caption{The cubical subdivision of the triangle. Some cubical cells are marked by the pair of cells that represents them.}
	\label{fig:cubical_triang}
\end{figure}
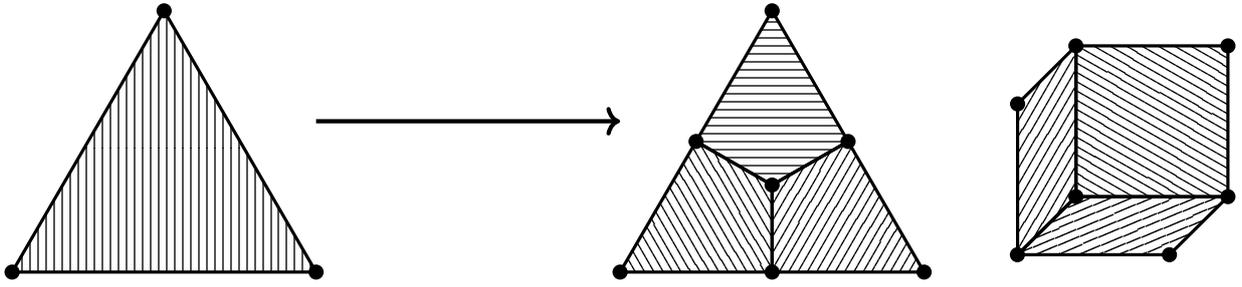

\begin{dfn}
	Let $P$ be a smooth polytope of $M\otimes\R$. A \emph{primitive} triangulation $K$ of $P$ is a triangulation whose vertices are integral, \textit{i.e.}~lie in the lattice $M$, and whose $n$-simplices have minimal normalised lattice volume, \textit{i.e.}~$\frac{1}{n!}$. For a $p$-simplex $\sigma^p$ of $K$, we denote its tangent space by $T\sigma^p\subset M\otimes\R$. It is the vector direction of the affine space spanned by its vertices. 
\end{dfn}

\begin{dfn}
	We say that a regular CW-complex $K$ is a \emph{$\Delta$-complex} if all of its closed cells are isomorphic to simplices as CW-complexes. 
\end{dfn}

\begin{dfn}
	Let $K$ be a $\Delta$-complex and $K''$ be its barycentric subdivision. Every flag of simplices of~$K$ defines an open simplex of $K''$. For a pair $\sigma^p\leq \sigma^q$ of simplices of $K$, the \emph{open cubical cell} associated with $\sigma^p\leq \sigma^q$ is the union of every open barycentric simplex whose indexing flag starts with $\sigma^p$ and ends with $\sigma^q$. The \emph{cubical subdivision} of $K$ is the regular cellular complex $K'$ on the support of $K$ whose open cells are given by the open cubical cells. In $K'$, the cell indexed by $\sigma^p\leq\sigma^q$ is a face of the cell indexed by $\sigma^{p'}\leq\sigma^{q'}$ if and only if $\sigma^{p'}\leq\sigma^p\leq\sigma^q\leq\sigma^{q'}$.
\end{dfn}
	
	By construction, $K''$ is finer than $K'$, and $K'$ is finer than $K$. The name ``cubical subdivision'' comes from the fact that every closed cell of $K'$, seen as a subcomplex of $K''$, has the combinatorics of a triangulation of a cube. An example of cubical subdivision is depicted in Figure~\ref{fig:cubical_triang}: the triangle is subdivided into three squares.

\begin{dfn}
	Let $P$ be a smooth polytope endowed with a primitive triangulation $K$. The \emph{dual hypersurface} of $K$ is the subcomplex $X$ of the cubical subdivision of $K$ made of the closed cells indexed by the pairs of simplices $\sigma^p\leq\sigma^q$ for which $p\geq 1$. See Figure~\ref{fig:dual_hypersurface} for examples. 
\end{dfn}

\begin{figure}[!ht]
	\centering
	\begin{subfigure}[t]{0.45\textwidth}
		\centering
		\begin{tikzpicture}[scale=2]
			\coordinate (a) at (0,0);
			\coordinate (b) at (1.5,0);
			\coordinate (c) at ($(a)+(60:1.5)$);
			\coordinate (d) at ($(a)+(-120:1.5)$);
			\coordinate (e) at ($(d)+(1.5,0)$);
			\coordinate (f) at ($(e)+(1.5,0)$);
			\coordinate (center) at ($(a)!0.33333333!(a)+(a)!0.33333333!(b)+(a)!0.33333333!(c)$);
			\coordinate (center1) at ($(a)!0.33333333!(a)+(a)!0.33333333!(b)+(a)!0.33333333!(e)$);
			\coordinate (center2) at ($(a)!0.33333333!(a)+(a)!0.33333333!(d)+(a)!0.33333333!(e)$);
			\coordinate (center3) at ($(a)!0.33333333!(b)+(a)!0.33333333!(e)+(a)!0.33333333!(f)$);
			\draw (c) -- (d) -- (f) -- cycle;
			\draw (a) -- (b) -- (e) -- cycle;
			\fill ($(a)!.5!(b)$) circle (.05);
			\fill ($(c)!.5!(b)$) circle (.05);
			\fill ($(a)!.5!(c)$) circle (.05);
			\fill ($(a)!.5!(d)$) circle (.05);
			\fill ($(a)!.5!(e)$) circle (.05);
			\fill ($(e)!.5!(b)$) circle (.05);
			\fill ($(e)!.5!(f)$) circle (.05);
			\fill ($(b)!.5!(f)$) circle (.05);
			\fill ($(d)!.5!(e)$) circle (.05);
			\fill (center) circle (0.05);
			\fill (center1) circle (0.05);
			\fill (center2) circle (0.05);
			\fill (center3) circle (0.05);
			\draw[very thick] ($(d)!.5!(e)$) -- (center2) -- (center1) -- (center3) -- ($(e)!.5!(f)$) ;
			\draw[very thick] (center3) -- ($(b)!.5!(f)$);
			\draw[very thick] (center2) -- ($(a)!.5!(d)$);
			\draw[very thick] ($(a)!.5!(c)$) -- (center) -- (center1);
			\draw[very thick] ($(b)!.5!(c)$) -- (center);
		\end{tikzpicture}
		\caption{The dual hypersurface of a triangulation of a triangle.}
	\end{subfigure}
	\hfill
	\begin{subfigure}[t]{0.45\textwidth}
		\centering
		\begin{tikzpicture}[scale=1.9]
			\draw (0,0,0) -- (2,0,0) -- (0,0,2) -- (0,0,0) -- (0,2,0) -- (2,0,0) ;
			\draw (0,2,0) -- (0,0,2);
			\fill (1,0,0) circle (.05);
			\fill (0,1,0) circle (.05);
			\fill (0,0,1) circle (.05);
			\fill (1,1,0) circle (.05);
			\fill (1,0,1) circle (.05);
			\fill (0,1,1) circle (.05);
			\fill ($1/3*(2,2,0)$) circle (.05);
			\fill ($1/3*(2,0,2)$) circle (.05);
			\fill ($1/3*(0,2,2)$) circle (.05);
			\fill ($1/3*(2,2,2)$) circle (.05);
			\fill ($1/4*(2,2,2)$) circle (.05);
			\filldraw[pattern={Lines[angle=60, yshift=4pt , line width=.5pt]},very thick] (1,0,0) -- ($1/3*(2,2,0)$) -- ($1/4*(2,2,2)$) -- ($1/3*(2,0,2)$) -- cycle;
			\filldraw[pattern={Lines[angle=60, yshift=4pt , line width=.5pt]},very thick] (0,1,0) -- ($1/3*(2,2,0)$) -- ($1/4*(2,2,2)$) -- ($1/3*(0,2,2)$) -- cycle;
			\filldraw[pattern={Lines[angle=60, yshift=4pt , line width=.5pt]},very thick] (0,0,1) -- ($1/3*(2,0,2)$) -- ($1/4*(2,2,2)$) -- ($1/3*(0,2,2)$) -- cycle;
			\filldraw[pattern={Lines[angle=60, yshift=4pt , line width=.5pt]},very thick] (0,1,1) -- ($1/3*(0,2,2)$) -- ($1/4*(2,2,2)$) -- ($1/3*(2,2,2)$) -- cycle;
			\filldraw[pattern={Lines[angle=60, yshift=4pt , line width=.5pt]},very thick] (1,0,1) -- ($1/3*(2,0,2)$) -- ($1/4*(2,2,2)$) -- ($1/3*(2,2,2)$) -- cycle;
			\filldraw[pattern={Lines[angle=60, yshift=4pt , line width=.5pt]},very thick] (1,1,0) -- ($1/3*(2,2,0)$) -- ($1/4*(2,2,2)$) -- ($1/3*(2,2,2)$) -- cycle;
			\coordinate (v) at (2,1.5,0);
			\begin{scope}[scale=.45, 3d view={160}{55}]
				\draw ($(v)+(0,0,0)$) -- ($(v)+(2,0,0)$) -- ($(v)+(0,0,2)$) -- ($(v)+(0,0,0)$) -- ($(v)+(0,2,0)$) -- ($(v)+(2,0,0)$) ;
				\draw ($(v)+(0,2,0)$) -- ($(v)+(0,0,2)$);
				\filldraw[pattern={Lines[angle=60, yshift=4pt , line width=.5pt]},very thick] ($(v)+(1,0,0)$) -- ($(v)+1/3*(2,2,0)$) -- ($(v)+1/4*(2,2,2)$) -- ($(v)+1/3*(2,0,2)$) -- cycle;
				\filldraw[pattern={Lines[angle=60, yshift=4pt , line width=.5pt]},very thick] ($(v)+(0,1,0)$) -- ($(v)+1/3*(2,2,0)$) -- ($(v)+1/4*(2,2,2)$) -- ($(v)+1/3*(0,2,2)$) -- cycle;
				\filldraw[pattern={Lines[angle=60, yshift=4pt , line width=.5pt]},very thick] ($(v)+(0,0,1)$) -- ($(v)+1/3*(2,0,2)$) -- ($(v)+1/4*(2,2,2)$) -- ($(v)+1/3*(0,2,2)$) -- cycle;
				\filldraw[pattern={Lines[angle=60, yshift=4pt , line width=.5pt]},very thick] ($(v)+(0,1,1)$) -- ($(v)+1/3*(0,2,2)$) -- ($(v)+1/4*(2,2,2)$) -- ($(v)+1/3*(2,2,2)$) -- cycle;
				\filldraw[pattern={Lines[angle=60, yshift=4pt , line width=.5pt]},very thick] ($(v)+(1,0,1)$) -- ($(v)+1/3*(2,0,2)$) -- ($(v)+1/4*(2,2,2)$) -- ($(v)+1/3*(2,2,2)$) -- cycle;
				\filldraw[pattern={Lines[angle=60, yshift=4pt , line width=.5pt]},very thick] ($(v)+(1,1,0)$) -- ($(v)+1/3*(2,2,0)$) -- ($(v)+1/4*(2,2,2)$) -- ($(v)+1/3*(2,2,2)$) -- cycle;
			\end{scope}
		\end{tikzpicture}
		\caption{Two views of the dual hypersurface of a tetrahedron.}
	\end{subfigure}
	\caption{Two examples of dual hypersurfaces.}
	\label{fig:dual_hypersurface}
\end{figure}
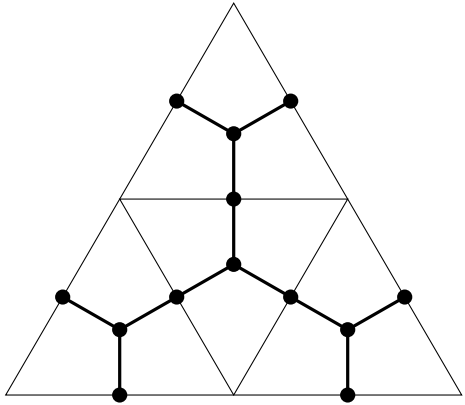
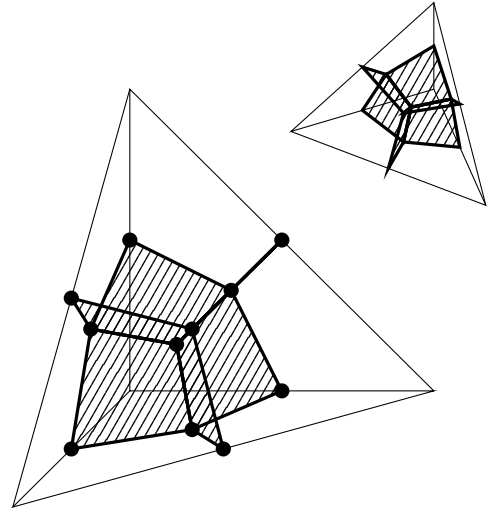

\begin{rem}
	Let $P$ be a smooth polytope and $Y$ denote the toric variety associated with $P$. The tropical locus $Y(\T)$ of $Y$ is a manifold with corner that is isomorphic to $P$. A tropical section $s$ of the line bundle associated with $P$ defines a tropical hypersurface $H$ of $Y(\T)$. When $H$ is smooth in the tropical sense, the section $s$ induces a primitive triangulation $K$ of $P$. This triangulation is necessarily convex, \textit{i.e.}~is made of the projections of the bottom faces of a polytope of $(M\otimes\R)\oplus \R$ living above $P$. In this case, the pair $(X;P)$, where $X$ denotes the dual hypersurface of $K$, is homeomorphic to the pair $(H;Y(\T))$. See \cite[Section~3.4]{Mik-Rau_tro_geo} for the appropriate definitions of tropical geometry.
\end{rem}

\begin{dfn}[\textit{cf.} {\cite[Section 1.1]{She_cel_des}}]
	Any CW-complex $E$ can be seen as a partially ordered set and thus as a small category. If we denote this category by $\mathbf{Cell}\;E$ and $R$ is a commutative ring, a \emph{cellular cosheaf of $R$-modules} $F$ on $E$ is a contravariant functor
	\begin{equation*}
		F\colon \mathbf{Cell}\;E^\op \longrightarrow \mathbf{Mod}_R.
	\end{equation*}
	 The morphism associated with an adjacent pair $e^p\leq e^q$, denoted by $-|^{e^q}_{e^p}\colon F(e^q)\rightarrow F(e^p)$, is called an \emph{extension} morphism of $F$. Dually, a \emph{cellular sheaf of $R$-modules} $G$ on $K$ is a covariant functor 
	\begin{equation*}
		G\colon \mathbf{Cell}\;E \longrightarrow \mathbf{Mod}_R.
	\end{equation*}
	 The morphism associated with an adjacent pair $e^p\leq e^q$, denoted by $-|^{e^q}_{e^p}\colon G(e^p)\rightarrow G(e^q)$, is called a \emph{restriction} morphism of $G$. If $\mathbf{C}$ is any category and one has two functors $A$, from $\mathbf{Cell }\; E$ to $\mathbf{C}$, and $B$, from $\mathbf{C}$ to $\mathbf{Mod}_R$, one gets either a cellular sheaf or a cellular cosheaf by composing $A$ with $B$. If $E$ is a $\Delta$-complex, a \emph{cubical sheaf} (resp.\ \emph{cubical cosheaf}\,) on $E$ is a cellular sheaf (resp.\ cosheaf) on its cubical subdivision.
\end{dfn}

\begin{dfn}[Morphisms of cellular sheaves and cosheaves, \textit{cf.} {\cite[Section 1.1]{She_cel_des}}]
	A \emph{morphism of cellular sheaves} (resp.\ \emph{cosheaves}) $f\colon F\rightarrow F'$ is a natural transformation. Such morphism is said to be \emph{injective} (resp.\ \emph{surjective},  \emph{invertible}) if the associated morphisms $f_e\colon F(e)\rightarrow F'(e)$ are injective (resp.\ surjective, invertible) for all cells $e$. The \emph{kernel}, \emph{image}, and \emph{cokernel} of such a morphism $f$ are the ``cell-wise'' kernel, image, and cokernel. They are themselves sheaves (resp.\ cosheaves) for which  the induced restriction (resp.\ extension) morphisms for $f$ is a natural transformation. 
\end{dfn}

We now give the definitions of the tropical sheaves and cosheaves associated with a primitive triangulation~$K$ of a smooth polytope $P$. All these objects usually take values in the category of Abelian groups; \textit{cf.} \cite[Definition 13]{Ite-Kat-Mik-Zha_tro_hom}. The definition in the given reference is the first occurrence of such objects. However, it would only apply to convex triangulations. Their definition was extended to all triangulations in \cite[Section 2.2]{Brug-LdM-Rau_Comb_pac}. Here, we will only consider their reduction modulo $2$. Therefore, we only set the notation for their reduced version.

\begin{dfn}\label{dfn:trop_cosheaves} Let $P$ be a smooth polytope of $M\otimes\R$ endowed with a primitive triangulation $K$. The first cellular cosheaf below is defined on $K$; all the other cellular sheaves and cosheaves are defined on the cubical subdivision of $K$.
	\begin{enumerate}
		\item The \emph{sedentarity} cosheaf $\Sed$ is defined as follows: for a simplex $\sigma^p$ of $K$, the group $\Sed(\sigma^p)$ is given by
		\begin{equation*}
			\Sed(\sigma^p):=\left\{v\in N\otimes\F_2\relmiddle| \alpha(v)=0,\, \forall \alpha\in(TQ\cap M)\right\},
		\end{equation*}
		 where $Q$ is the smallest face of $P$ containing the relative interior\footnote{The relative interior is the interior of $\sigma^p$ in the affine space it spans.} of $\sigma^p$. If $\sigma^p\leq\sigma^q$, the extension morphism ${\Sed(\sigma^q)\rightarrow\Sed(\sigma^p)}$ is the inclusion.
		\item The cubical subdivision of the sedentarity, which we denote by the same symbol $\Sed$, is defined as follows: for a cube of $K$ indexed by a pair of simplices $\sigma^p\leq\sigma^q$, the group $\Sed(\sigma^p;\sigma^q)$ is given by 
		\begin{equation*}
			\Sed\left(\sigma^p;\sigma^q\right):=\Sed(\sigma^q).
		\end{equation*}
		A cube indexed by $\sigma^p\leq\sigma^q$ is a face of the cube indexed by $\sigma^{p'}\leq\sigma^{q'}$ if and only if we have ${\sigma^{p'}\leq \sigma^p\leq\sigma^q \leq\sigma^{q'}}$. In this case the extension morphism $\Sed(\sigma^{p'};\sigma^{q'})\rightarrow \Sed(\sigma^p;\sigma^q)$ is also given by the inclusion.
		\item The cosheaf $F_1^P$ is defined as the quotient
		\begin{equation*}
			F_1^P:=\bigslant{N\otimes\F_2}{\Sed},
		\end{equation*}
		 where $N\otimes\F_2$ is understood as the constant cosheaf. 
		\item The cosheaves $F_p^P$, for all $p\in \N$, are the exterior powers of $F_1^P$:
		\begin{equation*}
			F_p^P:=\bigwedge^p F_1^P;
		\end{equation*}
		 in particular, $F_0^P$ is the constant cosheaf $\F_2$. 
		\item For all $p\in \N$, the sheaf $F^p_P$ is the dual of $F^P_p$; that is to say, 
		\begin{equation*}
			F^p_P:=\Hom_{\F_2}\left(F^P_p;\F_2\right).
		\end{equation*}
		\item\label{def3.9-6} For all $k\in \N$, the groups of the cosheaf $F_k^X$ are defined on cubes indexed by $\sigma^1\leq \sigma^q$ by the formula 
		\begin{equation*}
			F_k^X(\sigma^1;\sigma^q):=\bigwedge^k\left\{v\in\bigslant{N\otimes\F_2}{\Sed(\sigma^1;\sigma^q)} \relmiddle| \alpha(v)=0\right\},
		\end{equation*}
		 where $\alpha$ is a generator of $T\sigma^1\cap M$. For a more general cube $\sigma^p\leq\sigma^q$, the group is given by 
		\begin{equation*}
			F_k^X(\sigma^p;\sigma^q):=\sum_{\sigma^1\leq\sigma^p}F_k^X(\sigma^1;\sigma^q).
		\end{equation*}
		 See Figure~\ref{fig:F1} for an example. The morphisms between these groups are given by quotients and inclusions. Note that even though $F_k^X$ is defined on the whole cubical subdivision of $K$, its support\footnote{The support of $F_k^X$ is the union of the closed cubes carrying non-trivial groups.} is contained in $X$. It is a sub-cosheaf of $F^P_p$, and we denote the inclusion by $i_p:F^X_p\subset F^P_p$. 
		\item For all $p\in \N$, the sheaf $F^p_X$ is the dual of $F^X_p$; that is to say, 
		\begin{equation*}
			F^p_X:=\Hom_{\F_2}\left(F^X_p;\F_2\right).
		\end{equation*}
		 We denote the adjoint projection of $i_p$ by $i^p\colon F_P^p\rightarrow F_X^p$.
	\end{enumerate}
\end{dfn}

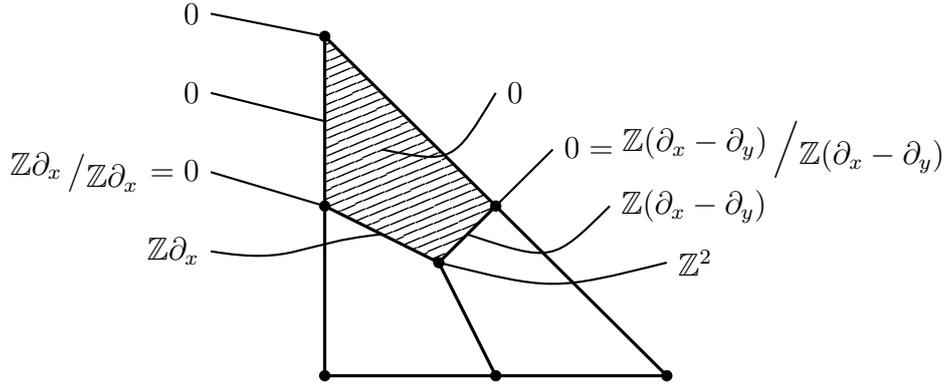
\begin{figure}[!ht]
	\centering
	\begin{tikzpicture}[scale=1.5]
		\foreach \p in {(0,0), (3,0), (0,3)}{
			\fill \p circle (.05);}
	
		\draw (0,0) -- (3,0) -- (0,3) -- cycle;
		\draw[very thick] (0,3) -- (0,1.5) -- (1,1) -- (1.5,1.5) -- cycle;
		\fill (1,1) circle (.05);
		\fill ($3/2*(1,1)$) circle (.05);
		\fill ($1/2*(3,0)$) circle (.05);
		\fill ($1/2*(0,3)$) circle (.05);
		\draw ($3/2*(1,1)$) -- (1,1) -- ($1/2*(3,0)$);
		\draw ($1/2*(0,3)$) -- (1,1);
		\fill[pattern={Lines[angle=22.5, yshift=4pt , line width=.5pt]}] (0,3) -- ($3/2*(0,1)$) -- (1,1) -- ($1/2*(3,3)$) -- cycle;
		\draw (-1,3.2) node[left]{$0$} -- (0,3);
		\draw (-1,2.5) node[left]{$0$} -- (0,2.25);
		\draw (-1,1.8) node[left]{$\displaystyle \bigslant{\F_2e_1}{\F_2e_1}=0$} -- (0,1.5);
		\draw (-1,1.1) node[left]{$\F_2e_1$} .. controls ($1/2*(-1,1.1)+1/4*(1,1)+1/4*(0,1.5)+(0,-.2)$) and ($1/2*(-1,1.1)+1/4*(1,1)+1/4*(0,1.5)$) .. ($1/2*(1,1)+1/2*(0,1.5)$) ;
		\draw ($1/2*(3,0)+10/4*(0,1)$) node[right]{$0$} .. controls ($(0,3)!.75!(1.5,1.5)$) and ($(0,3)!.75!(1.5,1.5)$) .. (0.5,2);
		\draw ($1/2*(3,0)+10/4*(0,1)+(.5,-.5)$) node[right]{$\displaystyle 0=\bigslant{\F_2(e_1-e_2)}{\F_2(e_1-e_2)}$} -- ($3/2*(1,1)$);
		\draw ($1/2*(3,0)+10/4*(0,1)+(1,-1)$) node[right]{$\F_2(e_1-e_2)$} .. controls (2,1) and (2,1) .. ($3/2*(1,1)+(-.25,-.25)$);
		\draw ($1/2*(3,0)+10/4*(0,1)+(1.5,-1.5)$) node[right]{$\F_2^2$} .. controls (2,.75) and (2,.75) .. (1,1);
	\end{tikzpicture}
	\caption{A triangle and the groups associated by $F_1^{P}$ with some of its cubical cells. We denote by $(e_1^*;e_2^*)$ the canonical basis of $M\cong\Z^2$ and by $(e_1;e_2)$ the dual basis of $N$.}
	\label{fig:F1}
\end{figure}

\begin{rem}
	Let $p\geq 0$ be an integer. In the notation of \cite{Brug-LdM-Rau_Comb_pac}, $F^P_p$ corresponds to $\mathcal{F}^0_p$, and $F^X_p$ corresponds to $\mathcal{F}^1_p$.
\end{rem}

\begin{dfn}
	Let $\sigma^p$ be a simplex of $K$. Its tangent space is a $p$-dimensional rational subspace of $M\otimes\R$. We denote by $\omega(\sigma^p)$ the generator of the line $\bigwedge^p(T\sigma^p\cap M)\otimes_\Z\F_2$.
\end{dfn}

We recall a basic construction of multilinear algebra in order to give a different definition of the cosheaves $(F^X_p)_{p\geq 0}$.

\begin{dfn}[Contraction]
	Let $V$ be a finite-dimensional vector space over a field $\F$ and $k,l\in\N$ be integers. For all $\alpha\in \bigwedge^l V^*$ and $v\in\bigwedge^{l+k}V$, the contraction $\alpha\cdot v$ is the only element of $\bigwedge^{k} V$ satisfying 
	\begin{equation*}
	\beta(\alpha\cdot v)=(\beta\wedge\alpha)(v) 
	\end{equation*}
	 for all $\beta\in \bigwedge^kV^*$. This construction is dual to the interior product.
\end{dfn}

\begin{prop}\label{prop:F_and_contraction}
	For all pairs of adjacent simplices $\sigma^p\leq\sigma^q$ of $K$, we have the exact sequence 
	\begin{equation*}
		\begin{tikzcd}[column sep =3em]
			0 \ar[r] & F_k^X(\sigma^p;\sigma^q) \ar[r] & F_k^P(\sigma^p;\sigma^q) \ar[r,"\omega(\sigma^p)\cdot-" above] & F_{k-p}^P(\sigma^p;\sigma^q).
		\end{tikzcd}
	\end{equation*}
\end{prop}

\begin{proof}
	Let us denote by $V$ the quotient of $N\otimes\F_2$ by $\Sed(\sigma^p;\sigma^q)$. For all simplices $\sigma^1\leq\sigma^p$, the linear form $\omega(\sigma^1)$ vanishes on $\Sed(\sigma^p;\sigma^q)$. We abuse the notation and also denote by $\omega(\sigma^1)$ the factorisation of $\omega(\sigma^1)$ by the quotient map $N\otimes\F_2\rightarrow V$. From point~\eqref{def3.9-6} of Definition~\ref{dfn:trop_cosheaves},  we have 
	\begin{equation*}
		F^X_k(\sigma^p;\sigma^q)=\sum_{\sigma^1\leq\sigma^q}\bigwedge^k\ker(\omega(\sigma^1))\subset\bigwedge^k V.
	\end{equation*}
	 Let us choose a vertex $\sigma^0$ of $\sigma^p$. Since $\sigma^p$ is primitive, the set $\{\omega(\sigma^1)\colon \sigma^0\leq\sigma^1\leq\sigma^p\}$ is, up to ordering, a basis of the vector space $\left(T\sigma^p\cap M\right)\otimes_\Z\F_2\subset V^*$. We order this set $\{\omega_1,\ldots,\omega_p\}$ and add  some vectors $\alpha_1,\ldots,\alpha_s\in V^*$ in order to form a basis of $V^*$. We denote by $\{e_1,\ldots,e_p,f_1,\ldots,f_s\}$ the dual basis. We see that $F^X_k(\sigma^p;\sigma^q)$ is spanned by the vectors\footnote{We use the standard notation $e_I:=e_{i_1}\wedge \cdots\wedge e_{i_k}$, where $I=\{i_1<\cdots<i_k\} $.} $e_I\wedge f_J$ for all $I\subset\{1,\ldots,p\}$ and $J\subset\{1,\ldots,s\}$ satisfying $|I|+|J|=k$ and $|I|<p$. We can observe that $\omega(\sigma^p)=\omega_1\wedge\cdots\wedge\omega_p$ and that 
	\begin{equation*}
		\omega(\sigma^p)\cdot e_I\wedge f_J=\left\{\begin{array}{cl} 
		f_J & \textnormal{if }I=\{1,\ldots,p\}, \\
		0 & \textnormal{otherwise}, \end{array}\right. 
	\end{equation*}
	 from which the proposition follows.
\end{proof}

\begin{dfn}\label{dfn:co_chain}
	Let $F$ be a cellular cosheaf of $\F_2$-vector spaces on a regular CW-complex $E$. We denote by $(C_k(E;F);\partial)_{k\geq 0}$ the \emph{chain complex of cellular chains with coefficients in $F$}. For all $k\geq 0$, its $\supth{k}$ group of chains is defined by the formula 
	\begin{equation*}
		C_k(E;F):=\bigoplus_{e^k\in E}F\left(e^k\right).
	\end{equation*}
	 We write an element of $C_k(E;F)$ in the following way: 
	\begin{equation*}
		c=\sum_{e^k\in E} c_{e_k}\otimes e^k.
	\end{equation*}
	 The \emph{boundary operator} $\partial\colon C_{k+1}(E;F)\rightarrow C_{k}(E;F)$ is defined by the formula 
	\begin{equation*}
		\partial c=\sum_{e^{k+1}\in E} \;\sum_{e^k\leq e^{k+1}} c_{e^{k+1}}\Big |^{e^{k+1}}_{e^k} \otimes e^k,
	\end{equation*}
	where $-|^{e^{k+1}}_{e^k}\colon F(e^{k+1})\rightarrow F(e^{k})$ denotes the extension morphism of $F$. In a regular CW-complex, if $e^k\leq e^{k+2}$, then there are exactly two distinct $e^{k+1}$ between $e^k$ and $e^{k+2}$. This ensures that $\partial\circ\partial=0$. A morphism of cosheaves $f\colon F\rightarrow F'$ induces a morphism between the associated chain complexes. This associated morphism of chain complexes is injective, surjective, or invertible if and only if $f\colon F\rightarrow F'$ is. If dually $G$ is a cellular sheaf of $\F_2$-vector spaces on $E$, we denote by $(C^k(E;G);\df)_{k\geq 0}$ the \emph{cochain complex of cellular cochains with coefficients in $G$}. For all $k\geq 0$, its $\supth{k}$ group of cochains is defined by the formula 
	\begin{equation*}
		C^k(E;G):=\prod_{e^k\in E}G\left(e^k\right).
	\end{equation*}
	 The \emph{coboundary operator}, or \emph{differential}, $\df\colon C^{k}(E;G)\rightarrow C^{k+1}(E;G)$ is defined, for all $\alpha\in C^k(E;G)$ and all cells $e^{k+1}$, by the formula 
	\begin{equation*}
		\df \alpha\left(e^{k+1}\right)=\sum_{e^k\leq e^{k+1}} \alpha\left(e^{k}\right)\Big |^{e^{k+1}}_{e^k}, 
	\end{equation*}
	 where $-|^{e^{k+1}}_{e^k}\colon G(e^{k})\rightarrow G(e^{k+1})$ denotes the restriction morphism of $G$. As in the case of chain complexes, we have $\df\circ\df=0$. A morphism of sheaves $f\colon F\rightarrow F'$ induces a morphism between the associated chain complexes. This associated chain complexes morphism is injective, surjective, or invertible if and only $f\colon F\rightarrow F'$ is. If $F^*$ is the cellular sheaf obtained as $F^*=\Hom_{\F_2}(F;\F_2)$ from a cellular cosheaf $F$, then the complex $(C^k(E;F^*);\df)_{k\geq 0}$ is the dual complex of $(C_k(E;F);\partial)_{k\geq 0}$. If $E$ is the cubical subdivision of a simplicial complex $K$ and $F$ is a cubical cosheaf (resp.\ sheaf), we write $(\Omega_k(K;F);\partial)_{k\geq 0}$ instead of $(C_k(E;F);\partial)_{k\geq 0}$ (resp.\ $(\Omega^k(K;F);\df)_{k\geq 0}$ instead of $(C^k(E;F);\df)_{k\geq 0}$).
\end{dfn}

\begin{dfn}
	The homology of a cellular cosheaf $F$ (resp.\ cohomology of a cellular sheaf) of a regular CW-complex $E$ is the homology (resp.\ cohomology) of the associated chain complex (resp.\ cochain complex) from Definition~\ref{dfn:co_chain}. It is denoted by $(H_k(E;F))_{k\geq 0}$ (resp.\ $(H^k(E;F))_{k\geq 0}$). 
\end{dfn}

\begin{rems}\label{rems:subdiv} Let $E_0$ be a regular $CW$-complex, $E_1$ be a subcomplex of $E_0$, and $E_0'$ be a subdivision of $E_0$.\begin{enumerate}
	\item Given a cellular sheaf (resp.\ cosheaf) $F_1$ on $E_1$, there is a unique sheaf (resp.\ cosheaf) $F_0$ on $E_0$ whose restriction to $E_1$ is $F_1$, and whose groups vanish on cells of $E_0$ not contained in $E_1$. Moreover, there is a natural isomorphism of cochain complexes between $(C^k(E_1;F_1);\df)_{k\geq 0}$ and $(C^k(E_0;F_0);\df)_{k\geq 0}$. We also have a natural isomorphism of the appropriate chain complexes in the case of a cosheaf. In addition, if $G$ is a cellular sheaf (resp.\ cosheaf) on $E_0$ whose groups vanish on cells of $E_0$ not contained in $E_1$, then $G$ is obtained by extending its restriction to $E_1$. 
	\item There is a subdivision functor $F\mapsto F'$ that makes a cellular sheaf (resp.\ cosheaf) on $E_0'$ out of a cellular sheaf (resp.\ cosheaf) on $E_0$. If $e'$ is a cell of $E_0'$ contained in the cell $e$ of $E_0$, then $F'(e')=F(e)$. For a cellular sheaf $F$, there is a natural quasi-isomorphic projection $(C^k(E_0';F');\df)_{k\geq 0}\rightarrow (C^k(E_0;F);\df)_{k\geq 0}$. For a cellular cosheaf $F$, there is a natural quasi-isomorphic injection $(C_k(E_0;F);\partial)_{k\geq 0}\rightarrow (C_k(E_0';F');\partial)_{k\geq 0}$; \textit{cf.} \cite[Section 1.5]{She_cel_des}.
\end{enumerate}
\end{rems}

The sheaves and cosheaves of Definition~\ref{dfn:trop_cosheaves} are sometimes defined on other cellular structures of $P$ and $|X|$, especially in the context of tropical geometry. However, following  Remarks~\ref{rems:subdiv} we may write, for all $p,q\geq 0$, 
\begin{itemize}
	\item $H_{p,q}(P;\F_2)$ instead of $H_{q}(K;F^P_p)$, respectively $H^{p,q}(P;\F_2)$ instead of $H^{q}(K;F_P^p)$; 
	\item $H_{p,q}(X;\F_2)$ instead of $H_{q}(K;F^X_p)$, respectively $H^{p,q}(X;\F_2)$ instead of $H^{q}(K;F_X^p)$.
	\end{itemize}
	In this article, every computation will be performed on the cubical subdivision of $K$. Using the universal coefficient theorem, \cite[Theorem 3.3a]{Car-Eil_hom_alg}, we see that the vector spaces $H_{p,q}(P;\F_2)$ and $H^{p,q}(P;\F_2)$ are dual to each other, and so are $H_{p,q}(X;\F_2)$ and $H^{p,q}(X;\F_2)$. We introduce the notation, for all $p,q\geq 0$, ${i_{p,q}\colon H_{p,q}(X;\F_2)\rightarrow H_{p,q}(P;\F_2)}$ and ${i^{p,q}\colon H^{p,q}(P;\F_2)\rightarrow H^{p,q}(X;\F_2)}$ for the morphisms induced in homology and cohomology by the inclusions $i_p:F^X_p\subset F^P_p$ and there adjoint projections. By construction and the universal coefficient theorem, the morphism $i^{p,q}$ is the adjoint of $i_{p,q}$.

\begin{dfn}[Cup product]
	Let $F$ be a cubical sheaf of algebras over $\F_2$ on a simplicial complex $K$. The cochain complex $(\Omega^k(K;F);\df)_{k\geq 0}$ is canonically endowed with the structure of a graded algebra:
	\begin{equation*}
		\begin{array}{rcl} \cup\colon \Omega^k(K;F)\otimes_{\F_2} \Omega^l(K;F) & \longrightarrow & \Omega^{k+l}(K;F) \\
		\alpha \otimes \beta & \longmapsto & \displaystyle \left[ (\sigma^p; \sigma^{p+k+l}) \mapsto \sum_{\sigma^p\leq \sigma^{p+k} \leq \sigma^{p+k+l}} \alpha(\sigma^p; \sigma^{p+k})\beta(\sigma^{p+k}; \sigma^{p+k+l})\right],
		\end{array}
	\end{equation*}
	 where the product $\alpha(\sigma^p; \sigma^{p+k})\beta(\sigma^{p+k} ;\sigma^{p+k+l})$ is understood as the product of $\alpha(\sigma^p; \sigma^{p+k})\big|_{(\sigma^p; \sigma^{p+k})}^{(\sigma^p; \sigma^{p+k+l})}$ and $\beta(\sigma^{p+k};\sigma^{p+k+l})\big|_{(\sigma^{p+k}; \sigma^{p+k+l})}^{(\sigma^p; \sigma^{p+k+l})}$ in $F(\sigma^p; \sigma^{p+k+l})$. It satisfies the Leibniz rule 
	\begin{equation*}
		\df(\alpha\cup\beta)=(\df\alpha)\cup\beta+\alpha\cup\df\beta 
	\end{equation*}
	 and therefore induces a well-defined product in cohomology. In particular, $(\Omega^k(K;F);\df;\cup)_{k\geq 0}$ is a graded differential algebra over $\F_2$.
\end{dfn}

Both the sheaves $\bigoplus_{p\in\N} F_P^p$ and $\bigoplus_{p\in\N} F_X^p$ are cubical sheaves of graded algebras for the wedge product. Moreover, the projection $\bigoplus_{p\in\N} F_P^p\rightarrow \bigoplus_{p\in\N} F_X^p$ is a morphism of sheaves of graded algebras. Therefore, the collections of vector spaces ${(H^{p,q}(P;\F_2))_{p,q\in \N}}$ and ${(H^{p,q}(X;\F_2))_{p,q\in\N}}$ both have the structure of bigraded algebras over $\F_2$, and the restriction ${\bigoplus_{p,q\in\N}i^{p,q}}$ is a morphism of bigraded algebras.

\section{Construction of T-hypersurfaces}

\begin{ntns}
	Let $M$ be a free Abelian group of rank $n$. Let $P$ be a smooth, full-dimensional, integral polytope of the $n$-dimensional vector space $M\otimes\R$  and $K$ be a primitive triangulation of $P$ with dual hypersurface $X$.
\end{ntns}

\begin{dfn}\label{dfn:RP_CW}
	We denote by $\R P$ the topological space obtained as the quotient of $(N\otimes\F_2)\times P$ by the equivalence relation %
	\begin{equation*}
		(v;x)\sim(v';x') \quad\text{if and only if}\quad x=x' \textnormal{ and } v-v'\in \Sed(x),
	\end{equation*}%
	 where $\Sed(x)$ denotes the space $\{v\in N\mid\alpha(v)=0,\,\forall \alpha \in TQ\}\otimes\F_2$, $Q$ being the only face of $P$ containing $x$ in its relative interior.
This space is obtained by gluing together $2^n$ copies of $P$ and is naturally a regular CW-complex. 
\end{dfn}

Let $Y$ denote the toric variety associated with $P$. The moment map $\mu\colon Y(\R)\rightarrow P$ induces a homeomorphism between $Y(\R_+)$ and $P$; \textit{cf.} \cite[Chapter 4]{Ful_tor_var}. Therefore, the moment map has a natural section $s$ that takes its values in $Y(\R_+)$. Moreover, $(N\otimes\F_2)$ acts on $Y(\R)$ as the $2$-torsion of the real locus of the torus of~$Y$. The map
\begin{equation*}\label{eq:pre_homeo}
	\begin{array}{rcl}
		G\colon (N\otimes\F_2)\times P & \longrightarrow & Y(\R) \\
		(v;p) & \longmapsto & v\cdot s(p) 
	\end{array}
\end{equation*}
induces an homeomorphism between $\R P$ and $Y(\R)$. Hence, $\R P$ can be thought of as a cellular structure on $Y(\R)$; see \cite[Theorem 5.4]{Gel-Kap-Zel_dis_res}. Using the homeomorphism induced by $G$, the projection of $\R P$ onto $P$, denoted by $|\cdot|\colon \R P \rightarrow P$, corresponds to the moment map.

\begin{dfn}
	We denote by $\R K$ the lift of $K$ to a subdivision of $\R P$.
\end{dfn}

The complex $\R K$ might not be a triangulation of $\R P$. However, it is always a $\Delta$-complex. To distinguish the cells of $K$ from those of $\R K$, we will denote the former by the symbols $\sigma^p$ and the latter by the symbols $\sigma^p_\R$. With this definition, the projection $|\cdot|\colon \R P \rightarrow P$ is automatically a cellular map. For a simplex $\sigma^p\in K$, the set of $p$-simplices of $\R K$ above $\sigma^p$ is canonically in bijection with the vector space%
	\begin{equation*}
		\bigslant{N\otimes\F_2}{\Sed(\sigma^p)}.
	\end{equation*}%
	\begin{dfn}
		We define the \emph{argument} of a simplex $\sigma_\R^p\in\R K$, denoted by $\arg(\sigma_\R^p)$, to be the element of $(N\otimes\F_2)/\Sed(\sigma^p)$ to which $\sigma_\R^p$ corresponds.
	\end{dfn} A cubical cell of $\R K$ is indexed by a pair of simplices $\sigma_\R^p\leq \sigma_\R^q$ and is projected by $|\cdot|$ on the cubical cell of $K$ indexed by $|\sigma_\R^p|\leq |\sigma_\R^q|$. Reciprocally, if $\sigma^p\leq\sigma^q$ is a pair of simplices of $K$, its set of lifting cells is in bijection with the vector space%
	\begin{equation*}
		\bigslant{N\otimes\F_2}{\Sed(\sigma^p;\sigma^q)}.
	\end{equation*}%
	\begin{dfn}
		If the cubical cell indexed by $\sigma_\R^p\leq \sigma_\R^q$ lifts the cell $\sigma^p\leq \sigma^q$, we define $\arg(\sigma_\R^p; \sigma_\R^q)$ to be the element of $(N\otimes\F_2)/\Sed(\sigma^p;\sigma^q)$ to which it corresponds.
	\end{dfn}
	
	One can show that $\arg(\sigma_\R^p; \sigma_\R^q)$ equals $\arg(\sigma_\R^q)$.

\begin{dfn}(Canonical cocycle)\label{canonical cocycle}
	The cellular complex $\R K$ possesses a canonical cocycle of degree 1 denoted by ${\omega_{\R X}\in Z^1(\R K;\F_2)}$. Its value on an edge $\sigma_\R^1$ is given by the following formula:%
	\begin{equation*}
		\omega_{\R X}\left(\sigma_\R^1\right)\coloneqq \omega\left(\left|\sigma^1_\R\right|\right)\left(\arg\left(\sigma_\R^1\right)\right).
	\end{equation*}%
\end{dfn}

	The projection map $|\cdot|\colon \R P \rightarrow P$ being cellular, it induces both a morphism of chain complexes and a morphism of cochain complexes respectively denoted by 
	\begin{equation*}
		|\cdot|\colon \left(C_p(\R K;\F_2)\longrightarrow C_p(K;\F_2)\right)_{p\geq0} \quad\text{and}\quad |\cdot|^*\colon\left(C^p(K;\F_2)\longrightarrow C^p(\R K;\F_2)\right)_{p\geq 0},
	\end{equation*}
       	for all primitive triangulations $K$ of $P$.

\begin{dfn}[T-hypersurface]\label{dfn:T-hypersurface}
	Let $\varepsilon\in C^0(K;\F_2)$. We may call such a cochain a \emph{sign distribution}\index{Sign Distribution} on $K$. We denote by $\X$ the subcomplex of the cubical subdivision of $\R K$ that is Poincar\'e dual to the cocycle $\df|\varepsilon|^*+\omega_{\R X}$. That is to say, $\X$ is the union of the closed cubes of $\R K$ indexed by the pairs of simplices $\sigma_\R^1\leq \sigma_\R^n$ for which $\df\varepsilon(|\sigma_\R^1|)+\omega_{\R X}(\sigma_\R^1)=1$. We say that $\X$ is the \emph{T-hypersurface} constructed from the sign distribution $\varepsilon$. See Figure~\ref{fig:T-curve_torus} for an example. We  say that the cohomology class $[\omega_{\R X}]\in H^1(\R P;\F_2)$ is the \emph{degree}\index{T-Hypersurface!Degree} of $\X$.
\end{dfn}

The set $\X$ is a submanifold of $\R P$ of codimension 1; \textit{cf.} \cite[Proposition 4.11]{Brug-LdM-Rau_Comb_pac}. Following the terminology of O.~Viro, we  call the hypersurface $\X$ a primitive \emph{patchwork}\index{Patchwork} when the primitive triangulation~$K$ is convex.

\begin{rem}\label{rem:poinc_dual_X}
	The image of the fundamental class of $\X$ in the homology of $\R P$ is Poincar\'e dual to the cohomology class of $\df|\varepsilon|^*+\omega_{\R X}$. See \cite[Theorem 67.1]{Munk_ele_alg} for instance. Since $\df|\varepsilon|^*$ is exact, the class of $\X$ in $\R P$ is Poincar\'e dual to the degree $[\omega_{\R X}]$ of $\X$. We can also note that $\omega_{\R X}$ is the cocycle that defines the ``positive'' T-hypersurface, \textit{i.e.}~$\R X_0$.
\end{rem}

\begin{ex}
	In Figure~\ref{fig:T-curve_torus}, we give an example of a T-curve in $\mathbb{P}^1(\R)\times\mathbb{P}^1(\R)$ from a triangulation of the square of size 3. It is made of two connected components: one contractible, and one non-contractible whose homology class is $(1;1)$ in the canonical basis ($[\mathbb{P}^1(\R)\times\{0\}]$, $[\{0\}\times\mathbb{P}^1(\R)]$). 
\end{ex}

\begin{figure}[!ht]
	\centering
	\begin{subfigure}[t]{0.45\textwidth}
		\centering
		\begin{tikzpicture}[scale=2]
			\draw[white] (.5,0) -- +(0,-.375);
			\draw[very thick] (0,0) rectangle (3,3);
			\begin{scope}
				\draw (0,0) -- (2,2) -- (0,3);
				\draw (2,2) -- (1,3);
				\draw (2,2) -- (2,3);
				\draw (2,2) -- (3,3);
				\draw (0,3) -- (1,2) -- (2,2);
				\draw (1,2) -- (0,2);
				\draw (1,2) -- (0,1);
				\draw (1,2) -- (1,1) -- (0,1);
				\draw (1,0) -- (3,2);
				\draw (2,0) -- (3,1);
				\draw (2,2) -- (2,0);
				\draw (1,1) -- (1,0);
				\draw (1,0) -- (2,2);
				\draw (2,1) -- (3,3);
				\draw (2,1) -- (3,2);
				\draw (2,1) -- (3,1);
			\end{scope}
			
			\draw (0,0) node[anchor= north east] {$\mathbf{0}$};
			\draw (1,0) node[anchor= north] {$\mathbf{1}$};
			\draw (2,0) node[anchor= north] {$\mathbf{1}$};
			\draw (3,0) node[anchor= north west] {$\mathbf{1}$};
			\draw (0,1) node[anchor= east] {$\mathbf{0}$};
			\draw (1,1) node[anchor= west] {$\mathbf{0}$};
			\draw (2,1) node[anchor= north west] {$\mathbf{1}$};
			\draw (3,1) node[anchor= west] {$\mathbf{1}$};
			\draw (0,2) node[anchor= east] {$\mathbf{1}$};
			\draw (1,2) node[anchor= south west] {$\mathbf{0}$};
			\draw (2,2) node[anchor= west] {$\mathbf{0}$};
			\draw (3,2) node[anchor= west] {$\mathbf{1}$};
			\draw (0,3) node[anchor= south east] {$\mathbf{0}$};
			\draw (1,3) node[anchor= south] {$\mathbf{0}$};
			\draw (2,3) node[anchor= south] {$\mathbf{1}$};
			\draw (3,3) node[anchor= south west] {$\mathbf{1}$};	
		\end{tikzpicture}
		\caption{The triangulation $K$ and the sign distribution $\varepsilon$.}
	\end{subfigure}
	\hfill
	\begin{subfigure}[t]{0.45\textwidth}
		\centering
		\begin{tikzpicture}
			\draw[very thick] (-3,-3) rectangle (3,3);
			\draw[very thick, dashed] (-3,0) -- (3,0);
			\draw[very thick, dashed] (0,-3) -- (0,3);
			
			\begin{scope}[very thin]
				\draw (0,0) -- (2,2) -- (0,3);
				\draw (2,2) -- (1,3);
				\draw (2,2) -- (2,3);
				\draw (2,2) -- (3,3);
				\draw (0,3) -- (1,2) -- (2,2);
				\draw (1,2) -- (0,2);
				\draw (1,2) -- (0,1);
				\draw (1,2) -- (1,1) -- (0,1);
				\draw (1,0) -- (3,2);
				\draw (2,0) -- (3,1);
				\draw (2,2) -- (2,0);
				\draw (1,1) -- (1,0);
				\draw (1,0) -- (2,2);
				\draw (2,1) -- (3,3);
				\draw (2,1) -- (3,2);
				\draw (2,1) -- (3,1);
			\end{scope}
			
			\begin{scope}[ultra thick]
				\draw (.5,0) -- (1,.5) -- ($(1,0)!.5!(2,2)$) -- ($(2,1)!.5!(2,2)$) -- ($(3,3)!.5!(2,2)$) -- ($(2,3)!.5!(2,2)$) -- ($(2,3)!.5!(1,3)$);
				\draw (0,2.5) -- (.5,2) -- (0,1.5);
			\end{scope}
			
			\begin{scope}[very thin]
				\draw (0,0) -- (-2,2) -- (0,3);
				\draw (-2,2) -- (-1,3);
				\draw (-2,2) -- (-2,3);
				\draw (-2,2) -- (-3,3);
				\draw (0,3) -- (-1,2) -- (-2,2);
				\draw (-1,2) -- (0,2);
				\draw (-1,2) -- (0,1);
				\draw (-1,2) -- (-1,1) -- (0,1);
				\draw (-1,0) -- (-3,2);
				\draw (-2,0) -- (-3,1);
				\draw (-2,2) -- (-2,0);
				\draw (-1,1) -- (-1,0);
				\draw (-1,0) -- (-2,2);
				\draw (-2,1) -- (-3,3);
				\draw (-2,1) -- (-3,2);
				\draw (-2,1) -- (-3,1);
			\end{scope}
			
			\begin{scope}[ultra thick]
				\draw (0,2.5) -- ($(-1,2)!.5!(0,3)$) -- ($(-1,2)!.5!(-2,2)$) -- ($(-2,2)!.5!(-1,1)$) -- ($(-1,1)!.5!(-1,0)$) -- ($(-1,1)!.5!(0,0)$) -- ($(0,1)!.5!(-1,2)$) -- (0,1.5) ;
				\draw (-1.5,0) -- ($(-2,1)!.5!(-1,0)$) -- ($(-2,2)!.5!(-2,1)$) -- ($(-2,1)!.5!(-3,3)$) -- ($(-2,1)!.5!(-3,2)$) -- ($(-2,1)!.5!(-3,1)$) -- ($(-2,0)!.5!(-3,1)$) -- ($(-2,0)!.5!(-3,0)$);
				
				\draw (-2.5,3) -- (-2,2.5) -- ($(-2,2)!.5!(-1,3)$) -- (-.5,3);
			\end{scope}
			
			\begin{scope}[very thin]
				\draw (0,0) -- (-2,-2) -- (0,-3);
				\draw (-2,-2) -- (-1,-3);
				\draw (-2,-2) -- (-2,-3);
				\draw (-2,-2) -- (-3,-3);
				\draw (0,-3) -- (-1,-2) -- (-2,-2);
				\draw (-1,-2) -- (0,-2);
				\draw (-1,-2) -- (0,-1);
				\draw (-1,-2) -- (-1,-1) -- (0,-1);
				\draw (-1,0) -- (-3,-2);
				\draw (-2,0) -- (-3,-1);
				\draw (-2,-2) -- (-2,0);
				\draw (-1,-1) -- (-1,0);
				\draw (-1,0) -- (-2,-2);
				\draw (-2,-1) -- (-3,-3);
				\draw (-2,-1) -- (-3,-2);
				\draw (-2,-1) -- (-3,-1);
			\end{scope}
			
			\begin{scope}[very thin]
				\draw (0,0) -- (2,-2) -- (0,-3);
				\draw (2,-2) -- (1,-3);
				\draw (2,-2) -- (2,-3);
				\draw (2,-2) -- (3,-3);
				\draw (0,-3) -- (1,-2) -- (2,-2);
				\draw (1,-2) -- (0,-2);
				\draw (1,-2) -- (0,-1);
				\draw (1,-2) -- (1,-1) -- (0,-1);
				\draw (1,0) -- (3,-2);
				\draw (2,0) -- (3,-1);
				\draw (2,-2) -- (2,0);
				\draw (1,-1) -- (1,0);
				\draw (1,0) -- (2,-2);
				\draw (2,-1) -- (3,-3);
				\draw (2,-1) -- (3,-2);
				\draw (2,-1) -- (3,-1);
			\end{scope}
			
			\begin{scope}[ultra thick]
				\draw (-3,-.5) -- (-2.5,0);
				\draw (-3,-1.5) -- (-1.5,0);
				\draw (-3,-2.5) -- ($(-3,-3)!.5!(-2,-1)$) -- ($(-3,-3)!.5!(-2,-2)$) -- ($(-3,-3)!.5!(-2,-3)$);
				\draw (-.5,-3) -- ($(0,-3)!.5!(-2,-2)$) -- ($(-1,-2)!.5!(-2,-2)$) -- ($(0,0)!.5!(0,-1)$) -- ($(0,0)!.5!(1,-1)$) -- ($(0,0)!.5!(1,0)$);
				\draw (1.5,-3) -- ($(1,-3)!.5!(2,-2)$) -- ($(0,-3)!.5!(2,-2)$) -- ($(0,-3)!.5!(1,-2)$) -- ($(0,-2)!.5!(1,-2)$) -- ($(0,-1)!.5!(1,-2)$) -- ($(1,-1)!.5!(1,-2)$) -- ($(1,-1)!.5!(2,-2)$) -- ($(1,0)!.5!(2,-2)$) -- ($(1,0)!.5!(2,-1)$) -- ($(2,0)!.5!(2,-1)$) -- ($(2,0)!.5!(3,-1)$) -- (3,-.5);
				\draw (3,-1.5) -- (2.5,-1.5) -- (3,-2.5);
			\end{scope}
			
			\begin{scope}[ultra thick]
				\draw (-2.5,3) -- +(0,0.25) node[anchor=south] {$\mathbf{a}$};
				\draw (-.5,3) -- +(0,0.25) node[anchor=south] {$\mathbf{b}$};
				\draw (1.5,3) -- +(0,0.25) node[anchor=south] {$\mathbf{c}$};
				\draw (-2.5,-3) -- +(0,-0.25) node[anchor=north] {$\mathbf{a}$};
				\draw (-.5,-3) -- +(0,-0.25) node[anchor=north] {$\mathbf{b}$};
				\draw (1.5,-3) -- +(0,-0.25) node[anchor=north] {$\mathbf{c}$};
				\draw (-3,-.5) -- +(-0.25,0) node[anchor=east] {$\mathbf{d}$};
				\draw (3,-.5) -- +(0.25,0) node[anchor=west] {$\mathbf{d}$};
				\draw (-3,-1.5) -- +(-0.25,0) node[anchor=east] {$\mathbf{e}$};
				\draw (3,-1.5) -- +(0.25,0) node[anchor=west] {$\mathbf{e}$};
				\draw (-3,-2.5) -- +(-0.25,0) node[anchor=east] {$\mathbf{f}$};
				\draw (3,-2.5) -- +(0.25,0) node[anchor=west] {$\mathbf{f}$};
			\end{scope}
		\end{tikzpicture}
		\caption{The associated T-curve $\X$.}
	\end{subfigure}
	\caption{An example of T-curve in $\mathbb{P}^1(\R)\times \mathbb{P}^1(\R)$.}
	\label{fig:T-curve_torus}
\end{figure}
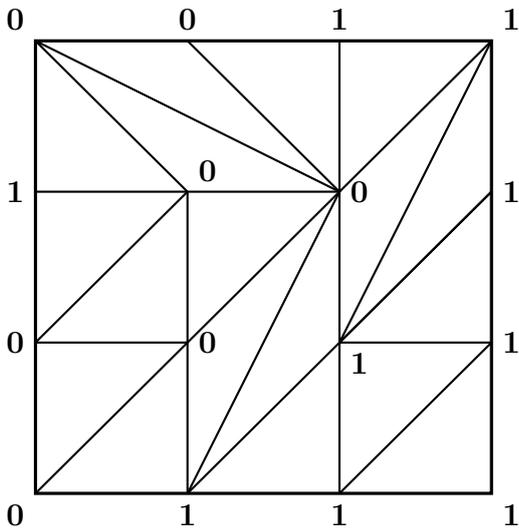
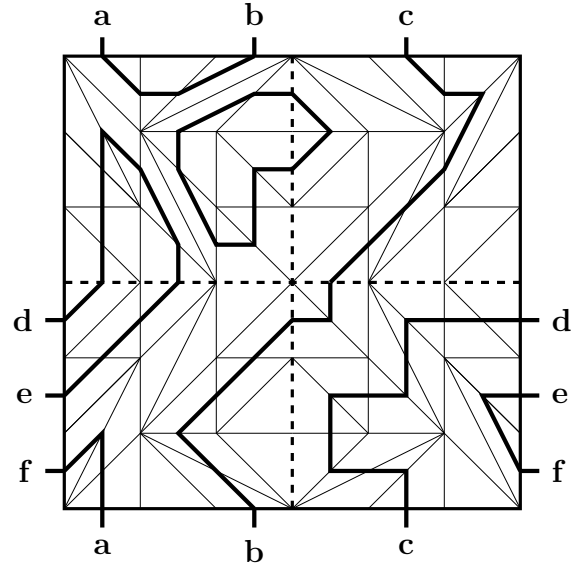

The image of $\X$ under the projection $|\cdot|\colon \R P\rightarrow P$ does not depend on $\varepsilon$. It is always the dual hypersurface $X$. Consider a cube of $K$ indexed by $\sigma^p\leq\sigma^q$. We denote by $\Arg_\varepsilon(\sigma^p;\sigma^q)$ the set of arguments $\arg(\sigma^p_\R;\sigma^q_\R)\in \smallslant{N\otimes\F_2}{\Sed(\sigma^q;\sigma^q)}$ of the lifts $\sigma^p_\R\leq \sigma^q_\R$ of $\sigma^p\leq\sigma^q$ belonging to $\X$. This is a cubical cosheaf of sets.

\begin{lem}\label{lem:desc_arg}
	For all pairs $\sigma^p\leq\sigma^q$ of simplices of $K$, the set $\Arg_\varepsilon(\sigma^p;\sigma^q)$ is the complementary subset in $V=(N\otimes\F_2)/\Sed(\sigma^q;\sigma^q)$ of an affine subspace parallel to $\{v\in V\mid \alpha(v)=0,\,\forall \alpha\in T\sigma^p(\F_2)\}$, so in particular of codimension $p$.
\end{lem}

\begin{proof}
	The set $\Arg_\varepsilon(\sigma^p;\sigma^q)$ is the set of the vectors $v\in V$ for which there exists a $\sigma^1\subset\sigma^p$ such that%
	\begin{equation*}
		\df\varepsilon(\sigma^1)+\omega(\sigma^1)(v)=1.
	\end{equation*}%
	 We can see a vector $v\in V$ as a closed $1$-cocycle on $\sigma^p$ using the following morphism:%
	\begin{equation*}
		\begin{array}{rcl}
			z\colon V & \longrightarrow & Z^1(\sigma^p;\F_2) \\
			v & \longmapsto & \left[ \sigma^1 \mapsto \omega(\sigma^1)(v) \right].
		\end{array}
	\end{equation*}%
	 From that point of view, $\Arg_\varepsilon(\sigma^p;\sigma^q)$ is the set $\{v\in V\mid z(v)\neq \df\varepsilon|_{\sigma^p}\}$, and it is the complementary subset in~$V$ of the fibre $z^{-1}(\df\varepsilon|_{\sigma^p})$. This fibre is an affine space parallel to the kernel of $z$. Since $K$ is primitive, $z$ is onto, and its kernel has codimension $p$. Furthermore, the family of linear forms $\{ v\mapsto z(v)(\sigma^1)\colon \sigma^1\leq\sigma^p\}$ spans $T\sigma^p(\F_2)$, so the kernel of $z$ is precisely $\{v\in V\mid \alpha(v)=0,\,\forall \alpha\in T\sigma^p(\F_2)\}$.
\end{proof}

\begin{rem}
	Let $\sigma^1$ be an edge of $K$. The sets $\Arg_\varepsilon(\sigma^1;\sigma^n)$ are all identical for all $n$-simplices $\sigma^n\geq\sigma^1$. A.~Renaudineau and K.~Shaw denoted a member of this collection by $\mathcal{E}$ and called it the real phase structure. 
\end{rem}

\begin{dfn}[Sign sheaves and cosheaves]\label{dfn:real_sheaves}
	We set the notation of two cubical cosheaves and two cubical sheaves on $K$:
	\begin{enumerate}
		\item $\K^{\R P}$ is the cosheaf of group algebras $(\sigma^p;\sigma^q)\mapsto \F_2[F^P_1(\sigma^p;\sigma^q)]$ (we will denote by $\x^v$ the generator associated with $v\in F^P_1(\sigma^p;\sigma^q)$). Its dual, the sheaf $\O_{\R P}$,  associates the $\F_2$-algebra of functions from $F^P_1(\sigma^p;\sigma^q)$ to $\F_2$ with the pair $(\sigma^p;\sigma^q)$.
		\item Likewise, we set $\K^\X$ to be the sub-cubical cosheaf of $\K^{\R P}$ that associates with $(\sigma^p;\sigma^q)$ the subspace of $\F_2[F^P_1(\sigma^p;\sigma^q)]$ freely generated by the elements of $\Arg_\varepsilon(\sigma^p;\sigma^q)$. Its dual, $\O_\X$, is the quotient of $\O_{\R P}$ that associates with $(\sigma^p;\sigma^q)$ the algebra of functions from $\Arg_\varepsilon(\sigma^p;\sigma^q)$ to $\F_2$.
		\item We denote by $i_*\colon\K^{\X}\rightarrow \K^{\R P}$ the inclusion, and by $i^*\colon\O_{\R P}\rightarrow \O_\X$ the adjoint projection. The latter is a morphism of sheaves of algebras.
	\end{enumerate}
\end{dfn}

\begin{rem}
	The cosheaf $\K^{\X}$ is the cubical subdivision of the sign cosheaves of \cite[Definition 3.12]{Ren-Sha_bou_bet} and \cite[Definition 4.12]{Brug-LdM-Rau_Comb_pac}. We did not use the same letter $\mathcal{S}$ as we will be mainly be interested in the Kalinin filtration of the cosheaf, which will lead to the introduction of the letter $\K$. We preferred using the same letter from the start. 
\end{rem}

With this definition, we can recall \cite[Proposition~3.17]{Ren-Sha_bou_bet} and its translation for the dual sheaves.

\begin{prop}
	There are chain-complex isomorphisms $\Phi_{\R P}$ and $\Phi_{\X}$ that make the following diagram commute:%
	\begin{equation*}
		\begin{tikzcd}
			\left(\Omega_q(K;\K^{\R P})\right)_{q\geq 0} \arrow[r,"\cong" below, "\Phi_{\R P}" above] & \left(\Omega_q(\R K;\F_2)\right)_{q\geq 0}\\%
			\left(\Omega_q(K;\K^\X)\right)_{q\geq 0} \arrow[r,"\cong" above, "\Phi_\X" below] \arrow[u,"i_*"] & \left(C_q(\X;\F_2)\right)_{q\geq 0}\rlap{.} \arrow[u] %
		\end{tikzcd}
	\end{equation*}%
	 As a consequence, the dual isomorphisms make the following dual diagram commute:%
	\begin{equation*}
		\begin{tikzcd}
			\left(\Omega^q(K;\O_{\R P})\right)_{q\geq 0} \arrow[d,"i^*"] & \arrow[l,"\cong" below, "\Phi^*_{\R P}" above] \left(\Omega^q(\R K;\F_2)\right)_{q\geq 0} \arrow[d] \\%
			\left(\Omega_q(K;\O_\X)\right)_{q\geq 0} & \arrow[l,"\cong" above, "\Phi^*_\X" below] \left(C^q(\X;\F_2)\right)_{q\geq 0}\rlap{.} %
		\end{tikzcd}
	\end{equation*}%
\end{prop}

\begin{proof}
	The first part of the statement is given by \cite[Proposition 3.17]{Ren-Sha_bou_bet}. The second part is a consequence of the duality. We can however remind ourselves of the construction of $\Phi_{\R P}$ ($\Phi_{\X}$ is its restriction). For a cube $\sigma^p\leq\sigma^q$ of $K$ and $v\in(N\otimes\F_2)/\Sed(\sigma^q)$, let us write $(v;\sigma^p;\sigma^q)$ for the corresponding lift in $\R K$. The isomorphism $\Phi_{\R P}$ sends the chain $\x^v\otimes(\sigma^p;\sigma^q)$ to the chain $(v;\sigma^p;\sigma^q)$.
\end{proof}

\section{Renaudineau--Shaw spectral sequences}

We recall the construction of the filtration given by A.~Renaudineau and K.~Shaw in \cite{Ren-Sha_bou_bet} and describe the dual filtration in cohomology. 

\subsubsection*{The augmentation filtration in group algebras} We denote by, respectively, $\vect$, $\hopf$, and $\gralg$ the categories of finite-dimensional vector spaces over $\F_2$, of finite-dimensional Hopf algebras over $\F_2$, and of finite-dimensional graded algebras over $\F_2$.

\begin{dfn}\label{dfn:functor_group_alg}
	We denote by $\F_2[-]\colon \vect \rightarrow \hopf$ the functor that associates with a finite-dimensional vector space its group ring over $\F_2$. For all finite-dimensional vector spaces $V$ over $\F_2$, we denote by $\x^v\in\F_2[V]$ the generator associated with the vector $v\in V$. The generator $\x^0$ is the unit of $\F_2[V]$. Hence, it equals $1$. We use the notation $\m_V$ for the augmentation ideal of $\F_2[V]$. This is the kernel of the augmentation of $\F_2[V]$: the morphism of algebras $\aug_V\colon \F_2[V]\rightarrow \F_2$ that sends every generator $\x^v$ to~$1$. Note that this is the only algebra morphism $\F_2[V]\rightarrow \F_2$. This ideal is known to be generated by the elements of the form $1+\x^v$ for all $v\in V$; \textit{cf.} \cite[Lemma 6.2]{Bih-Fra-Mcc-Ham_eve_tor} or \cite[Lemma 4.1]{Ren-Sha_bou_bet}. We denote by $\gr\colon \hopf \rightarrow \gralg$ the functor that associates with a Hopf algebra $A$ the graded algebra built from the filtration of iterated powers of its augmentation ideal. For instance, if $V$ is a finite-dimensional vector space over $\F_2$, we have%
	\begin{equation*}
	\gr \F_2[V]= \bigoplus_{k\geq 0} \left(\bigslant{\m^k_V}{\m_V^{k+1}}\right).
	\end{equation*}%
	 Moreover, for all $k\in\N$, we denote the functor $V\mapsto \m^k_V$ by $\m^k\colon \vect \rightarrow \vect$. These functors are linked by natural inclusions $\m^{l}\rightarrow \m^k$ for all $l\geq k$.
\end{dfn}

\begin{lem}[\textit{cf.} {\cite[Lemma 6.2]{Bih-Fra-Mcc-Ham_eve_tor}} and {\cite[Lemma 4.1]{Ren-Sha_bou_bet}}]
	For all $k\in\N$, the ideal $\m^k_V$ is spanned, as a vector subspace, by the  elements%
	\begin{equation*}
		\sum_{w\in W} \x^w 
	\end{equation*}%
	 for all $k$-dimensional vector subspaces $W$ of\, $V$.
\end{lem}

\begin{prop}[\textit{cf.} {\cite[Proposition 6.3]{Bih-Fra-Mcc-Ham_eve_tor}}]\label{prop:nat_iso_group_ring}
	We denote by ${\bigwedge \colon \vect \rightarrow \gralg}$ the exterior algebra functor. There is a natural isomorphism $\eta \colon \bigwedge \rightarrow \gr \F_2[-]$ defined, for all finite-dimensional $\F_2$-vector spaces $V$ and all $k\in\N$, by the formula%
	\begin{equation*}
		\begin{array}{rcl}
		\displaystyle \eta_V^k \colon \bigwedge^k V & \longrightarrow & \displaystyle \bigslant{\m_V^k}{\m_V^{k+1}} \\
		v_1\wedge \ldots \wedge v_k & \longmapsto & \displaystyle \prod_{i=1}^k(1+\x^{v_i}) \quad\left(\bmod \m_V^{k+1}\right).
		\end{array}
	\end{equation*}%
\end{prop}

This description is a special case of D.~Quillen's description of graded rings associated with the augmentation filtration of a group algebras; \textit{cf.} \cite{Qui_ass_gra}. The proof of the special case of interest here is recalled in \cite[Proposition 4.3]{Ren-Sha_bou_bet}.

\begin{prop}\label{prop:mult_generator}
	Let $V$ be a finite-dimensional vector space over $\F_2$. For all $v\in V$, the graded endomorphism of $\bigwedge V\cong \gr \F_2[V]$ induced by the multiplication by $\x^v$ is the identity. 
\end{prop}

\begin{proof}
	Let $v$ be a vector of $V$ and $v_1,\ldots,v_k\in V$ be a linearly independent family of $V$. We have%
	\begin{equation*}
	 	\prod_{i=1}^k(1+\x^{v_i})=\x^v\prod_{i=1}^k(1+\x^{v_i})+(1+\x^v)\prod_{i=1}^k(1+\x^{v_i}).
	\end{equation*}%
	 Hence, $\prod_{i=1}^k(1+\x^{v_i})$ and $\x^v\prod_{i=1}^k(1+\x^{v_i})$ correspond to the same element in $\m_V^k/\m_V^{k+1}$.
\end{proof}

\begin{prop}\label{prop:intersection_linear_sub-space}
	Let $V$ be a finite-dimensional vector space over $\F_2$. For all linear subspaces $W\subset V$ and all $k\in\N$, the intersection of the subalgebra $\F_2[W]\subset\F_2[V]$ with $\m_V^k$ is the ideal $\m_W^k$.  
\end{prop}

\begin{proof}
	This is a consequence of the functoriality of $\F_2[-]$ and the natural transformation $\eta$ described in Proposition~\ref{prop:nat_iso_group_ring}.
\end{proof}

In the category of finite-dimensional Hopf algebras, there is a duality contravariant functor. We denote it by ${(-)^*\colon \left(\hopf\right)^\op\rightarrow \hopf}$. If we only remember the vector space structure, it corresponds to the classical functor $(-)^*=\Hom_{\F_2}(-;\F_2)$. If $A$ is a finite-dimensional Hopf algebra, the product of $A^*$ is the adjoint of the coproduct of $A$, the unit of $A^*$ is the adjoint of the counit of $A$, and \textit{vice-versa}. 

\begin{dfn}\label{dfn:functor_func_algebra}
	We denote by $\O$ the functor ${\left(\F_2[-]\right)^*\colon\left(\vect\right)^\op \rightarrow \hopf}$. The functor $\O$ is naturally isomorphic to the functor of $\F_2$-valued function algebras ${V\mapsto \{f\colon V\rightarrow \F_2\}}$ with pointwise addition and multiplication. We will not describe the coalgebra structure as we will not need it further. The natural isomorphism is realised by means of a duality pairing. For all finite-dimensional vector spaces $V$, all functions $f\in \O(V)$, and all elements $P=\sum_{v\in V}p_v\x^v\in\F_2[V]$, we define such a pairing by%
	\begin{equation*}
		\langle \,f\, ; P\,\rangle \coloneqq  \sum_{v\in V} f(v)p_v.
	\end{equation*}%
\end{dfn}

\begin{rem}\label{rem:polynomial}
	We justify the notation $\O(V)$ by the known fact that, for all finite-dimensional vector spaces $V$ over a finite field $\F$, every function $f\colon V\rightarrow \F$ is polynomial. It is only a matter of expressing the indicator functions as polynomials in coordinates of $V$. If one chooses coordinates on $V$, one has a surjective algebra morphism $\F_2[X_1,\ldots,X_n]\rightarrow \O(V)$. We can push forward the degree filtration on $\O(V)$. One can easily check that the pushforward filtration does not depend on the choice of coordinates and has length $n$.
\end{rem}

\begin{dfn}
	The functor $\O^{k}\colon \left(\vect\right)^\op \rightarrow \vect$ is given, for all $k\in\N$ and all finite-dimensional vector spaces $V$ over $\F_2$, by the formula%
	\begin{equation*}
		\O^{k}(V)\coloneqq \im\left( \left(\bigslant{\F_2[V]}{\m^{k+1}}\right)^* \hookrightarrow \O(V) \right).
	\end{equation*}%
	See Definition~\ref{dfn:dual_filtration} or \cite[Equation~(1.1.6)]{Del_the_hod}. Furthermore, these spaces are naturally included in one another: $\O^{k}\hookrightarrow \O^{l}$ for all $l\geq k$. We denote by $\gr\O\colon \left(\vect\right)^\op \rightarrow \grvect$ the functor $\bigoplus_{k\geq 0}\O^{k}/\O^{k-1}$.
\end{dfn}

By construction and Proposition~\ref{prop:nat_iso_group_ring}, we have the following. 

\begin{prop}\label{prop:natural_iso_O_1}
	The functor $\gr\O$ is naturally isomorphic to $\bigwedge(-)^*$ through $\eta^*$.
\end{prop}

Since we did not define a product on $\gr\O$, we should emphasise that in Proposition~\ref{prop:natural_iso_O_1} the two functors $\gr\O$ and $\bigwedge(-)^*$ are naturally isomorphic as functors to the category of graded vector spaces. However, Proposition~\ref{prop:alge_filtration} will ensure that we can endow $\gr\O$ with a product that is compatible with $\eta^*$.

\begin{lem}\label{lem:surjective_map}
	For all functions $u\colon \{1,\ldots,l\}\rightarrow \{1,\ldots,k\}$, the number%
	\begin{equation*}
		\sum_{\beta\in\F_2^k}\prod_{j=1}^l \beta_{u(j)}\in\F_2 
	\end{equation*}%
	 equals $1$ if and only if $u$ is surjective.
\end{lem}

\begin{proof}
	If $u$ is onto, the only non-zero term in the sum is the one for which $\beta_i=1$ for all $i$. If, on the contrary, $u$ is not surjective, it misses an index $i_0$. We have%
	 \begin{equation*}
	 	\sum_{\beta\in\F_2^k}\prod_{j=1}^l \beta_{u(j)}=\sum_{\substack{\beta\in\F_2^k\\\beta_{i_0}=0}}\prod_{j=1}^l \beta_{u(j)}+\sum_{\substack{\beta\in\F_2^k\\\beta_{i_0}=1}}\prod_{j=1}^l \beta_{u(j)}=0 
	\end{equation*}%
	 since the two terms are equal. 
\end{proof}

\begin{prop}\label{prop:alge_filtration}
	The composition $\O^{k}\otimes\O^{l} \rightarrow \O\otimes \O \rightarrow \O$, where the first natural transformation is given by the tensor product of the inclusions and the second is given by the product of\, $\O$, takes values in $\O^{k+l}$. In other words, $\O^{k}\O^{l}\subset \O^{k+l}$.
\end{prop}

\begin{proof}
	Let $V$ be a finite-dimensional vector space over $\F_2$. We need to show that for all $k,l\in\N$, all $f\in\O^{k}(V)$, and all $g\in\O^{l}(V)$, the product $fg$ belongs to $\O^{k+l}(V)$. To do so we will show that a function is an element of $\O^{k}(V)$ if and only if it can be represented by a polynomial of degree at most $k$ (\textit{cf.} Remark~\ref{rem:polynomial}). The functions of $\O^{k}(V)$ are those which vanish against elements of $\m_V^{k+1}$. Following \cite[Lemma~4.1]{Ren-Sha_bou_bet}, the vector space $\m_V^{k+1}$ is generated by the elements $\sum_{v\in W}\x^v$ for all linear subspaces $W$ of $V$ of dimension $k+1$. Such an element can be written as $\prod_{i=1}^{k+1}(1+\x^{w_i})$, where $\{w_1,\ldots,w_{k+1}\}$ is a basis of $W$. Let us fix a set of linear coordinates $x_1,\ldots,x_n$ of $V$, \textit{i.e.}~a basis of $V^*$. Because of Fermat's little theorem,\footnote{$a^2=a$ for all $a\in\F_2$} a function $f\in\O(V)$ can be uniquely written as%
	\begin{equation*}
		f=\sum_{\alpha\in \{0,1\}^n} f_\alpha x^\alpha 
	\end{equation*}%
	 in the usual multi-index notation: $x^\alpha\coloneqq \prod_{i=1}^nx_i^{\alpha_i}$ and $|\alpha|\coloneqq \sum_{i=1}^n\alpha_i$. For all integers $k\geq 1$, let us denote by $[k]$ the integer interval $\{1,\ldots,k\}$. Let $\alpha\in\F_2^n$ with $|\alpha|=l$, and let $W=\langle w_1,\ldots,w_k\rangle$ be a $k$-dimensional subspace of $V$. We have%
	\begin{equation*}
		\begin{split}
			\left\langle x^\alpha\; ;\prod_{i=1}^k(1+\x^{w_i})\right\rangle &=\sum_{w\in W} \langle x^\alpha\;;\;\x^w\rangle \\
			&= \sum_{w\in W} \prod_{j=1}^{n} x_j(w)^{\alpha_j} \\
			&= \sum_{\beta\in\F_2^k}\prod_{j=1}^{n} \left(x_j\left(\sum_{i=1}^k \beta_iw_i\right)\right)^{\alpha_j}\\
			&= \sum_{\beta\in\F_2^k} \sum_{u:[n]\rightarrow[k]} \left(\,\prod_{j=1}^n \beta_{u(j)}^{\alpha_j}\, \right)\left(\,\prod_{j=1}^n \left(x_j(w_{u(j)})\right)^{\alpha_j}\,\right).
		\end{split}
	\end{equation*}%
	 Up to reordering, we can assume that $\alpha_i=1$ if and only if $1\leq i\leq l$; thus %
	\begin{equation*}
		\begin{split}
			\sum_{\beta\in\F_2^k} \sum_{u:[n]\rightarrow[k]} \left(\,\prod_{j=1}^n \beta_{u(j)}^{\alpha_j} \right)\left(\,\prod_{j=1}^n \left(x_j(w_{u(j)})\right)^{\alpha_j}\right)&=\sum_{\beta\in\F_2^k} \sum_{u:[l]\rightarrow[k]} \left(\,\prod_{j=1}^l \beta_{u(j)}\right)\left(\,\prod_{j=1}^l x_j(w_{u(j)})\right) \\
			&= \sum_{u:[l]\rightarrow[k]} \left(\,\sum_{\beta\in\F_2^k}\prod_{j=1}^l \beta_{u(j)}\right)\left(\,\prod_{j=1}^l x_j(w_{u(j)})\right).
		\end{split}
	\end{equation*}%
	 By Lemma~\ref{lem:surjective_map}, the number $\sum_{\beta\in\F_2^k}\prod_{j=1}^l \beta_{u(j)}$ equals $1$ if and only if $u$ is onto. Therefore, the number%
	\begin{equation*}
		\left\langle x^\alpha\; ;\prod_{i=1}^k(1+\x^{w_i})\right\rangle
	\end{equation*}%
	 vanishes whenever $k$ is bigger than $l$. Thus the polynomial functions of degree at most $k$ belong to $\O^{k}(V)$. If $k$ equals $l$ and $\mathfrak{S}_k$ denotes the $\supth{k}$ symmetric group, the number%
	\begin{equation*}
		\left\langle x^\alpha\; ;\prod_{i=1}^k(1+\x^{w_i})\right\rangle=\sum_{u\in\mathfrak{S}_k} \prod_{j=1}^l x_j(w_{u(j)}) 
	\end{equation*}%
	 is the minor of the matrix of the coordinates of $w_1,\ldots,w_k$ in the dual basis of $x_1,\ldots,x_n$ with the row set selected by the multi-index $\alpha$. As a consequence, we can find, for all polynomials of degree $k$, a subspace~$W$ of dimension $k$ such that the evaluation of $f$, as a linear form, against the generator of $\m^k_V$ associated with~$W$ is $1$. This implies that $\{f\colon \deg f=k\}$ is included in $\O^{k}(V)\setminus \O^{k-1}(V)$. Thus, the two partitions of $\O(V)$, by the degree on one hand, and by the sets $\O^{k}(V)\setminus\O^{k-1}(V)$ on the other hand, coincide.
\end{proof}

As a direct consequence of Proposition~\ref{prop:alge_filtration}, we have the following. 

\begin{prop}\label{prop:filt_alg_fun}
	The functor $\O$ together with the filtration $(\O^{k})_{k\in\N}$ takes values in the category of filtered finite-dimensional algebras over $\F_2$, and therefore $\gr\O$ takes values in the category of graded algebras.
\end{prop}

Moreover, we have the following.

\begin{prop}
  The natural isomorphism $\eta^*\colon \gr\O \rightarrow \bigwedge(-)^*$ is compatible with the products.
\end{prop}

\begin{proof}
	Let $V$ be a finite-dimensional vector space over $\F_2$, and let $k\in\N$. Let us consider $f\in \O^{k}(V)$ and express it in a system of linear coordinates $x_1,\ldots,x_n$:%
	\begin{equation*}
		f=\sum_{|\alpha|\leq k} f_\alpha x^\alpha = \sum_{\substack{I\subset \{1,..,n\} \\ |I|\leq k}} f_I\prod_{i\in I}x_i\, .
	\end{equation*}%
	 Using Lemma~\ref{lem:surjective_map}, we find that, for all $v_1\wedge\ldots\wedge v_k \in \bigwedge^k V$, %
	\begin{equation*}
		\begin{split}
			\left\langle f\;\left(\bmod \O^{k-1}(V)\right)\;;\;\eta_V^k(v_1\wedge\ldots\wedge v_k)\right\rangle &= \sum_{|I|=k} f_I\left\langle \prod_{i\in I}x_i\;;\prod_{j=1}^k(1+\x^{v_j})\right\rangle \\
			&= \sum_{|I|=k} f_I \sum_{u\in\mathfrak{S}_k}\prod_{i=1}x_{I(i)}(v_{u(i)}) \\
			&= \sum_{|I|=k} f_I \left(\,\bigwedge_{i\in I}x_i\,\right)\left(\,\bigwedge_{i=1}^nv_i\,\right),
		\end{split}
	\end{equation*}%
	 where $I(i)$ is the $\supth{i}$ element of $I$. The inverse of $(\eta_V^{k})^*$ sends a wedge product of linear forms $\alpha_1\wedge\cdots\wedge\alpha_k$ to the polynomial function $\alpha_1\cdots\alpha_k$ modulo $\O^{k-1}(V)$. It respects the product structures.
\end{proof}

\begin{dfn}\label{dfn:contract_grp_alg}
	Let $V$ be a finite-dimensional vector space over $\F_2$. For all $f\in\O(V)$ and all $P\in\F_2[V]$ written as $P=\sum_{v\in V}p_v\x^v$, we denote by $f\cdot P$ the element of $\F_2[V]$ defined by the formula%
	\begin{equation*}
		f\cdot P\coloneqq \sum_{v\in V}f(v)p_v\x^v.
	\end{equation*}%
\end{dfn}

\begin{prop}\label{prop:contraction}
	Let $V$ be a finite-dimensional vector space over $\F_2$. For all $f\in\O^{l}(V)$ and all $P\in\m^{k+l}_V$, $f\cdot P$ belongs to $\m^k_V$. Moreover, the induced product between the graded pieces%
	\begin{equation*}
		\bigwedge^lV^*\otimes \bigwedge^{k+l} V \longrightarrow \bigwedge^{k} V,
	\end{equation*}%
	 is the contraction.
\end{prop}

\begin{proof}
	Let $a\in\O^{1}(V)$ be an affine function. We denote its linear part by $\df a$. A straightforward computation yields for all $P,Q\in \F_2[V]$%
	\begin{equation*}
		a\cdot(PQ)=(a\cdot P)Q+P(\df a\cdot Q).
	\end{equation*}%
	 Using this formula, we can show, by induction, that $a\cdot\m^{k+1}_V$ is included in $\m_V^k$ for all $k\geq 0$. From the definition, $(fg)\cdot P$ equals $f\cdot(g\cdot P)$ for all $f,g\in\O(V)$ and all $P\in\F_2[V]$. Consequently, $(a_1\cdots a_l)\cdot\m^{k+l}_V$ is included in $\m^{k}_V $ for all $a_1,\ldots,a_l\in\O^{1}(V)$. Finally, we find the following inclusions:
	\begin{equation*}
		\O^{l}(V)\cdot\m_V^{k+l}\subset \sum_{i=0}^l\m_V^{k+l-i}\subset \m_V^{k}.
	\end{equation*}
	 The duality pairing $\langle f;P\rangle$ of $f\in\O(V)$ with $P\in\F_2[V]$ can be written as $\aug_V(f\cdot P)$. Hence, the operation $P\mapsto f\cdot P$ is adjoint to $g\mapsto gf$. This remains true in the graded algebras, and it is precisely the definition of the contraction. 
\end{proof}

\begin{lem}\label{lem:contraction}
	Let $V$ be a finite-dimensional vector space over $\F_2$ and $A$ be an affine subspace of\, $V$ of codimension~$l$. Let $\chi_A$ denote the indicator function of $A$. We have the following exact sequence:%
	\begin{equation}\label{F}
		\begin{tikzcd}
		0 \ar[r] & \langle\,\x^v\colon v\in V\setminus A\,\rangle \ar[r] & \F_2[V] \ar[r,"\chi_A\cdot-"] & \langle\,\x^v\colon v\in A\,\rangle \ar[r] & 0.
		\end{tikzcd}
	\end{equation}%
	 Moreover, if $\omega$ denotes the generator of $\bigwedge^l\{\alpha~|~\alpha(v)=0,\,\forall v\in TA\}$, $K$ denotes $\langle\,\x^v\colon v\in V\setminus A\,\rangle$, and $o$ is a point of $A$, then for all $k\in\N$, we have the following diagram with exact rows and columns:%
	\begin{equation}\label{F$k$}
		\begin{tikzcd}[column sep=1.25cm, row sep=.5cm]
			 & 0 & 0 & 0 \\
			0 \ar[r] & \bigslant{K\cap\m^{k}_V}{K\cap\m^{k+1}_V} \ar[r] \ar[u] & \displaystyle \bigwedge^k V \ar[u] \ar[r,"\omega\cdot-" above] & \displaystyle \bigwedge^{k-l} TA \ar[r] \ar[u] & 0  \\
			0 \ar[r] & K\cap\m^{k}_V \ar[r] \ar[u] & \displaystyle \m^k_V \ar[r,"\x^o(\chi_A\cdot-)" above] \ar[u] & \displaystyle \m^k_{TA} \ar[r] \ar[u] & 0  \\
			0 \ar[r] & K\cap\m^{k+1}_V \ar[r] \ar[u] & \displaystyle \m^{k+1}_V \ar[r,"\x^o(\chi_A\cdot-)" above] \ar[u] & \displaystyle \m^{k+1}_{TA} \ar[r] \ar[u] & 0  \\
			 & 0 \ar[u] & 0 \ar[u] & 0\rlap{.} \ar[u] 
		\end{tikzcd}
	\end{equation}%
\end{lem}

\begin{proof}
	The space $\F_2[V]$ is the direct sum of $\langle\,\x^v\colon v\in A\,\rangle$ and $\langle\,\x^v\colon v\in V\setminus A\,\rangle$. A direct computation yields that the operation $P\mapsto \chi_A\cdot P$ is the projection onto $\langle\,\x^v\colon v\in A\,\rangle$ parallel to $\langle\,\x^v\colon v\in V\setminus A\,\rangle$. Therefore, we have the exact sequence \eqref{F}:%
	\begin{equation*}
		\begin{tikzcd}
		0 \ar[r] & \langle\,\x^v\colon v\in V\setminus A\,\rangle \ar[r] & \F_2[V] \ar[r,"\chi_A\cdot-"] & \langle\,\x^v\colon v\in A\,\rangle \ar[r] & 0.
		\end{tikzcd}
	\end{equation*}%
	 Since $A$ has codimension $l$, it is the intersection of $l$ affine hyperplanes. The indicator functions of affine hyperplanes have degree $1$. Hence, $\chi_A$ has degree $l$. Following Proposition~\ref{prop:contraction}, the linear map $P\mapsto \chi_A\cdot P$ has degree $-l$ as a morphism of filtered vector spaces. Moreover, because of Proposition~\ref{prop:mult_generator}, choosing a point $o\in A$ yields an isomorphism of filtered vector spaces ${P\in\F_2[TA]\mapsto \x^o P\in \langle\,\x^v\colon v\in A\,\rangle}$. Therefore, ${P\mapsto\x^o(\chi_A\cdot P)}$ is a surjective filtered morphism of degree $-l$ from $\F_2[V]$ onto $\F_2[TA]$. Consequently, for all $k\in \N$, we obtain, by partial use of the snake lemma, see \cite[Lemma 9.1]{Lan_alg}, the following commutative diagram with exact rows and columns:%
	\begin{equation*}
		\begin{tikzcd}[column sep=4em, row sep=1.5em]
			0 & 0 & 0\\
			\displaystyle \bigwedge^k V \ar[u] \ar[r,"f_k" above] & \displaystyle \bigwedge^{k-l} TA \ar[r] \ar[u] & C_k \ar[r] \ar[u] & 0  \\
			\displaystyle \m^k_V \ar[r,"\x^o(\chi_A\cdot-)" above] \ar[u] & \displaystyle \m^{k-l}_{TA} \ar[r] \ar[u] & C_{(k)} \ar[r] \ar[u] & 0  \\
			\displaystyle \m^{k+1}_V \ar[r,"\x^o(\chi_A\cdot-)" above] \ar[u] & \displaystyle \m^{k-l+1}_{TA} \ar[r] \ar[u] & C_{(k+1)} \ar[r] \ar[u] & 0  \\
			 0 \ar[u] & 0\rlap{,} \ar[u] 
		\end{tikzcd}
	\end{equation*}%
	 where $f_k$ is induced by $\x^o(\chi_A\cdot-)$ between the $\supth{k}$ graded pieces and the rightmost column is made of the cokernels of the associated morphisms. Using Propositions~\ref{prop:mult_generator} and~\ref{prop:contraction}, we find that the morphism $f_k$ is the contraction against $\omega$. It is a surjective morphism, and $C_k$ vanishes. Since all the $C_k$ vanish, the $C_{(k)}$ must vanish as well. Using now the snake lemma in its complete form, we find the commutative diagram \eqref{F$k$}.  
\end{proof}

\begin{prop}\label{prop:same_filtrations}
	Let $V$ be a finite-dimensional vector space over $\F_2$ and ${A\subset V}$ be an affine subspace of codimension $l$. For all families $H_1,\ldots,H_l\subset V$ of affine hyperplanes that intersect along $A$, and all $k\in \N$, we have%
	\begin{equation*}
		\langle\, \x^v\colon v\in V\setminus A\,\rangle\cap\m_V^k=\sum_{i=1}^l\left(\langle\, \x^v\colon v\in V\setminus H_i\,\rangle\cap\m_V^k\right).
	\end{equation*}%
\end{prop}

\begin{proof}
	Let us denote by $K$ the subspace $\langle\, \x^v\colon v\in V\setminus A\,\rangle$ of $\F_2[V]$, also equal to $\sum_{i=1}^l\langle\, \x^v\colon v\in V\setminus H_i\,\rangle$. It is endowed with two decreasing filtrations:%
	\begin{equation*}
		K_{k}\coloneqq \langle\, \x^v\colon v\in V\setminus A\,\rangle\cap\m_V^k
	\end{equation*}%
	 and%
	\begin{equation*}
		K'_{k}\coloneqq \sum_{i=1}^l\left(\langle\, \x^v\colon v\in V\setminus H_i\,\rangle\cap\m_V^k\right).
	\end{equation*}%
	We need to show that they coincide. Since $K'_{k}$ is included in $K_{k}$ for all $k\in\N$, we only need to show that the induced morphisms between the graded pieces are surjective. Because of Propositions~\ref{prop:mult_generator} and~\ref{prop:intersection_linear_sub-space}, $K'_{k}$ is equal to $\sum_{i=1}^l\x^{o_i}\m^k_{TH_i}$, where, for all $1\leq i\leq l$, $o_i$ is a point of $V\setminus H_i$ and $TH_i$ is the direction of $H_i$. It follows that the image of $K'_{k}$ in $\bigwedge^k V$ is%
	\begin{equation*}
		\sum_{i=1}^l\bigwedge^kTH_i\, .
	\end{equation*}%
	 As a consequence of the definition of the other filtration $(K_{i})_{i\geq 0}$, its associated graded vector space is embedded in the exterior algebra $\bigwedge V$. Moreover, if $\omega$ is the generator of $\bigwedge^l\left(\smallslant{V}{TA}\right)^*$ and $k$ is a non-negative integer, we have, from Diagram~(\ref{F$k$}) of Lemma~\ref{lem:contraction}, the  exact sequence%
	\begin{equation}\label{E$k$}
		\begin{tikzcd}
		0 \ar[r]& K_k \ar[r,"c_k"] & \bigwedge^k V \ar[r,"\omega\cdot-" above]& \bigwedge^{k-l} TA \ar[r]& 0
		\end{tikzcd}
	\end{equation}%
	 and the commutative diagram%
	\begin{equation*}
		\begin{tikzcd}[row sep=.5em]
			K'_k \ar[dd,"a_k" left] \ar[dr,"b_k" above right] \\
			& \displaystyle \bigwedge^k V \\
			K_k\rlap{,} \ar[ur,"c_k" below right]
		\end{tikzcd}
	\end{equation*}
where $a_k$ is induced from the inclusions $(K'_{k}\subset K_{k})_{k\in\N}$, $b_k$ is induced from ${(K'_{k}\subset \m^k_V)_{k\in\N}}$, and $c_k$ is induced from ${(K_{k}\subset \m^k_V)_{k\in\N}}$. The morphism $c_k$ is injective. Proposition~\ref{prop:F_and_contraction} and (\ref{E$k$}) imply that $c_k$ and $b_k$ have the same image, so  $a_k$ is onto. This implies that all the inclusions $K'_{k}\subset K_{k}$, for all $k\in \N$, are onto, hence equalities.
\end{proof}

\subsubsection*{The induced filtrations}

\begin{dfn}\label{dfn:filt_KP_OP}
	In Definition~\ref{dfn:real_sheaves}, the cubical cosheaf $\K^{\R P}$ and the cubical sheaf $\O_{\R P}$ were given by the formulae $\K^{\R P}=\F_2[F_1^P]$ and $\O_{\R P}=\O(F_1^P)$. The functoriality of the filtrations of $\F_2[-]$ and $\O$ ensures that this cosheaf and this sheaf are, respectively, endowed with a decreasing filtration
	\begin{equation*}
		0\subset \K_{n}^{\R P}\subset \cdots\subset \K_{0}^{\R P}=\K^{\R P}
	\end{equation*}%
	 and an increasing filtration%
	\begin{equation*}
		0\subset \O^{0}_{\R P}\subset \cdots\subset \O^{n}_{\R P}=\O_{\R P},
	\end{equation*}%
	 and that the associated graded pieces satisfy, for all $k\in \N$, %
	\begin{equation}\label{K1}
		0\longrightarrow \mathcal{K}_{k+1}^{\R P} \longrightarrow \mathcal{K}_{k}^{\R P} \xrightarrow[\bv^P_k]{} F_k^P \longrightarrow 0 
	\end{equation}%
	 and%
	\begin{equation}\label{O1}
		0\longrightarrow \O^{k-1}_{\R P} \longrightarrow \O^{k}_{\R P} \xrightarrow[\bv^k_P]{} F^k_P \longrightarrow 0\, .
	\end{equation}%
	 The projections $\bv_k^P$ and $\bv^k_P$ are given by the projections onto the graded pieces composed with the natural isomorphisms $\eta$ and $(\eta^*)^{-1}$. We denote by%
	\begin{itemize}
	\item $(E^r_{p,q}(\R P))_{p,q,r\geq 0}$ the spectral sequence converging towards $(H_q(\R P;\F_2))_{q\geq 0}$ that arises from the bounded filtration of $\K_{\R P}$, 
	\item $(E_r^{p,q}(\R P))_{p,q,r\geq0}$ the dual spectral sequence converging towards $(H^q(\R P;\F_2))_{q\geq0}$ that arises from the bounded filtration of $\O_{\R P}$.
	\end{itemize}
        The spectral sequence $(E_r^{p,q}(\R P))_{p,q,r\geq 0}$ is a spectral ring because $\O_{\R P}$ is a sheaf of filtered algebras; see Proposition~\ref{prop:filt_alg_fun}.
\end{dfn}

\begin{rem}\label{rem:image_throught_bv}
	If $\x^o\prod_{i=1}^k(1+\x^{v_i})$ belongs to $\K^{\R P}_{k}(\sigma^p;\sigma^q)$, its image under $\bv^P_k(\sigma^p;\sigma^q)$ is $\bigwedge_{i=1}^kv_i$. This is a simple consequence of Propositions~\ref{prop:nat_iso_group_ring} and~\ref{prop:mult_generator}.
\end{rem}

These filtrations can be transported on $\R X_\varepsilon$.

\begin{dfn}\label{dfn:filt_KX_OX}
	The filtration of $\K^{\R P}$ from Definition~\ref{dfn:filt_KP_OP} induces a filtration of the sub-cosheaf ${\K^\X \subset \K^{\R P}}$ by intersection. Dually, the filtration of $\O_{\R P}$ induces a filtration of $\O_\X$ by projection. With these definitions the inclusion ${i_*\colon \K^\X\rightarrow \K^{\R P}}$ and the projection ${i^*\colon \O_{\R P}\rightarrow \O_\X}$ are morphisms of filtered cosheaves and sheaves, respectively. We denote by%
	\begin{equation*}
		\left(E^r_{p,q}(\X)\right)_{p,q,r\geq 0} \quad\textnormal{and}\quad \left(E_r^{p,q}(\X)\right)_{p,q,r\geq 0} 
	\end{equation*}%
	 the spectral sequences associated with the filtrations of $\K^\X$ and of $\O_\X$, respectively. The sheaves and cosheaves morphisms $i^*$ and $i_*$ induce spectral sequences morphisms 
	\begin{equation*}
		i_{p,q}^r \colon E^r_{p,q}(\X) \longrightarrow E^r_{p,q}(\R P) \quad \textnormal{and}\quad i^{p,q}_r \colon E_r^{p,q}(\R P) \longrightarrow E_r^{p,q}(\X).
	\end{equation*}
	 The multiplicative property of the filtration of $\O_{\R P}$ is passed on to the filtration of $\O_{\X}$ as ${i^*\colon \O_{\R P}\rightarrow \O_\X}$ is a morphism of sheaves of algebras. Therefore, $(E_r^{p,q}(\X))_{p,q,r\geq0}$ is also a spectral ring. Furthermore, the collection $(i^{p,q}_r)_{p,q,r\geq 0}$ is a morphism of spectral rings.
\end{dfn}

The following lemma will ensure that the filtrations $(\K^\X_k)_{k\geq 0}$ and $(\O_\X^{k})_{k\geq 0}$ are dual to each other.

\begin{lem}\label{lem:dual_filt_X}
	Let $W_0\subset V_0$ be two finite-dimensional vector spaces over $\F$. Let $V_1$ be a subspace of\, $V_0$ and $W_1$ denote $V_1\cap W_0$. Then the image of\, $V_1^\perp$ in $W_0^*$ is $W_1^\perp$.
\end{lem}

\begin{proof}
	The inclusion of the image of $V_1^\perp$ in $W_1^\perp$ is straightforward. Let us show that we can obtain every form of $W_1^\perp$ as the restriction to $W_0$ of a form of $V_1^\perp$. We can write $V_0$ as the direct sum $W_1\oplus U\oplus U'$ in such a way that $W_1\oplus U$ is $V_1$ and $W_1\oplus U'$ is $W$. If $\alpha\in W_1^\perp$, we extend it by $0$ on $U$ to obtain a form in $V_1^\perp$ that restricts to the original one.  
\end{proof}

The main theorem of \cite{Ren-Sha_bou_bet} rests on the computation of the first page of the Renaudineau--Shaw spectral sequence and more specifically the graded pieces of $\K^\X$. However, we gave a different definition of the filtration of $\K^\X$. Before recalling the results of \cite{Ren-Sha_bou_bet} about the spectral sequence $(E^r_{p,q}(\X))_{p,q,r\geq 0}$, we should check that the two definitions coincide. Let us recall their construction.

\begin{dfn}[\textit{cf.} {\cite[Definition 4.5]{Ren-Sha_bou_bet}}]\label{dfn:filt_X_RS}
	Let us denote by $((\K^\X_k)')_{k\geq 0}$ the decreasing filtration of $\K^\X$ constructed as follows:
	\begin{enumerate}
		\item Let $\sigma^1\leq\sigma^q$ be a pair of simplices of $K$. The set $\Arg_\varepsilon(\sigma^1;\sigma^q)$ is an affine hyperplane of the vector space $F^P_1(\sigma^1;\sigma^q)=\smallslant{N\otimes\F_2}{\Sed(\sigma^1;\sigma^q)}$ (\textit{cf.}~Lemma~\ref{lem:desc_arg}). Let us denote its vectorial direction by $TA$ and choose an origin $o\in\Arg_\varepsilon(\sigma^1;\sigma^q)$. Let $k\in\N$, and let $(\K^\X_k)'(\sigma^1;\sigma^q)$ denote the image of the $\supth{k}$ power augmentation ideal $\m_{TA}^k$ by the isomorphism $\F_2[TA]\rightarrow \K^\X(\sigma^1;\sigma^q)$ obtained from the multiplication by $\x^o$.
		\item Let $\sigma^p\leq\sigma^q$ be a pair of simplices of $K$ and $k\in \N$. The subspace $(\K^\X_k)'(\sigma^p;\sigma^q)$ of $\K^\X(\sigma^p;\sigma^q)$ is obtained as the sum%
		\begin{equation*}
			\left(\K^\X_k\right)'(\sigma^p;\sigma^q)\coloneqq \sum_{\sigma^1\leq\sigma^p}\left(\K^\X_k\right)'(\sigma^1;\sigma^q).
		\end{equation*}
	\end{enumerate}%
\end{dfn}

\begin{rem}
	In Definition~\ref{dfn:filt_X_RS}, we actually defined the cubical subdivision of the filtration given in \cite[Definition 4.5]{Ren-Sha_bou_bet}.
\end{rem}

\begin{prop}\label{prop:same_filtrations2}
	Let $\sigma^p\leq \sigma^q$ be a pair of simplices of $K$. For all integers $k\geq 0$, %
	\begin{equation*}
	\left(\K^\X_k\right)'(\sigma^p;\sigma^q)=\K^\X_k(\sigma^p;\sigma^q).
	\end{equation*}%
\end{prop}

\begin{proof}
	Let $k\geq 0$ be an integer. From Definition~\ref{dfn:filt_KX_OX}, we have%
	\begin{equation*}
		\K^\X_k(\sigma^p;\sigma^q)=\left\langle\x^v\colon v\in\Arg_\varepsilon(\sigma^p;\sigma^q)\right\rangle\cap \m^k_V\, ,
	\end{equation*}%
	 where $V$ is $\smallslant{N\otimes\F_2}{\Sed(\sigma^p;\sigma^q)}$. Lemma~\ref{lem:desc_arg} ensures that $\Arg_\varepsilon(\sigma^p;\sigma^q)$ is the complementary subset of an affine subspace of $V$. Let $\sigma^0$ be a vertex of $\sigma^p$. We can apply Proposition~\ref{prop:same_filtrations} to find that%
	\begin{equation*}
		\left\langle\x^v\colon v\in\Arg_\varepsilon(\sigma^p;\sigma^q)\right\rangle\cap \m^k_V=\sum_{\sigma^0\subset\sigma^1\subset\sigma^p}\left\langle\x^v\colon v\in\Arg_\varepsilon(\sigma^1;\sigma^q)\right\rangle\cap \m^k_V\, .
	\end{equation*}%
	 Indeed,  for all edges $\sigma^0\subset\sigma^1\subset\sigma^p$, the sets $V\setminus\Arg_\varepsilon(\sigma^1;\sigma^q)$ are hyperplanes intersecting along the set $V\setminus\Arg_\varepsilon(\sigma^p;\sigma^q)$. We can add the remaining ${\left\langle\x^v\colon v\in\Arg_\varepsilon(\sigma^1;\sigma^q)\right\rangle\cap \m^k_V}$ on both sides to find that%
	\begin{equation*}
		\left\langle\x^v\colon v\in\Arg_\varepsilon(\sigma^p;\sigma^q)\right\rangle\cap \m^k_V=\sum_{\sigma^1\subset\sigma^p}\left\langle\x^v\colon v\in\Arg_\varepsilon(\sigma^1;\sigma^q)\right\rangle\cap \m^k_V\, .
	\end{equation*}%
	 Because of Propositions~\ref{prop:mult_generator} and~\ref{prop:intersection_linear_sub-space}, the intersection of the subspace ${\langle\x^v\colon v\in\Arg_\varepsilon(\sigma^1;\sigma^q)\rangle}$ with $\m^k_V$ corresponds to $(\K^\X_k)'(\sigma^1;\sigma^q)$. Using the second construction step of Definition~\ref{dfn:filt_X_RS}, we finally find that the two filtrations coincide.
\end{proof}

A.~Renaudineau and K.~Shaw showed that  the graded pieces of the filtration of $\K^\X$ are isomorphic to the cosheaves $(F_k^X)_{k\geq 0}$; \textit{cf.} \cite[Lemma~4.8 and Proposition~4.10]{Ren-Sha_bou_bet}. To do so they introduced Borel--Viro morphisms for $\X$; \textit{cf.} \cite[Definition 4.9]{Ren-Sha_bou_bet}. The paragraph following this definition characterises these morphisms as follows: let $\sigma^p\leq\sigma^q$ be two simplices of $K$ if $\x^o\prod_{i=1}^k(1+\x^{v_i})$ belongs to $\K^{\X}_{k}(\sigma^p;\sigma^q)$; its image under $\bv^X_{k}(\sigma^p;\sigma^q)$ is $\bigwedge_{i=1}^kv_i\in F_k^X(\sigma^p;\sigma^q)\subset F_k^P(\sigma^p;\sigma^q)$. Note that \cite[Lemma 4.4]{Ren-Sha_bou_bet} ensures that the elements of this form span $\K^{\X}_{k}(\sigma^p;\sigma^q)$. Together with Remark~\ref{rem:image_throught_bv}, this fact implies the following proposition.

\begin{prop}\label{prop:graded_cosheaves}
	For all integers $k\geq0$, the restriction of the Borel--Viro morphism $\bv^P_k$ to $\K^\X_k$ is the Borel--Viro morphism $\bv_k^X$ introduced in \cite[Definition 4.9]{Ren-Sha_bou_bet}. In particular, we have the following commutative diagram with exact rows:
	\begin{equation}\label{K2}
		\begin{tikzcd}
			0 \ar[r] & \K^{\R P}_{k+1} \ar[r] & \K^{\R P}_{k} \ar[r,"\bv^P_k" above] & F^P_k \ar[r] & 0 \\
			0 \ar[r] & \K^{\X}_{k+1} \ar[r] \ar[u,"i_*" right] & \K^{\X}_{k} \ar[r, "\bv^X_k" above] \ar[u,"i_*" right] & F^X_k \ar[r] \ar[u, "i_k" right] & 0 \\
			 & 0 \ar[u] & 0 \ar[u] & 0\rlap{.} \ar[u] & 
		\end{tikzcd}
	\end{equation}%
\end{prop}

This proposition has the following dual counterpart.
		
\begin{prop}\label{prop:graded_sheaves}	
	The Borel--Viro morphism $\bv_P^k$ of Definition~\ref{dfn:filt_KP_OP} composed with $i^k\colon F^k_P\rightarrow F^k_X$ can be factored by $i^*\colon \O_{\R P}^{k}\rightarrow \O_\X^{k}$; we denote this factorisation by $\bv_X^k$. As a consequence, we have the following commutative diagram with exact rows and columns:%
	\begin{equation}\label{O2}
		\begin{tikzcd}
			0 \ar[r] & \O^{k-1}_{\R P} \ar[r] \ar[d,"i^*" right] & \O^{k}_{\R P} \ar[r,"\bv_P^k" above] \ar[d,"i^*" right] & F^k_P \ar[r] \ar[d, "i^k" right] & 0 \\
			0 \ar[r] & \O^{k-1}_\X \ar[r] \ar[d] & \O^{k}_\X \ar[r, "\bv_X^k" below] \ar[d] & F^k_X \ar[r] \ar[d] & 0 \\
			 & 0 & 0 & 0\rlap{.} & 
		\end{tikzcd}
	\end{equation}
\end{prop}

\begin{proof}
	Let $\sigma^p\leq\sigma^q$ be a pair of simplices of the triangulation $K$ and $k\in\N$. By dualising Diagram~(\ref{F$k$}) of Lemma~\ref{lem:contraction} and using Proposition~\ref{prop:F_and_contraction}, we derive the following commutative diagram with exact rows and columns:%
	\begin{equation*}
		\begin{tikzcd}
			 & 0 & 0 & 0 \\
			0 \ar[r] & \displaystyle \bigwedge^{k-p} (TA)^* \ar[r,"\omega\wedge -" above] \ar[u] & \displaystyle F_P^k(\sigma^p;\sigma^q) \ar[u] \ar[r,"i^k" above] & F_X^k(\sigma^p;\sigma^q)  \ar[r] \ar[u] & 0  \\
			0 \ar[r] & \displaystyle \O^{k-p}(TA) \ar[r,"\text{ext}_A" above] \ar[u,"\bv_{TA}^{k-p}" left] & \displaystyle \O_{\R P}^{k}(\sigma^p;\sigma^q) \ar[r,"i^*" above] \ar[u,"\bv_P^k" left] & \O_\X^{k}(\sigma^p;\sigma^q) \ar[r] \ar[u,"\bv_X^k" left] & 0  \\
			0 \ar[r] & \displaystyle \O^{k-1-p}(TA) \ar[r,"\text{ext}_A" above] \ar[u] & \displaystyle \O_{\R P}^{k-1}(\sigma^p;\sigma^q) \ar[r,"i^*" above] \ar[u] & \O_\X^{k-1}(\sigma^p;\sigma^q) \ar[r] \ar[u] & 0  \\\
			 & 0 \ar[u] & 0 \ar[u] & 0\rlap{.} \ar[u] 
		\end{tikzcd}
	\end{equation*}	
	In the diagram, $A$ denotes the $p$-codimensional affine subspace of $F_1^P(\sigma^p;\sigma^q)$ defined as the complement of $\Arg_\varepsilon(\sigma^p;\sigma^q)$. The morphism $\ext_A$ is the composition of the translation isomorphism 
	\begin{equation*}
		f\in\O(TA)\longmapsto [x\mapsto f(x+o)]\in\O(A),
	\end{equation*}
for a choice of $o\in A$, with the extension by $0$ to a map of $\O\left(F_1^P(\sigma^p;\sigma^q)\right)$. The vector $\omega$ is the generator of the dual of $\bigwedge^p\smallslant{F_1^P(\sigma^p;\sigma^q)}{TA}$.
\end{proof}

\subsubsection*{Properties of the spectral sequences} The two following theorems describe the already known properties of the Renaudineau--Shaw spectral sequences.

\begin{thm}[\textit{cf.} {\cite[Section 5.15]{Bor-Hae_cla_hom}}]\label{cor:struct_spec_seq_P}
	The spectral sequences $(E_{p,q}^r(\R P))_{p,q,r\geq 0}$ and $(E^{p,q}_r(\R P))_{p,q,r\geq 0}$ degenerate at the first page $($i.e.~$\partial^r$ and $\df_r$ vanish for all integers $r\geq 1)$. Moreover, the graded algebra $(H^q(\R P;\F_2);\cup)_{q\geq 0}$ is isomorphic to $(H^{q,q}(P;\F_2);\cup)_{q\geq 0}$.
\end{thm}

\begin{thm}[Structure of the spectral sequences of $\X$, \textit{cf.} {\cite[Lemma~6.3 and Proposition 4.12]{Ren-Sha_bou_bet}}]\label{cor:struct_spec_seq_X}
	The first page $( E^1_{p,q}(\X))_{p,q\geq 0}$ is isomorphic to the tropical homology $( H_{p,q}(X;\F_2))_{p,q\geq 0}$. Moreover, for all integers $r\geq 2$, the only possibly non-trivial groups of the $\supth{r}$ page $(E^r_{p,q}(\X))_{p,q\geq 0}$ are located on the line segments ${\{p=q \textnormal{ and } 0\leq q\leq n-1\}}$ and ${\{p+q=n-1 \textnormal{ and } 0\leq q\leq n-1\}}$. Figure~\ref{fig:spec_seq_RX} depicts the shape of such a spectral sequence. Hence, the only possibly non-trivial boundary operators $(\partial^r_{p,q})_{p,q\geq0}$ are%
	\begin{equation*}
		\partial^r_{\frac{n-r}{2},\frac{n-r}{2}}\quad\textnormal{and}\quad\partial^r_{\frac{n-r}{2}-1,\frac{n+r}{2}},
	\end{equation*}%
	 when $r$ is congruent to $n$ modulo 2. By duality, the same holds for the dual page $(E_r^{p,q}(\X))_{p,q\geq0}$. However, this time the possibly non-trivial differentials are%
	\begin{equation*}
		\df_r^{\frac{n+r}{2},\frac{n-r}{2}-1}\quad\textnormal{and}\quad\df_r^{\frac{n+r}{2}-1,\frac{n+r}{2}-1},
	\end{equation*}%
	 under the same hypothesis $r\equiv n\;(\bmod 2)$.
\end{thm}

\begin{figure}[!ht]
	\centering
	\begin{subfigure}[t]{0.45\textwidth}
		\centering
		\begin{equation*}
			\begin{tikzcd}[row sep=3, column sep=3]
				& & & E_1^{3,3} & & & \\
				& & 0 & &  0 \ar[ll] & & \\
				& 0 & & E_1^{2,2} \ar[ll] & & 0 \ar[ll] & \\
				E_1^{0,3} & & E_1^{1,2} \ar[ll] & & E_1^{2,1 \ar[ll]} & & E_1^{3,0} \ar[ll] \\
				& 0 & & E_1^{1,1} \ar[ll] & & 0 \ar[ll] & \\
				& & 0 & &  0 \ar[ll] & & \\
				& & & E_1^{0,0} & & & \\
			\end{tikzcd}
		\end{equation*}
		\caption{The first page.}
	\end{subfigure}
	\hfill
	\begin{subfigure}[t]{0.45\textwidth}
		\centering
		\begin{equation*}
			\begin{tikzcd}[row sep=3, column sep=3]
				& & & E_2^{3,3} & & & \\
				& & 0 & &  0 \ar[llld] & & \\
				& 0 & & E_2^{2,2} \ar[llld] & & 0 \ar[llld] & \\
				E_2^{0,3} & & E_2^{1,2} & & E_2^{2,1} \ar[llld] & & E_2^{3,0} \ar[llld] \\
				& 0 & & E_2^{1,1} & & 0 \ar[llld] & \\
				& & 0 & &  0 & & \\
				& & & E_2^{0,0} & & & \\
			\end{tikzcd}
		\end{equation*}
		\caption{The second page.}
	\end{subfigure}
	\hfill
	\begin{subfigure}[t]{0.45\textwidth}
		\centering
		\begin{equation*}
			\begin{tikzcd}[row sep=3, column sep=3]
				& & & E_3^{3,3} & & & \\
				& & 0 & &  0 \ar[lllldd] & & \\
				& 0 & & E_3^{2,2} & & 0 \ar[lllldd] & \\
				E_3^{0,3} & & E_3^{1,2} & & E_3^{2,1} & & E_3^{3,0} \ar[lllldd] \\
				& 0 & & E_3^{1,1} & & 0 & \\
				& & 0 & &  0 & & \\
				& & & E_3^{0,0} & & & \\
			\end{tikzcd}
		\end{equation*}
		\caption{The third page.}
	\end{subfigure}
	\caption{Some pages of the spectral sequence of a T-hypersurface of dimension 3.}
	\label{fig:spec_seq_RX}
\end{figure}

Following Theorem~\ref{cor:struct_spec_seq_X}, we will say that a page of the spectral sequences of $\X$ is \emph{irrelevant}\index{Irrelevant} if its differentials or boundary operators are all trivial because they have either a trivial source or a trivial target. The other ones, which we call the \emph{relevant}\index{Relevant} pages, are those for which $0\leq r\leq 1$, or $2\leq r\leq n-1$ and $r\equiv n\; (\bmod 2)$.

\begin{thm}[Poincar\'e duality]\label{thm:symmetry}
	For all $r\geq 1$, the $E_r$-pages of the Renaudineau--Shaw spectral sequence computing the cohomology of\, $\X$ satisfy Poincar\'e duality. That is to say:%
	\begin{enumerate}
		\item The vector space $E^{p,q}_{r}(\X)$ vanishes whenever $p>n-1$ or $q>n-1$.
		\item The vector space $E^{n-1,n-1}_{r}(\X)$ has dimension~$1$.
		\item The bilinear pairing $\cup\colon E^{p,q}_{r}(\X)\otimes E^{n-1-p,n-1-q}_{r}(\X)\rightarrow E^{n-1,n-1}_{r}(\X)$ is non-degenerate.
	\end{enumerate}%
	 In particular, $E^{p,q}_{r}(\X)$ is isomorphic to $E^{n-1-p,n-1-q}_{r}(\X)$, and $\df^{p,q}_r$ and $\df^{n-1-p+r,n-2-q}_r$ have the same rank and kernel dimension.
\end{thm}

\begin{proof}
	By construction, the first page of the spectral sequence $(E^{p,q}_{1}(\X)_{p,q\geq 0}$ is the cohomology of the sheaf of graded algebras associated with $\O_\X$. Following Proposition~\ref{prop:graded_sheaves}, this sheaf is isomorphic to the sheaf of graded algebras $\bigoplus_{k=1}^{n-1} F_X^k$. On that account, $E^{p,q}_1(\X)$ is isomorphic, for all $p,q\in\N$, to the tropical cohomology group $H^{p,q}(X;\F_2)$. Furthermore, this isomorphism respects the cup product. Using \cite[Theorem 5.3]{Jel-Rau-Sha_Lef_11} and \cite[Theorem 3.3]{Brug-LdM-Rau_Comb_pac}, we see that  this page satisfies the Poincar\'e duality, and the theorem follows from Lemma~\ref{lm:casc_poinca_dual}.    
\end{proof}

\section{Degeneracy and real Lefschetz property}

\begin{ntns}
	Let $M$ be a free Abelian group of rank $n$. Let $P$ be a smooth polytope of the vector space $M\otimes\R$, $K$ be a primitive triangulation of $P$ with dual hypersurface $X$, and $\varepsilon\in C^0(K;\F_2)$ be a sign distribution on $K$.
\end{ntns}

\begin{dfn}\label{dfn:Real_Lef_Prop}
	Let $H$ be a hypersurface of $\R P$. We say that $H$ has the \emph{real Lefschetz property} if the morphisms induced by the inclusion $i^q:H^q(\R P;\F_2)\longrightarrow H^q(H;\F_2)$ are injective for all $q\leq\left\lfloor \frac{n-1}{2}\right\rfloor$.
\end{dfn}

As announced in the introduction, we will link the degeneracy of the Renaudineau--Shaw spectral sequence of $\X$ to the real Lefschetz property.

\begin{thm}[Vanishing criterion]\label{thm:charact_vanishing}
	Let $r\geq 2$ be an integer congruent to $n$ modulo $2$. The differentials of the page $E_r(\X)$ vanish if and only if ${i^q\colon H^q(\R P;\F_2)\rightarrow H^q(\X;\F_2)}$ is injective when $q$ equals $\frac{n-r}{2}$. 
\end{thm}

\begin{proof}
	Let $r\geq 2$ and $r\equiv n\;(\bmod 2)$. By Theorem~\ref{cor:struct_spec_seq_X}, the only possibly non-trivial differentials of the $\supth{r}$ page are%
	\begin{equation*}
	 	\df_r^{\frac{n+r}{2},\frac{n-r}{2}-1}\quad\textnormal{and}\quad\df_r^{\frac{n+r}{2}-1,\frac{n+r}{2}-1}.
	\end{equation*}%
	 Following Theorem~\ref{thm:symmetry}, they have the same rank, so one vanishes if and only the other does. Let $q$ be $\frac{n-r}{2}$; we have the commutative diagram%
	\begin{equation}\label{D}
		\begin{tikzcd}
			E_r^{q+r,q-1}(\R P) \ar[d,"\df_r^{q+r,q-1}" left] \ar[r,"i^{q+r,q-1}_r" above] & E_r^{q+r,q-1}(\X) \ar[d,"\df_r^{q+r,q-1}" right]\\
			E_r^{q,q}(\R P) \ar[r,"i^{q,q}_r" below] & E_r^{q,q}(\X)\rlap{.} 
		\end{tikzcd}
	\end{equation}%
	 Using Theorems~\ref{cor:struct_spec_seq_P} and~\ref{cor:struct_spec_seq_X} and Proposition~\ref{prop:graded_sheaves}, the commutative diagram~(\ref{D}) can be written as follows:%
	\begin{equation*}
		\begin{tikzcd}
			0 \ar[d] \ar[r] & E_2^{q+r,q-1}(\X) \ar[d,"\df_r^{q+r,q-1}" right]\\
			H^{q,q}(P;\F_2) \ar[r,"i^{q,q}" below] & H^{q,q}(X;\F_2)\rlap{.} 
		\end{tikzcd}
	\end{equation*}%
	 Since $2q=n-r<n-1$ by assumption, the tropical Lefschetz hyperplane section theorem, see \cite[Theorem 1.1]{Arn-Ren-Sha_Lef_sec} and \cite[Proposition 3.2]{Bru-Man_she_dec}, implies that $i^{q,q}$ is an isomorphism. As a consequence, $\df_r^{q+r,q-1}$ vanishes if and only if the map ${i^{q,q}_{r+1}}\colon E_{r+1}^{q,q}(\R P) \rightarrow E_{r+1}^{q,q}(\X)$ is injective. Furthermore, we have the following exact sequence:%
	\begin{equation*}
	0 \longrightarrow E^{q,q}_{\infty}(\X) \longrightarrow H^q(\X;\F_2) \longrightarrow E^{n-1-q,q}_\infty(\X) \longrightarrow 0.
	\end{equation*}%
	 Since $E^{q,q}_{\infty}(\X)$ is isomorphic to $E^{q,q}_{r+1}(\X)$, $H^q(\R P;\F_2)$ is isomorphic to $E_{r+1}^{q,q}(\R P)$, and the morphism ${i^q\colon H^q(\R P;\F_2)\rightarrow H^q(\X;\F_2)}$ respects the filtration, the morphism $i^{q,q}_{r+1}$ is injective if and only if $i^q$ is. 
\end{proof}

From this vanishing criterion, we deduce the following corollary.

\begin{cor}[Degeneracy criterion]\label{cor:charact_degeneracy}
	Let $r\geq 2$ be an integer. The Renaudineau--Shaw spectral sequence of\, $\X$ degenerates at the $\supth{r}$ page if and only if the maps ${i^q\colon H^q(\R P;\F_2)\rightarrow H^q(\X;\F_2)}$ are injective for all ${q\leq \left\lfloor\frac{n-r}{2}\right\rfloor}$.
\end{cor}

Corollary~\ref{cor:charact_degeneracy} can be interpreted as a comparison of two invariants of the pair $\X\subset\R P$.
\begin{dfn}\label{dfn:dege_index}
	We define the \emph{degeneracy index}\index{Degeneracy Index} of $\X$ as%
	\begin{equation*}
		r(\X)\coloneqq \min\left\{r_0\geq 0\relmiddle| \df^{p,q}_r=0,\,\forall p,q\in\N,\,\forall r\geq r_0\right\}.
	\end{equation*}
\end{dfn}

The second invariant was introduced by I.\,O.~Kalinin for projective hypersurfaces and named the \emph{rank}\index{Rank} of the hypersurface by O.~Viro.

\begin{dfn}\label{dfn:rank_X}
	The \emph{rank} of $\X$ is defined as%
	\begin{equation*}
		\ell(\X)\coloneqq \max\left\{q_0\geq 0\mid i^q\colon H^q(\R P;\F_2)\rightarrow H^q(\X;\F_2) \textnormal{ is injective for all }q\leq q_0\right\}.
	\end{equation*}%
\end{dfn}

\begin{prop}\label{prop:inj_r_1}
	If the dimension $n$ of $\X$ is odd and $r(\X)=1$, the map induced by the inclusion $i^\frac{n-1}{2}\colon H^{\frac{n-1}{2}}(\R P;\F_2) \rightarrow H^{\frac{n-1}{2}}(\X;\F_2)$ is injective.
\end{prop}

\begin{proof}
	The hypersurface $\X$ is maximal for the Smith--Thom inequality; \textit{cf.} \cite[Theorem~6.1]{Ren-Sha_bou_bet}. In the first page of the Renaudineau--Shaw spectral sequences, we have the following commutative diagram:%
	\begin{equation*}
		\begin{tikzcd}[column sep=5em]
			E^{\frac{n+1}{2},\frac{n-3}{2}}_1(\R P) \ar[r,"i^{\frac{n+1}{2},\frac{n-3}{2}}_1" above] \ar[d,"\df^{\frac{n+1}{2},\frac{n-3}{2}}_1" left] & E^{\frac{n+1}{2},\frac{n-3}{2}}_1(\X) \ar[d,"\df^{\frac{n+1}{2},\frac{n-3}{2}}_1" right] \\
			E^{\frac{n-1}{2},\frac{n-1}{2}}_1(\R P) \ar[r,"i^{\frac{n-1}{2},\frac{n-1}{2}}_1" below] & E^{\frac{n-1}{2},\frac{n-1}{2}}_1(\X)\rlap{.}
		\end{tikzcd}
	\end{equation*}%
	 Using the maximality hypothesis as well as  Theorems~\ref{cor:struct_spec_seq_P} and~\ref{cor:struct_spec_seq_X} and Proposition~\ref{prop:graded_sheaves}, this commutative diagram can be written as follows:%
	\begin{equation*}
		\begin{tikzcd}[column sep=5em]
			0 \ar[d] \ar[r] & H^{\frac{n+1}{2},\frac{n-3}{2}}(X;\F_2) \ar[d,"0" right]\\
			H^{\frac{n-1}{2},\frac{n-1}{2}}(P;\F_2) \ar[r,"i^{\frac{n-1}{2},\frac{n-1}{2}}" below] & H^{\frac{n-1}{2},\frac{n-1}{2}}(X;\F_2)\rlap{,}
		\end{tikzcd}
	\end{equation*}%
where $i^{\frac{n-1}{2},\frac{n-1}{2}}$ is injective by the tropical Lefschetz hyperplane section theorem, \cite[Theorem 1.1]{Arn-Ren-Sha_Lef_sec} and \cite[Proposition 3.2]{Brug-LdM-Rau_Comb_pac}. Since $\X$ is maximal, $H^{\frac{n-1}{2}}(\X;\F_2)$ is isomorphic to $H^{\frac{n-1}{2},\frac{n-1}{2}}(X;\F_2)$, and the morphism
	\begin{equation*}
		i^\frac{n-1}{2}\colon H^\frac{n-1}{2}(\R P;\F_2)\longrightarrow  H^\frac{n-1}{2}(\X;\F_2)
	\end{equation*}
        is conjugate to the morphism
	 \begin{equation*}
	 	i^{\frac{n-1}{2},\frac{n-1}{2}}\colon H^{\frac{n-1}{2},\frac{n-1}{2}}(P;\F_2) \longrightarrow H^{\frac{n-1}{2},\frac{n-1}{2}}(X;\F_2).
	\end{equation*}
As a consequence, $i^{\frac{n-1}{2}}$ is also injective.
\end{proof}

\begin{cor}
	If\, $\X$ is maximal relatively to the Smith--Thom inequality, then it satisfies the real Lefschetz property.
\end{cor}

\begin{proof}
  The T-hypersurface $\X$ is maximal if and only if $r(\X)=1$. In this case, the real Lefschetz property is ensured by Corollary~\ref{cor:charact_degeneracy} and Proposition~\ref{prop:inj_r_1}.
\end{proof}

\begin{rem}
This was already known for algebraic hypersurfaces of projective spaces; \textit{cf.} \cite[Remarks on the Homomorphisms $j_*$]{Kha_add_cong}.
\end{rem}

\begin{cor}\label{cor:inequalities}
	We have the inequalities%
	\begin{equation*}
		\ell(\X) \geq \left\lfloor \frac{n-r(\X)}{2}\right\rfloor,
	\end{equation*}%
	 with equality if $r(\X)\geq 3+\frac{1-(-1)^n}{2}$, and%
	\begin{equation*}
		r(\X) \leq \max\left(2;n-2\ell(\X)-1\right),
	\end{equation*}%
	 with equality if $\ell(\X)\leq\frac{n-5}{2}$.
\end{cor}

\begin{proof}
	We begin by proving the first inequality. 
	Let us denote by $R$ and $L$ the sets%
	\begin{equation*}
		\left\{\begin{aligned}
			R&\coloneqq \left\{2\leq r_0\leq n\relmiddle| \df^{p,q}_r=0,\,\forall p,q\in\N,\,\forall r\geq r_0\right\},\\
			L&\coloneqq \left\{q_0\geq 0\relmiddle| i^q\colon H^q(\R P;\F_2)\rightarrow H^q(\X;\F_2) \textnormal{ is injective for all }q\leq q_0\right\}.
		\end{aligned}\right.
	\end{equation*}%
	We denote the map $r\mapsto \left\lfloor \frac{n-r}{2}\right\rfloor$ by $f$. Corollary~\ref{cor:charact_degeneracy} and Proposition~\ref{prop:inj_r_1} imply that $f(R)$ is a subset of $L$. Since $f$ is non-increasing, $\ell(\X)=\max (L)$ is at least equal to $f(\min R)=f(r(\X))$, and we find the first inequality. If $r(\X)\geq 3+\frac{1-(-1)^n}{2}$, then the differentials of at least one relevant page are non-trivial. Then, using Corollary~\ref{cor:charact_degeneracy}, we find that the map $i^q$ is not injective for the corresponding index $q$, and the equality follows. 
	
	To prove the second inequality, we use the converse implication of Corollary~\ref{cor:charact_degeneracy} to find that%
	\begin{equation*}
		\begin{split}
			r(\X)& \textstyle \leq \min\left\{r\geq 2\relmiddle| i^q \textnormal{ is injective for all }q\leq \left\lfloor \frac{n-r}{2} \right\rfloor\right\}\\
			&\textstyle \leq \min\left\{r\geq 2\relmiddle| \left\lfloor \frac{n-r}{2}\right\rfloor \leq \ell(\X)\right\}\\
			&\textstyle \leq \min\left(\left\{r\geq 2\relmiddle| \frac{n-r}{2} \leq \ell(\X)\right\} \cup\left\{r\geq 2\mid  \left\lfloor \frac{n-r}{2}\right\rfloor \leq \ell(\X)<\frac{n-r}{2}\right\}\right) \\
			&\textstyle \leq \min \left(\{\max(2;n-2\ell(\X))\} \cup \left\{r\geq 2\relmiddle|  \left\lfloor \frac{n-r}{2}\right\rfloor \leq \ell(\X)<\frac{n-r}{2}\right\}\right).
		\end{split}
	\end{equation*}
	 The set $\left\{r\geq 2\relmiddle|  \left\lfloor \frac{n-r}{2}\right\rfloor \leq \ell(\X)<\frac{n-r}{2}\right\}$ is either empty or reduced to ${\{n-2\ell(\X)-1\}}$. The latter case only occurs when $n-2\ell(\X)-1\geq 2$,  \textit{i.e.}~$\ell(\X)\leq\frac{n-3}{2}$. In that case, we find $r(\X)\leq n-2\ell(\X)-1$. If on the contrary $\ell(\X)\geq \frac{n-1}{2}$, then $\max(2;n-2\ell(\X))=2$ and $r(\X)\leq 2$. Therefore, $r(\X)$ is at most equal to $\max(2;n-2\ell(\X)-1)$. The associated equality follows from Corollary~\ref{cor:charact_degeneracy} since under the assumption $\ell(\X)\leq\frac{n-5}{2}$ there is at least one relevant page of index at least 2 with non-trivial differentials.
\end{proof} 

	Corollary~\ref{cor:charact_degeneracy} can be improved for a particular class of real toric varieties whose cohomology ring has a property related to the hard Lefschetz theorem; \textit{cf.} \cite[Corollary 4.13]{Wel_dif_ana}.

\begin{dfn}\label{dfn:iota}
	Let $Y$ be a topological space and $\alpha\in H^1(Y;\F_2)$. We denote by $\iota(\alpha)$ the minimum of the integers $q\geq -1$ for which there exists a non-zero class $\beta\in H^{q+1}(Y;\F_2)$ for which $\alpha\cup \beta$ vanishes. We denote the maximum of the $\iota(\alpha)$ for all $\alpha\in H^1(Y;\F_2)$ by $\iota(Y)$.
\end{dfn}

\begin{rem}
	For any $n$-dimensional smooth polytope $P$, the number $\iota(\R P)$ is at most $n-1$. The hard Lefschetz theorem might lead us to think that $\iota(\R P)$ is always at least equal to $\left\lfloor\frac{n}{2}\right\rfloor-1$ since the cohomology ring of $\R P$ with $\F_2$-coefficients is isomorphic, up to dividing the grading by 2, to the reduction modulo $2$ of the integral cohomology ring of its complex locus $\C P$. However, the powers of the class of an ample line bundle might be divisible by $2$ from a certain rank. This phenomenon is illustrated in point~\eqref{exs:iota(P)-4} of Examples~\ref{exs:iota(P)}. Nevertheless, we can deduce from the hard Lefschetz theorem that $\iota(\R P)$ being greater than $\left\lfloor\frac{n}{2}\right\rfloor-1$ is a constraint on the polytope $P$. Since the Betti numbers of $\R P$ form an unimodal sequence centered at $\frac{n}{2}$, if $\iota(\R P)\geq\left\lfloor\frac{n}{2}\right\rfloor$, then the $\supth{q}$ Betti numbers of $\R P$ must be equal for all $n-\iota(\R P)-1\leq q\leq \iota(\R P)+1$, as in point~\eqref{exs:iota(P)-3} of Examples~\ref{exs:iota(P)}. The only polytope $P$ satisfying $\iota(\R P)=n-1$ is the simplex.
\end{rem}

\begin{exs}\label{exs:iota(P)}
	Here we give examples of the numbers $\iota(Y)$ when $Y$ is the real locus of a smooth projective toric variety. 
	\begin{enumerate}
		\item\label{exs:iota(P)-1} If $P$ is a non-singular simplex, its associated toric variety is a projective space. The cohomology ring of $\R P$ is isomorphic to $\smallslant{\F_2[h]}{(h^{n+1})}$, where $n$ is the dimension of the simplex. We have $\iota(\R P)=n-1$.
		\item\label{exs:iota(P)-2} If $P$ is the product of two non-singular simplices, then the cohomology ring associated with $\R P$ is $\smallslant{\F_2[h_1,h_2]}{(h_1^{n_1+1},h_2^{n_2+1})}$. If $n_1\leq n_2$, any class of the form $h_2+uh_1$, $u\in\F_2$, has $\iota(h_2+uh_1)=n_2-1$ and the other non-zero classes have $\iota(h_1)=n_1-1$. Therefore, $\iota(\R P)=\max(n_1;n_2)-1$. 
		\item\label{exs:iota(P)-3} We say that a simple integer polytope $P$ is a \emph{blow-up}\index{Blow-up} of a simple integer polytope $Q$ if $P$ is obtained by chopping off one of the corners of $Q$. More precisely, the corner to be chopped off must be adjacent to edges of integral length at least $2$, and the chopping hyperplane must pass through integer points of these edges at equal integer distances from the corner. See Figure~\ref{fig:blow-up} for an example. The corner vertex of the polytope corresponds to a fixed point of the torus action on the associated toric variety. Chopping off the corner corresponds to blowing up the fixed point. One can find a description of the normal fan in \cite[Section 2.5]{Ful_tor_var}. The space $\R P$ is the connected sum of $\R Q$ with $\R\P^n$. Using the Mayer--Vietoris exact sequence, \textit{cf.} \cite[Theorem 33.1]{Munk_ele_alg}, we find that $\bigoplus_{q\in\N}H^q(\R P;\F_2)$ is the quotient of $(\bigoplus_{q\in\N}H^q(\R Q;\F_2))[x]$ by the ideal $(x\alpha \colon \alpha\in \bigoplus_{q\geq 1}H^q(\R Q;\F_2))+(x^n+[\R Q]^*)$, where $[\R Q]^*$ is the generator of $H^n(\R Q;\F_2)$. Therefore, if $\alpha\in H^1(\R Q;\F_2)$, then the class $\alpha+x\in H^1(\R P;\F_2)$ satisfies $\iota(\alpha+x)=\min(\iota(\alpha);n-2)$. However, the class $\alpha$, seen as a class of $P$, has $\iota$-number equal to either $0$ or $-1$. Since $x$ has $\iota$-number equal to~$0$, we find that $\iota(\R P)=\min(\iota(\R Q);n-2)$.
		
		\begin{figure}[!ht]
	\centering
	\begin{tikzpicture}[scale=1]
		\draw (0,0) rectangle (2,2);
		\fill (0,0) circle (.05);
		\fill (1,0) circle (.05);
		\fill (2,0) circle (.05);
		\fill (0,1) circle (.05);
		\fill (1,1) circle (.05);
		\fill (2,1) circle (.05);
		\fill (0,2) circle (.05);
		\fill (1,2) circle (.05);
		\fill (2,2) circle (.05);
		\draw[dotted, very thick] ($(0,1)+(-135:.5)$) -- ($(1,2)+(+45:.5)$);
		
		\draw[thick,->] (2.25,1) -- (3.75,1);
		
		\draw (4,0) -- (6,0) -- (6,2) -- (5,2) -- (4,1) -- cycle;
		\fill (4,0) circle (.05);
		\fill (5,0) circle (.05);
		\fill (6,0) circle (.05);
		\fill (4,1) circle (.05);
		\fill (5,1) circle (.05);
		\fill (6,1) circle (.05);
		\fill (5,2) circle (.05);
		\fill (6,2) circle (.05);
	\end{tikzpicture}
	\caption{A blow-up of $\mathbb{P}^1\times\mathbb{P}^1$.}
	\label{fig:blow-up}
\end{figure}
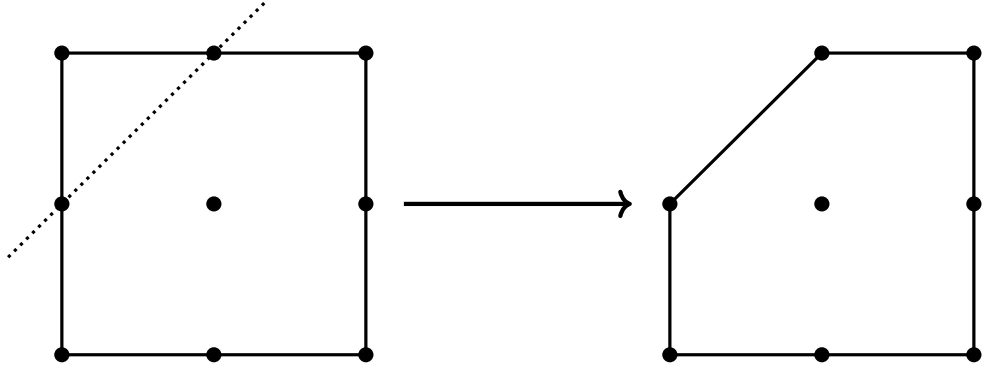
		
		\item\label{exs:iota(P)-4} Let $I^n$ be the $n$-dimensional cube. The cohomology ring $\bigoplus_{q\in\N}H^q(\R I^n;\F_2)$ is the exterior algebra $\bigwedge H^1(\R I^n;\F_2)$, so any class of degree $1$ squares to $0$. Therefore, $\iota(\R I^n)=0$.
	\end{enumerate}
\end{exs}

\begin{cor}\label{cor:charact_degeneracy2}
	Let $r_0\geq 2$ be an integer congruent to $n$ modulo $2$. If\, $\iota(\R P)$ is at least $\left\lfloor\frac{n}{2}\right\rfloor-2$, then the differentials $\df^r_{p,q}$ of the Renaudineau--Shaw spectral sequence computing the cohomology of\, $\X$ vanish for all ${p,q\in\N}$ and all ${r\geq r_0}$ if and only if the induced map ${i^q\colon H^q(\R P;\F_2)\rightarrow H^q(\X;\F_2)}$ is injective for ${q=\frac{n-r_0}{2}}$. In particular, all the differentials $\df^r_{p,q}$ are trivial for all $r\geq 2$ if and only if the induced map ${i^q\colon H^q(\R P;\F_2)\rightarrow H^q(\X;\F_2)}$ is injective for $q= \left\lfloor\frac{n}{2}\right\rfloor-1$.  
\end{cor}

\begin{proof}
	Let $\omega$ be a class of $P$ for which $\iota(\omega)\geq\left\lfloor\frac{n}{2}\right\rfloor-2$. For all $q\in\N$, we have the following commutative diagram:%
	\begin{equation*}
		\begin{tikzcd}
			H^q(\R P;\F_2) \ar[r,"i^q" above] \ar[d,"\omega\cup-" left] & H^q(\X;\F_2) \ar[d,"i^1(\omega)\cup-" right] \\
			H^{q+1}(\R P;\F_2) \ar[r,"i^{q+1}" below] & H^{q+1}(\X;\F_2)\rlap{.}
		\end{tikzcd}
	\end{equation*}%
	 Therefore, if $q+1\leq\left\lfloor \frac{n}{2}\right\rfloor -1$ and $i^{q+1}$ is injective, then so is $i^q$ since $\omega\cup-$ is injective. Corollary~\ref{cor:charact_degeneracy2} follows from Corollary~\ref{cor:charact_degeneracy}.
\end{proof}

If $P$ is a non-singular simplex of odd size and $K$ is a convex triangulation, $\X$ is isotopic to a real algebraic hypersurface of odd degree in a real projective space by Viro's patchworking theorem, \cite[Theorem 4.3.A]{Vir_pat_rea} or \cite[Th\'eor\`eme 4.2]{Ris_con_hyp}. In this case, it is known that $\ell(\X)=n-1$; hence all the differentials of\, $(E_r^{p,q}(\X))_{p,q\geq 0}$ vanish for all $r\geq 2$. This sufficient criterion of degeneracy can be generalised to the class of polytopes satisfying $\iota(\R P)\geq \left\lfloor\frac{n}{2}\right\rfloor-1$ regardless of the convexity of the triangulation.

\begin{prop}\label{prop:ell_iota}
	We have the inequality%
	\begin{equation*}
		\ell(\X)\geq \iota[\omega_{\R X}].
	\end{equation*}%
	 In particular, it follows from Corollary~\ref{cor:charact_degeneracy2} that if $\iota[\omega_{\R X}]\geq \left\lfloor\frac{n}{2}\right\rfloor-1$, then all the differentials of $(E_r^{p,q}(\X))_{p,q\geq 0}$ vanish for all $r\geq 2$.
\end{prop}

	Before proving the proposition, we can note that, for all choices of sign distributions, the images of the fundamental classes of the T-hypersurfaces $\X$ yield the same homology class of $\R P$. This class is Poincar\'e dual to $[\omega_{\R X}]$; \textit{cf.} Remark~\ref{rem:poinc_dual_X}.

\begin{proof}
	For all $0\leq q\leq n-1$, we have the following commutative diagram:%
	\begin{center}
	\begin{tikzpicture}[commutative diagrams/every diagram]
		\node (P0) at (90:2.3cm) {$H^q(\R P;\F_2)$ };
  		\node (P1) at (90+72:2cm) {$H_{n-q}(\R P;\F_2)$} ;
  		\node (P2) at (90+2*72:2cm) {\makebox[5ex][r]{$H_{n-1-q}(\R P;\F_2)$}};
  		\node (P3) at (90+3*72:2cm) {\makebox[5ex][l]{$H_{n-1-q}(\X;\F_2)$\rlap{.}}};
  		\node (P4) at (90+4*72:2cm) {$H^q(\X;\F_2)$};
  		\path[commutative diagrams/.cd, every arrow, every label]
   			(P0) edge node[swap] {$-\cap[\R P]$} (P1)
    		(P1) edge node[swap] {$[\omega_{\R X}]\cap-$} (P2)
    		(P3) edge node {$i_{n-1-q}$} (P2)
    		(P4) edge node {$-\cap[\X]$} (P3)
    		(P0) edge node {$i^q$} (P4);
	\end{tikzpicture}
	\end{center}
	 In the diagram, $\cap$ denotes the cap product, $[\R P]$ is the generator of $H_n(\R P;\F_2)$, and $[\X]$ is the generator of $H_{n-1}(\X;\F_2)$. We note that, by construction, the image of $[\X]$ in the homology of $\R P$ does not depend on $\varepsilon$: it is Poincar\'e dual to $[\omega_{\R X}]$; see Remark~\ref{rem:poinc_dual_X}. Since both $\R P$ and $\X$ satisfy  Poincar\'e duality with $\F_2$-coefficients, both $-\cap [\R P]$ and $-\cap [\X]$ are isomorphisms. We also have the following commutative diagram:
	\begin{equation*}
		\begin{tikzcd}[column sep=4em]
			H^q(\R P;\F_2) \ar[r,"{[\omega_{\R X}]}\cup -" above] \ar[d,"-\cap{[\R P]}" left,"\cong" right] & H^{q+1}(\R P;\F_2) \ar[d,"-\cap{[\R P]}" right,"\cong" left] \\%
			H_{n-q}(\R P;\F_2) \ar[r,"{[\omega_{\R X}]}\cap-" below] & H_{n-q-1}(\R P;\F_2)\rlap{.}%
		\end{tikzcd}
	\end{equation*}
	Therefore, the morphism $i^q$ is injective for all $q\leq\iota[\omega_{\R X}]$, and $\ell(\X)\geq \iota[\omega_{\R X}]$.
\end{proof} 

\begin{cor}\label{cor:odd_deg}
	The Renaudineau--Shaw spectral sequence of a T-hypersurface of odd degree in a projective space degenerates at the second page.
\end{cor}

\begingroup
\renewcommand{\arraystretch}{1.3}
\begin{table}[!ht]
	\begin{subtable}[t]{.45\textwidth}
		\centering
		\begin{tabular}{|c|c|c|}
			\hline
			$\P^n_{d}\times \P^{n+m}_{d'}$ & $d'$ even & $d'$ odd \\ \hline
			$d$ even & $0$ & $n+m-1$ \\ \hline
			$d$ odd & $n-1$ & $n+m-1$ \\ \hline 
		\end{tabular}
		\caption{The values of $\iota[\omega_{\R X}]$ in the product of an {$n$-projective} space with an $(n+m)$-projective space as a function of the bidegree $(d;d')$.}\label{sub-a}
	\end{subtable}
	\hfill
	\begin{subtable}[t]{.45\textwidth}
	  \centering
		\begin{tabular}{|c|c|c|}
			\hline
			$\P^n_{d}\times \P^{n+m}_{d'}$ & $d'$ even & $d'$ odd \\ \hline
			$d$ even & $\max(2;n+m-1)$ & $2$ \\ \hline
			$d$ odd & $\max(2;m+1)$ & $2$ \\ \hline 
		\end{tabular}
		\caption{The corresponding upper bound on $r(\X)$ deduced from Corollary~\ref{cor:inequalities} and Proposition~\ref{prop:ell_iota}.}
                \label{sub-b}
	\end{subtable}

\bigskip 
        
	\begin{subtable}[t]{.45\linewidth}
		\centering
		\begin{tabular}{|c|c|c|}
			\hline
			$\textnormal{Bl.}\P^n_{d}$ & $[\X]\cdot E=0$ & $[\X]\cdot E=E^2$ \\ \hline
			$d$ even & $0$ & $0$ \\ \hline
			$d$ odd & $0$ & $n-2$ \\ \hline 
		\end{tabular}
		\caption{The values of $\iota[\omega_{\R X}]$ in the blow-up of an {$n$-projective} space as a function of the parity of the degree $d$ and the intersection with the exceptional divisor $E$. \\}\label{sub-c}
	\end{subtable}
	\hfill
	\begin{subtable}[t]{.45\linewidth}
		\centering
		\begin{tabular}{|c|c|c|}
			\hline
			$\textnormal{Bl.}\P^n_{d}$ & $[\X]\cdot E=0$ & $[\X]\cdot E=E^2$ \\ \hline	$d$ even & $\max(2;n-1)$ & $\max(2;n-1)$ \\ \hline
			$d$ odd & $\max(2;n-1)$ & $2$ \\ \hline 
		\end{tabular}
		\caption{The corresponding upper bound on $r(\X)$ deduced from Corollary~\ref{cor:inequalities} and Proposition~\ref{prop:ell_iota}.}\label{sub-d}
	\end{subtable}
	\caption{The value of $\iota[\omega_{\R X}]$ as well as the corresponding upper bound on the degeneracy index in a product of projective spaces ((\subref{sub-a}) and (\subref{sub-b})) and in the blow-up of a projective space ((\subref{sub-c}) and (\subref{sub-d})).}
	\label{tab:BlP}
\end{table}
\endgroup

Combining Corollary~\ref{cor:inequalities} and Proposition~\ref{prop:ell_iota} yields a generalisation in some cases of the statement ``The Renaudineau--Shaw spectral sequence of a projective primitively patchworked hypersurface of odd degree degenerates at the second page'', namely the upper bound%
\begin{equation*}
	r(\X) \leq \max\left(2;n-2\iota[\omega_{\R X}]-1\right).
\end{equation*}%
 Therefore, we can sometimes directly know if the Renaudineau--Shaw spectral sequence degenerates at the second page by only looking at the degree of $\X$, \textit{i.e.}~at the class $[\omega_{\R X}]$. In Table~\ref{tab:BlP}  we describe the  possible values of $\iota[\omega_{\R X}]$ and the associated upper bound on the degeneracy index. In (\subref{sub-a}) and (\subref{sub-b}), we study the case of a product of non-singular simplices and in (\subref{sub-c}) and (\subref{sub-d}) the case of a blow-up of a non-singular simplex. Both these examples have an $H^1$ of dimension 2. However, for some polytopes, \textit{e.g.}~a cube of dimension at least $4$, this upper bound is completely vacuous whatever the degree of the hypersurface may be.

\section{Degeneracy for a family of triangulations}\label{sec7}

In this section we will construct a particular sequence of primitive triangulations of the non-singular simplices inspired from the triangulations on which I.~Itenberg constructs maximal surfaces of every degree; \textit{cf.} \cite[Section 5]{Ite_top_rea}. We will then show that any T-hypersurface constructed on these triangulations has degeneracy index at most equal to 2. To do so we will prove that every T-hypersurface constructed on these triangulations satisfies the real Lefschetz property. We begin with some technical propositions.

\vspace{5pt}

We recall that if $K$ is an abstract simplicial complex, its geometric realisation $|K|$ is a subspace of the real vector space $V$ freely generated by the vertices of $K$. It is given by the union of all the geometric simplices $|\sigma|:=\textnormal{Conv.Hull.}\{p:p\in \sigma\}$ for all abstract simplices $\sigma$ of $K$. These geometric simplices induce a canonical regular cellular structure on $|K|$.

\begin{prop}\label{prop:isotopy_graph}
	Let $K$ be a finite simplicial complex, $\Sigma K$ be its suspension, and ${\varepsilon\in C^0(\Sigma K;\F_2)}$. We denote by $a_+$ and $a_-$ the two suspension apexes. If ${\varepsilon(a_+)+\varepsilon(a_-)}$ equals $1$, then the hypersurface $X$, dual to $\df\varepsilon$, is isotopic to $|K|$ relatively to $|K|\cap X$.
\end{prop}

\begin{proof}
	Let $V$ denote the real vector spaces freely generated by the vertices of $\Sigma K$. The hypersurface $X$ is the support of the subcomplex of the cubical subdivision of $\Sigma K$ made of the cubes indexed by the pairs of simplices $\sigma^1\leq \sigma^n$ for which $\df\varepsilon(\sigma^1)=1$. To prove the proposition, we will start by constructing a self-homeomorphism $G$ of $|\Sigma K|$ for which $G(X)$ has a simpler description. 
	Let $\sigma\in\Sigma K$. If the restriction $\varepsilon|_\sigma$ is non-constant, we can find a set of positive real numbers $u(\sigma)$ indexed by the vertices of $\sigma$ whose sum equals $1$ and that satisfies%
			\begin{equation*}
				\sum_{p\in\sigma^{(0)}}(-1)^{\varepsilon(p)}u_p(\sigma)=0.
			\end{equation*}%
			 If the restriction $\varepsilon|_\sigma$ is constant, we set $u_p(\sigma)=1/(\dim\sigma +1)$ for all vertices $p$ of $\sigma$. We endow $|\Sigma K|$ with two regular cellular structures that subdivide its canonical cellular structure. Let $\Sd\Sigma K$ denote the barycentric subdivision of $K$.
	\begin{enumerate}
		\item The first subdivision $E_0$ is the pushforward of the canonical cellular structure of $\Sd\Sigma K$ by the piecewise-affine homeomorphism $v\colon |\Sd\Sigma K|\rightarrow |\Sigma K|$ that sends the vertex of $\Sd\Sigma K$ indexed by $\sigma\in K$ to the barycenter of $|\sigma|$.
		\item The second subdivision $E_1$ is the pushforward of the canonical cellular structure of $\Sd\Sigma K$ by the piecewise-affine homeomorphism $u\colon |\Sd\Sigma K|\rightarrow |\Sigma K|$ that sends the vertex of $\Sd\Sigma K$ indexed by $\sigma\in K$ to $\sum_{p\in\sigma}u_p(\sigma)p$.
	\end{enumerate}
	 The self-homeomorphism $G:=u\circ v^{-1}$ of $|\Sigma K|$ is then a cellular isomorphism from $E_0$ to $E_1$. The hypersurface $G(X)$ is the intersection of $|\Sigma K|$ with the hyperplane of $V$ given by the equation%
	\begin{equation*}
		f\coloneqq \sum_{p\in \Sigma K^{(0)}} (-1)^{\varepsilon(p)}x_p=0,
	\end{equation*}
	where $x_p$ denotes the coordinate associated with the vertex $p$. We can note that $G(|K|)=|K|$ as $G$ is cellular relatively to two subdivisions of the canonical cellular structure of $\Sigma K$, and $K$ is a subcomplex of $\Sigma K$. Thus, it suffices to prove the proposition for $G(X)=\{f=0\}$ instead of $X$. The hypersurface $\{f=0\}$ avoids the two apexes $a_+$ and $a_-$, so $|\Sigma K|\cap \{f=0\}$ equals $\left(|\Sigma K|\setminus\{a_\pm\}\right)\cap \{f=0\}$. Let us consider the following homeomorphism:%
	\begin{equation*}
		\begin{array}{rcl}
			\varphi\colon |K|\times]-1;1[ & \longrightarrow & {|\Sigma K| \setminus \{a_+;a_-\}} \\%
			(p;t) & \longmapsto & (1-|t|)p + \max(0;t)a_+ - \min(0;t)a_-.
		\end{array}
	\end{equation*}%
	We will prove that the inverse image of $\{f=0\}$ by $\varphi$ is the graph of a function $g\colon |K|\rightarrow ]-1;1[$. Since $\varepsilon(a_+)+\varepsilon(a_-)$ equals 1, $f(\varphi(p;t))$ is equal to $(1-|t|)f(p)+(-1)^{\varepsilon(a_+)}t$. If $f(\varphi(p;t))$ vanishes and $f(p)$ does not, $|t|$ is positive. In this case, the sign of $t$ is given by%
	\begin{equation*}
	\frac{t}{|t|}=\frac{(-1)^{\varepsilon(a_+)+1}f(p)}{|f(p)|}.
	\end{equation*}%
	By noting that $|t|=\frac{t}{|t|}t$, we deduce that
	
	\begin{equation*}
		\left\{ \begin{aligned}
			f(\varphi(p;t))&=0, \\
			f(p)&\neq 0
			\end{aligned}\right. \quad \textnormal{if and only if} \quad \left\{ \begin{aligned}
			f(p) +(-1)^{\varepsilon(a_+)}(1+|f(p)|)t&=0, \\
			f(p)&\neq 0.
			\end{aligned}\right.
	\end{equation*}
	If both $f(\varphi(p;t))$ and $f(p)$ vanish, then $t$ equals $0$. Thus, we find that $f(\varphi(p;t))$ vanishes if and only if
	\begin{equation*}		
		{f(p) +(-1)^{\varepsilon(a_+)}(1+|f(p)|)t=0}.
	\end{equation*}%
	 Therefore, the hypersurface $\{f=0\}$ is the graph of the function ${g\colon p\mapsto (-1)^{\varepsilon(a_+)+1}f(p)/(1+|f(p)|)}$ in the coordinates given by $\varphi$. We consider the map $H\colon [0;1]\times |K| \rightarrow |\Sigma K|$ defined for all $(t;p)\in[0;1]\times |K|$ by the formula%
	\begin{equation*}
		H(t;p)\coloneqq \varphi\left(p;tg(p)\right).
	\end{equation*}%
	For all $0\leq t\leq 1$, the map $H(t;-)$ is the embedding of the graph of $tg$ in the coordinates given by $\varphi$. the embedding $H(0;-)$ is the inclusion of $|K|$ in $|\Sigma K|$. The set $H(1;|K|)$ is the set $\{f=0\}$. If $p$ belongs to $|K|\cap X$, then $g(p)$ vanishes and $H(t;p)$ equals $p$ for all $t$.
\end{proof}

	We denote by $B^{n+1}$ the $(n+1)$-dimensional Euclidean ball and by $\pi$ the projection of $B^{n+1}$ onto $\R\P^{n+1}$ by antipodal identification of the points of $\partial B^{n+1}$. In this quotient representation of $\R\P^{n+1}$, we see $\R\P^n\subset \R\P^{n+1}$ as the image of $\partial B^{n+1}$ by $\pi$. Let $X$ be a subset of $\R\P^n$; we define $CX\subset \R\P^{n+1}$ as the image under $\pi$ of the cone at the origin $0$ over $\pi^{-1}(X)$. 

\begin{lem}\label{lem:Cone_over_projective}
	If\, $X\subset \R\P^n$ is homeomorphic to $\R\P^p$ and homologous to a linear subspace, then $CX$ is homeomorphic to $\R\P^{p+1}$ and homologous to a linear subspace as well.
\end{lem}

\begin{proof}
	The space $\pi^{-1}(X)$ is a double cover of $X$. If it is connected and $p>1$, then $\pi^{-1}(X)\rightarrow X$ is the universal cover of $X$ and $\pi^{-1}(X)$ is a sphere. If $p=1$ and it is connected, $\pi^{-1}(X)$ is a circle. In these cases, $CX$ is homeomorphic to $\R\P^{p+1}$. Now we prove that if $X$ is homologous to a linear subspace of $\R\P^n$, then $\pi^{-1}(X)$ is connected. We have the commutative diagram of exact sequences of singular chain complexes%
	\begin{equation*}
		\begin{tikzcd}
			0 & \ar[l] C_*(\R\P^n;\F_2) & \ar[l,"\pi_*" above] C_*(\Sph^n;\F_2) & \ar[l,"i" above] C_*(\R\P^n;\F_2) & \ar[l] 0 \\
			0 & \ar[l] \ar[u,"j_*" right] C_*(X;\F_2) & \ar[l,"\pi_*" below] \ar[u] C_*(\pi^{-1}(X);\F_2) & \ar[l, "i" below] \ar[u,"j_*" right] C_*(X;\F_2) & \ar[l] 0\rlap{,}
		\end{tikzcd}
	\end{equation*}%
	 where the vertical morphism are given by the corresponding inclusions of topological spaces and the morphism $i$ sends a singular simplex $\sigma$ to the sum of its two lifts to $\Sph^n$. It yields the following commutative diagram of exact sequences%
	\begin{equation*}
		\begin{tikzcd}
			0 & \ar[l] H_0(\R\P^n;\F_2) & \ar[l,"\cong" above] H_0(\Sph^n;\F_2) & \ar[l,"0" above] H_0(\R\P^n;\F_2) & \ar[l,"\cong" above] H_1(\R\P^n;\F_2) & \ar[l] \cdots \\
			0 & \ar[l] \ar[u,"\cong" right] H_0(X;\F_2) & \ar[l,"\pi_*" below] \ar[u] H_0(\pi^{-1}(X);\F_2) & \ar[l, "i" below] \ar[u,"\cong" right] H_0(X;\F_2) & \ar[l] \ar[u,"j_*" right] H_1(X;\F_2) & \ar[l] \cdots\rlap{.}
		\end{tikzcd}
	\end{equation*}%
	This implies in particular that $b_0(\pi^{-1}(X);\F_2)=1+\rk(i)=2-\rk(j_*)$. By assumption the image of the fundamental class of $X$ in the homology of $\R\P^n$ is $[\R\P^p]$, so its intersection with $[\R\P^{n+1-p}]$ is the class of the line $[\R\P^1]$. If $\alpha$ is the Poincar\'e dual of $[\R\P^{n+1-p}]$, we have
        \begin{equation*}
	j_*(j^*\alpha\cap[X])=[\R\P^{n+1-p}]\cdot[\R\P^p]=[\R\P^1],
\end{equation*} so $j_*$ has rank 1 and $\pi^{-1}(X)$ is connected. 
	
	\vspace{5pt}
	
	To see that $CX$ is still homologous to a linear subspace, we can notice that we have the following commutative square:
	\begin{equation*}
		\begin{tikzcd}
		H^1(\R\P^{n+1};\F_2)\ar[d] \ar[r,"\cong" above] & H^1(\R\P^n;\F_2) \ar[d,"\cong" right]\\
		H^1(CX;\F_2)\ar[r,"\cong" below] & H^1(X;\F_2)\rlap{.}
		\end{tikzcd}
	\end{equation*}
	 So the generator of $H^1(\R\P^{n+1};\F_2)$ is sent to the generator of $H^1(CX;\F_2)$. Using the structure of the cohomology rings of projective spaces, we see that the restriction morphism
	\begin{equation*}
		H^q(\R\P^{n+1};\F_2)\longrightarrow H^q(CX;\F_2)
	\end{equation*}
is an isomorphism for all $q\leq p+1$.
\end{proof}
\begin{figure}[!ht]
	\centering
	\begin{tikzpicture}[scale=2]
		\draw (0,1,0) -- (1,1,0);
		\draw (1,1,0) -- (0,1,1);
		\draw (0,1,1) -- (0,1,0);
		\draw[dashed] (1,0,0) -- (0,0,0) -- (0,0,1);
		\draw (1,0,0) -- (0,0,1);
		\draw (1,0,0) -- (1,1,0);
		\draw (0,0,1) -- (0,1,1);
		\draw[dashed] (0,0,0) -- (0,1,0);
		\fill (0,0,0) circle (.05) node[anchor=south west] {$(v_0;0)$};
		\fill (1,0,0) circle (.05) node[anchor=north west] {$(v_2;0)$};
		\fill (0,0,1) circle (.05) node[anchor=north east] {$(v_1;0)$};
		\fill (0,1,0) circle (.05) node[anchor=south west] {$(v_0;1)$};
		\fill (1,1,0) circle (.05) node[anchor=south west] {$(v_2;1)$};
		\fill (0,1,1) circle (.05) node[anchor=south east] {$(v_1;1)$};
		\draw[->, thick] (1.25,.5,0) -- (2.25,.5,0);
		\coordinate (v) at (3,0,0);
		\coordinate (u) at (1,.5,0);
		\coordinate (u1) at (0,0,0);
		\coordinate (u2) at (1.35,-.5,0);

		\draw ($(v)+(u)+(0,1,0)$) -- ($(v)+(u)+(1,1,0)$);
		\draw ($(v)+(u)+(1,1,0)$) -- ($(v)+(u)+(0,1,1)$);
		\draw ($(v)+(u)+(0,1,1)$) -- ($(v)+(u)+(0,1,0)$);
		\draw ($(v)+(u)+(1,1,0)$) -- ($(v)+(u)+(1,0,0)$);
		\draw ($(v)+(u)+(0,1,1)$) -- ($(v)+(u)+(1,0,0)$);
		\draw[dashed] ($(v)+(u)+(0,1,0)$) -- ($(v)+(u)+(1,0,0)$);
		\fill ($(v)+(u)+(0,1,0)$) circle (.05);
		\fill ($(v)+(u)+(1,1,0)$) circle (.05);
		\fill ($(v)+(u)+(0,1,1)$) circle (.05);
		\fill ($(v)+(u)+(1,0,0)$) circle (.05);
		
		\draw[dashed] ($(v)+(u1)+(0,1,0)$) -- ($(v)+(u1)+(0,0,1)$);
		\draw ($(v)+(u1)+(0,0,1)$) -- ($(v)+(u1)+(0,1,1)$);
		\draw ($(v)+(u1)+(0,1,1)$) -- ($(v)+(u1)+(0,1,0)$);
		\draw ($(v)+(u1)+(0,0,1)$) -- ($(v)+(u1)+(1,0,0)$);
		\draw ($(v)+(u1)+(0,1,1)$) -- ($(v)+(u1)+(1,0,0)$);
		\draw ($(v)+(u1)+(0,1,0)$) -- ($(v)+(u1)+(1,0,0)$);
		\fill ($(v)+(u1)+(0,1,0)$) circle (.05);
		\fill ($(v)+(u1)+(0,0,1)$) circle (.05);
		\fill ($(v)+(u1)+(0,1,1)$) circle (.05);
		\fill ($(v)+(u1)+(1,0,0)$) circle (.05);
		
		\draw ($(v)+(u2)+(0,1,0)$) -- ($(v)+(u2)+(0,0,1)$);
		\draw[dashed] ($(v)+(u2)+(0,0,1)$) -- ($(v)+(u2)+(0,0,0)$);
		\draw[dashed] ($(v)+(u2)+(0,0,0)$) -- ($(v)+(u2)+(0,1,0)$);
		\draw ($(v)+(u2)+(0,0,1)$) -- ($(v)+(u2)+(1,0,0)$);
		\draw[dashed] ($(v)+(u2)+(0,0,0)$) -- ($(v)+(u2)+(1,0,0)$);
		\draw ($(v)+(u2)+(0,1,0)$) -- ($(v)+(u2)+(1,0,0)$);
		\fill ($(v)+(u2)+(0,1,0)$) circle (.05);
		\fill ($(v)+(u2)+(0,0,1)$) circle (.05);
		\fill ($(v)+(u2)+(0,0,0)$) circle (.05);
		\fill ($(v)+(u2)+(1,0,0)$) circle (.05);
		
	\end{tikzpicture}
	\caption{Triangulation of the product of a segment with a triangle.}
	\label{fig:sub_segm_x_triang}
\end{figure}
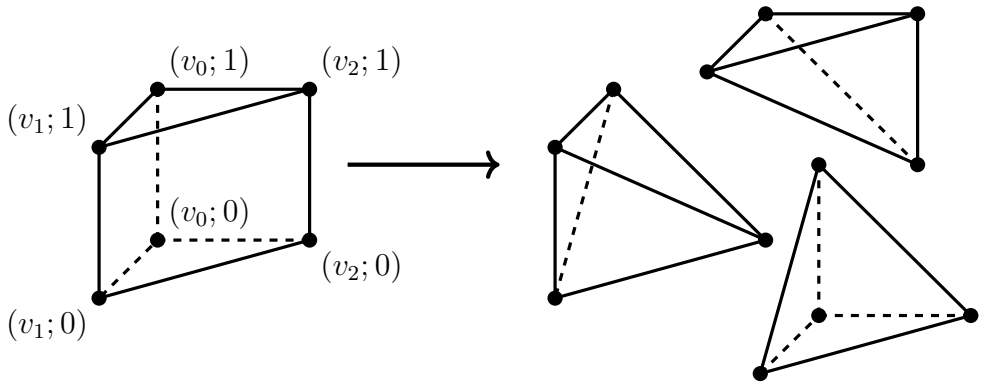

We recall that for a simplex $\sigma^n$ whose vertices are ordered, we have a canonical triangulation of the prism $\sigma^n\times[0;1]$. Its maximal simplices are the joins ${[v_0;\ldots;v_i]\times\{1\}*[v_i;\ldots;v_n]\times\{0\}}$, where $0\leq i\leq n$ and $[v_1;\ldots;v_k]$ is the face of $\sigma$ indexed by the vertices $v_1,\ldots,v_k$. See Figure~\ref{fig:sub_segm_x_triang} for an example. This triangulation allows us to construct triangulations of the standards simplices inductively.

\begin{dfn}\label{dfn:K+L}
	Let $n,d\geq 1$ be two integers and $\P^n_d$ be the convex hull of $\{0;de_1;\ldots;de_n\}\subset \R^n$. If $K$ is a triangulation of $\P^n_d$ and $L$ is a triangulation of $\P^{n-1}_{d+1}$, we define $K+L$ as the unique triangulation of $\P^n_{d+1}$ for which%
	\begin{enumerate}
		\item the restriction of $K+L$ to $\P^n_{d+1}\cap\{x_n\geq 1\}$ is the translation of $K$;
		\item the restriction of $K+L$ to $[e_n,e_n+de_1,\ldots,e_n+de_i]*[(d+1)e_i,\ldots,(d+1)e_{n-1}]$, for all $0\leq i\leq n-1$, is the join of the restriction of the translation of $K$ to $[e_n,e_n+de_1,\ldots,e_n+de_i]$ with the restriction of~$L$ to $[(d+1)e_i,\ldots,(d+1)e_{n-1}]$. 
	\end{enumerate}%
	 Figure~\ref{fig:triang_K+L} depicts an example of such a triangulation. 
\end{dfn}

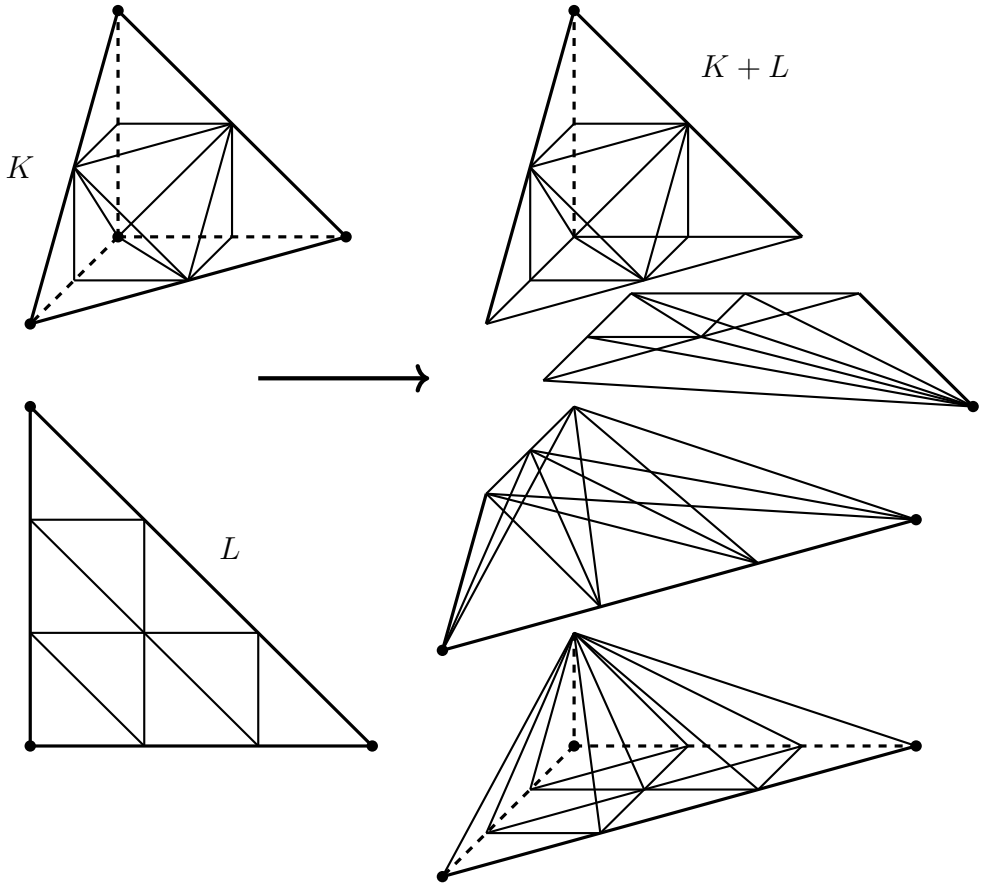
\begin{figure}[!ht]
	\centering
	
	\begin{tikzpicture}[scale=1.5]
		
		\fill (0,0,0) circle (.05);
		\fill (0,0,2) circle (.05);
		\fill (2,0,0) circle (.05);
		\fill (0,2,0) circle (.05);
		\draw[dashed,  thick] (0,0,0) -- (0,0,2);
		\draw[dashed,  thick] (0,0,0) -- (2,0,0); 
		\draw[dashed,  thick] (0,0,0) -- (0,2,0); 
		\draw[ thick] (2,0,0) -- (0,0,2);
		\draw[ thick] (2,0,0) -- (0,2,0); 
		\draw[ thick] (0,2,0) -- (0,0,2);
		\draw (0,1,0) -- (1,1,0);
		\draw (1,1,0) -- (0,1,1);
		\draw (0,1,0) -- (0,1,1);
		\draw (0,0,0) -- (0,1,1);
		\draw (0,0,0) -- (1,1,0);
		\draw (1,0,1) -- (0,1,1);
		\draw (1,0,1) -- (1,1,0);
		\draw (0,0,0) -- (1,0,1);
		\draw (1,0,0) -- (1,0,1);
		\draw (1,0,0) -- (1,1,0);
		\draw (0,0,1) -- (1,0,1);
		\draw (0,0,1) -- (0,1,1);
		\draw (-.25,1,1) node[anchor=east] {$K$};
		
		\coordinate (u) at (-.77,-4.5,0);
		\fill ($(u)+(0,0,0)$) circle (.05);
		\fill ($(u)+(3,0,0)$) circle (.05);
		\fill ($(u)+(0,3,0)$) circle (.05);
		\draw[ thick]  ($(u)+(0,0,0)$) -- ($(u)+(3,0,0)$);
		\draw[ thick]  ($(u)+(0,3,0)$) -- ($(u)+(3,0,0)$);
		\draw[ thick]  ($(u)+(0,0,0)$) -- ($(u)+(0,3,0)$);
		\draw ($(u)+(1,0,0)$) -- ($(u)+(0,1,0)$);
		\draw ($(u)+(2,0,0)$) -- ($(u)+(0,2,0)$);
		\draw ($(u)+(1,0,0)$) -- ($(u)+(1,2,0)$);
		\draw ($(u)+(2,0,0)$) -- ($(u)+(2,1,0)$);
		\draw ($(u)+(0,1,0)$) -- ($(u)+(2,1,0)$);
		\draw ($(u)+(0,2,0)$) -- ($(u)+(1,2,0)$);
		\draw ($(u)+(1.75,1.75,0)$) node {$L$};
		\draw[->, thick] ($(u)+(2,3.25,0)$) -- ($(u)+(3.5,3.25,0)$);
		 
		\coordinate (v) at (4,0,0);
		 
		\draw ($(v)+(1.5,1.5,0)$) node {$K+L$};
		\fill ($(v)+(0,2,0)$) circle (.05); 
		\draw[dashed,  thick] ($(v)+(0,0,0)$) -- ($(v)+(0,2,0)$); 
		\draw[ thick] ($(v)+(2,0,0)$) -- ($(v)+(0,2,0)$); 
		\draw[ thick] ($(v)+(0,2,0)$) -- ($(v)+(0,0,2)$);
		\draw ($(v)+(0,0,0)$) -- ($(v)+(2,0,0)$);
		\draw ($(v)+(0,0,0)$) -- ($(v)+(0,0,2)$);
		\draw ($(v)+(2,0,0)$) -- ($(v)+(0,0,2)$);
		\draw ($(v)+(0,1,0)$) -- ($(v)+(1,1,0)$);
		\draw ($(v)+(1,1,0)$) -- ($(v)+(0,1,1)$);
		\draw ($(v)+(0,1,0)$) -- ($(v)+(0,1,1)$);
		\draw ($(v)+(0,0,0)$) -- ($(v)+(0,1,1)$);
		\draw ($(v)+(0,0,0)$) -- ($(v)+(1,1,0)$);
		\draw ($(v)+(1,0,1)$) -- ($(v)+(0,1,1)$);
		\draw ($(v)+(1,0,1)$) -- ($(v)+(1,1,0)$);
		\draw ($(v)+(0,0,0)$) -- ($(v)+(1,0,1)$);
		\draw ($(v)+(1,0,0)$) -- ($(v)+(1,0,1)$);
		\draw ($(v)+(1,0,0)$) -- ($(v)+(1,1,0)$);
		\draw ($(v)+(0,0,1)$) -- ($(v)+(1,0,1)$);
		\draw ($(v)+(0,0,1)$) -- ($(v)+(0,1,1)$);
		
		\coordinate (v3) at (4.5,-.5,0);
		
		\draw ($(v3)+(0,0,0)$) -- ($(v3)+(0,0,2)$);
		\draw ($(v3)+(2,0,0)$) -- ($(v3)+(0,0,2)$);
		\draw ($(v3)+(2,0,0)$) -- ($(v3)+(0,0,0)$);
		\draw ($(v3)+(0,0,0)$) -- ($(v3)+(3,-1,0)$);
		\draw ($(v3)+(0,0,2)$) -- ($(v3)+(3,-1,0)$);
		\draw ($(v3)+(1,0,0)$) -- ($(v3)+(3,-1,0)$);
		\draw ($(v3)+(0,0,1)$) -- ($(v3)+(3,-1,0)$);
		\draw ($(v3)+(1,0,1)$) -- ($(v3)+(3,-1,0)$);
		\draw ($(v3)+(1,0,1)$) -- ($(v3)+(0,0,0)$);
		\draw ($(v3)+(1,0,1)$) -- ($(v3)+(1,0,0)$);
		\draw ($(v3)+(1,0,1)$) -- ($(v3)+(0,0,1)$);
		\draw[ thick] ($(v3)+(2,0,0)$) -- ($(v3)+(3,-1,0)$);
		\fill ($(v3)+(3,-1,0)$) circle (.05);
		
		\coordinate (v2) at (4,-1.5,0);
		
		\draw[ thick] ($(v2)+(0,-1,0)+(0,0,3)$) -- ($(v2)+(0,-1,0)+(3,0,0)$);
		\draw ($(v2)+(0,-1,0)+(0,1,0)$) -- ($(v2)+(0,-1,0)+(1,0,2)$);
		\draw ($(v2)+(0,-1,0)+(0,1,0)$) -- ($(v2)+(0,-1,0)+(2,0,1)$);
		\draw ($(v2)+(0,-1,0)+(0,1,0)$) -- ($(v2)+(0,-1,0)+(0,0,3)$);
		\draw[ thick] ($(v2)+(0,-1,0)+(0,1,2)$) -- ($(v2)+(0,-1,0)+(0,0,3)$);
		\draw ($(v2)+(0,-1,0)+(0,1,0)$) -- ($(v2)+(0,-1,0)+(3,0,0)$);
		\draw ($(v2)+(0,-1,0)+(0,1,0)$) -- ($(v2)+(0,-1,0)+(0,1,2)$);
		\draw ($(v2)+(0,-1,0)+(3,0,0)$) -- ($(v2)+(0,-1,0)+(0,1,1)$);
		\draw ($(v2)+(0,-1,0)+(2,0,1)$) -- ($(v2)+(0,-1,0)+(0,1,1)$);
		\draw ($(v2)+(0,-1,0)+(1,0,2)$) -- ($(v2)+(0,-1,0)+(0,1,1)$);
		\draw ($(v2)+(0,-1,0)+(0,0,3)$) -- ($(v2)+(0,-1,0)+(0,1,1)$);
		\draw ($(v2)+(0,-1,0)+(3,0,0)$) -- ($(v2)+(0,-1,0)+(0,1,2)$);
		\draw ($(v2)+(0,-1,0)+(2,0,1)$) -- ($(v2)+(0,-1,0)+(0,1,2)$);
		\draw ($(v2)+(0,-1,0)+(1,0,2)$) -- ($(v2)+(0,-1,0)+(0,1,2)$);
		\fill ($(v2)+(3,-1,0)$) circle (.05);
		\fill ($(v2)+(0,-1,3)$) circle (.05);
		
		\coordinate (v1) at (4,-3.5,0);
		
		\draw[dashed,  thick] ($(v1)+(0,0,0)$) -- ($(v1)+(0,-1,0)$);
		\draw[dashed,  thick] ($(v1)+(0,-1,0)$) -- ($(v1)+(0,-1,3)$);
		\draw[dashed,  thick] ($(v1)+(0,-1,0)$) -- ($(v1)+(3,-1,0)$);
		\draw[ thick] ($(v1)+(3,-1,0)$) -- ($(v1)+(0,-1,3)$);
		\fill ($(v1)+(0,-1,0)$) circle (.05);
		\fill ($(v1)+(0,-1,3)$) circle (.05);
		\fill ($(v1)+(3,-1,0)$) circle (.05);
		
		\draw ($(v1)+(0,-1,0)+(1,0,0)$) -- ($(v1)+(0,-1,0)+(0,0,1)$);
		\draw ($(v1)+(0,-1,0)+(2,0,0)$) -- ($(v1)+(0,-1,0)+(0,0,2)$);
		\draw ($(v1)+(0,-1,0)+(1,0,0)$) -- ($(v1)+(0,-1,0)+(1,0,2)$);
		\draw ($(v1)+(0,-1,0)+(2,0,0)$) -- ($(v1)+(0,-1,0)+(2,0,1)$);
		\draw ($(v1)+(0,-1,0)+(0,0,1)$) -- ($(v1)+(0,-1,0)+(2,0,1)$);
		\draw ($(v1)+(0,-1,0)+(0,0,2)$) -- ($(v1)+(0,-1,0)+(1,0,2)$);
		
		\draw ($(v1)+(0,-1,0)+(1,0,0)$) -- ($(v1)+(0,-1,0)+(0,1,0)$);
		\draw ($(v1)+(0,-1,0)+(2,0,0)$) -- ($(v1)+(0,-1,0)+(0,1,0)$);
		\draw ($(v1)+(0,-1,0)+(0,0,1)$) -- ($(v1)+(0,-1,0)+(0,1,0)$);
		\draw ($(v1)+(0,-1,0)+(0,0,2)$) -- ($(v1)+(0,-1,0)+(0,1,0)$);
		\draw ($(v1)+(0,-1,0)+(1,0,2)$) -- ($(v1)+(0,-1,0)+(0,1,0)$);
		\draw ($(v1)+(0,-1,0)+(2,0,1)$) -- ($(v1)+(0,-1,0)+(0,1,0)$);
		\draw ($(v1)+(0,-1,0)+(1,0,1)$) -- ($(v1)+(0,-1,0)+(0,1,0)$);
		\draw ($(v1)+(0,-1,0)+(3,0,0)$) -- ($(v1)+(0,-1,0)+(0,1,0)$);
		\draw ($(v1)+(0,-1,0)+(0,0,3)$) -- ($(v1)+(0,-1,0)+(0,1,0)$);
		
		\end{tikzpicture}
	
	\caption{A triangulation $K+L$. The boundary edges of the simplices are thickened.}
	\label{fig:triang_K+L}
\end{figure}

\begin{lem}\label{lem:convexity_sum}
	Let $P$ be a full-dimensional polytope of $\R^n$ endowed with a triangulation $K$ and $f,g\colon P\rightarrow \R$ be two functions, affine on every simplex of\, $K$. If $f$ is convex and, for any convex subset $A$ of $P$, $g|_A$ is convex as soon as $f|_A$ is affine, then there is an $\eta>0$ for which $f+\eta g$ is convex. 
\end{lem}

\begin{proof}
	Let $\sigma$ be a maximal simplex of $K$. For all functions $h\colon P\rightarrow \R$ affine on every simplex of $K$, we denote by $a(h;\sigma)$ the unique affine function of $\R^n$ satisfying $h|_\sigma=a(h;\sigma)|_\sigma$. For all vertices $v$ of $K$ and all maximal simplices $\sigma$ of $K$, we denote by $\delta(h;v;\sigma)$ the difference $h(v)-a(h;\sigma)(v)$. The function $h$ is convex if and only if%
	\begin{equation*}
		h=\max\limits_{\sigma\in K}\,a(h;\sigma),
	\end{equation*}%
	  which is itself equivalent to $\delta(h;v;\sigma)\geq 0$ for all vertices $v$ and maximal simplices $\sigma$. Moreover, if $h$ is convex, then $\delta(h;v;\sigma)$ vanishes if and only if $v$ and $\sigma$ belong to the same affinity component of $h$. Indeed, since $h$ is convex, its affinity components are the projections of the $n$-dimensional faces of the epigraph of $h$, \textit{i.e.}~the sets $\{p\in P\mid h(p)=a(h;\sigma)(p)\}$, where $\sigma$ is a maximal simplex of $K$. By assumption, we have $\delta(g;v;\sigma)\geq 0$ whenever $\delta(f;v;\sigma)=0$. We will distinguish two cases:\begin{enumerate}
	 \item If $\min\{\delta(g;v;\sigma)\colon v,\sigma\in K\}\geq 0$, the function $g$ is convex. Thus, $f+\eta g$ is convex for all $\eta>0$. 
	 \item In the other case, $f$ cannot be affine, and there is at least one couple $(v;\sigma)$ for which $\delta(f;v;\sigma)$ is positive. Any positive number $\eta$ satisfying%
	 \begin{equation*}
	 	0<\eta\leq \frac{\min\{\delta(f;v;\sigma)>0\colon v,\sigma\in K\}}{-\min\{\delta(g;v;\sigma)\colon v,\sigma \in K\}}
	 \end{equation*}%
	  yields a convex function $f+\eta g$, since $\delta(f+\eta g;v;\sigma)=\delta(f;v;\sigma)+\eta\delta(g;v;\sigma)\geq 0$ for all vertices $v$ and all maximal simplices $\sigma$.\qedhere
	 \end{enumerate}
\end{proof}

\begin{prop}\label{prop:convexity}
	If both $K$ and $L$ are convex, then so is $K+L$.
\end{prop}

\begin{proof}
	The triangulation $K+L$ of $\P^n_{d+1}$ is a subdivision of the first triangulation $M$ of $\P^n_{d+1}$ made of the translation of $\P^n_d$ by $e_n$ and the joins of the second point of Definition~\ref{dfn:K+L} subdividing the prism ${\P^n_{d+1}\cap\{0\leq x_n\leq 1\}}$. The triangulation $M$ is convex. It has exactly $n+1$ maximal simplices; the interior of each is a connected component of the complement in $\P^n_{d+1}$ of a union of $n$ affine hyperplanes. A direct computation yields their equations: $(d+1)(x_n-1)+\sum_{k=i}^{n-1} x_k$, for all $1\leq i\leq n-1$, and $x_n=1$. On that account, the function%
	\begin{equation*}
		\nu\colon x\in\P^n_d \longmapsto |x_n-1|+\sum_{i=1}^{n-1}\left|(d+1)(x_n-1)+\sum_{k=i}^{n-1}x_k\right|
	\end{equation*}%
	 certifies the convexity of $M$. Let $\nu_K$ and $\nu_L$ be two functions respectively certifying the convexity of $K$ and~$L$. We consider $[\nu_K+\nu_L]$, the unique function $\P^{n}_{d+1}\rightarrow \R$ affine on every simplex of $K+L$ whose restriction to $\P^{n}_{d+1}\cap\{x_n\geq 1\}$ is the translation of $\nu_K$ by $e_n$ and whose restriction to $\P^{n}_{d+1}\cap\{x_n=0\}$ is $\nu_L$. In other words, $[\nu_K+\nu_L]$ is defined on the joins of the bottom prism $\P^{n}_{d+1}\cap\{0\leq x_n\leq 1\}$  by the convex interpolation of $\nu_K(-+e_n)$ and $\nu_L$. By construction, this function is convex on the simplices of $M$. Therefore, following Lemma~\ref{lem:convexity_sum}, we can find an $\eta>0$ small enough so that $\nu+\eta[\nu_K+\nu_L]$ is convex. This function certifies the convexity of $K+L$.
\end{proof}

\begin{rem}
	Lemma~\ref{lem:convexity_sum} and Proposition~\ref{prop:convexity} are contained in \cite{Kem-Knu-Mum-Sai_tor_emb}.
\end{rem}


\begin{dfn}[Viro triangulations]\label{dfn:Viro_triang}
	The \emph{Viro triangulations}\index{Triangulation!Viro} $(V^n_d)_{n,d\geq 1}$ of $(\P^n_d)_{n,d\geq 1}$ are defined by the following recursive procedure: $V_1^n$ is the trivial triangulation, $V^1_d$ is the unique primitive triangulation of the segment, and $V^{n+1}_{d+1}=V^{n+1}_{d}+V^{n}_{d+1}$ for all $n,d\geq 1$. By recursion and Proposition~\ref{prop:convexity}, $V_d^n$ is primitive, integral, and convex for all $n,d\geq 1$.
\end{dfn}
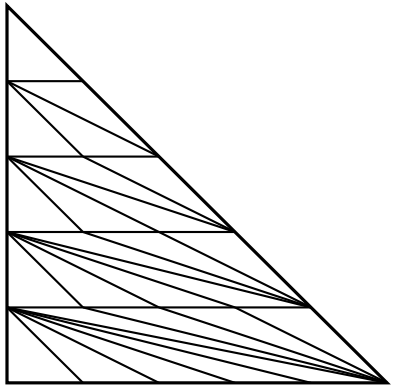
\begin{figure}[H]
	\centering
	\begin{tikzpicture}[scale=.7]
		\draw[thick] (0,0) -- (5,0) -- (0,5) -- cycle;
		\draw (0,1) -- (4,1);
		\draw (0,2) -- (3,2);
		\draw (0,3) -- (2,3);
		\draw (0,4) -- (1,4);
		\draw (0,4) -- (1,3);
		\draw (0,4) -- (2,3);
		\draw (0,3) -- (1,2);
		\draw (0,3) -- (2,2);
		\draw (0,3) -- (3,2);
		\draw (1,3) -- (3,2);
		\draw (0,2) -- (1,1);
		\draw (0,2) -- (2,1);
		\draw (0,2) -- (3,1);
		\draw (0,2) -- (4,1);
		\draw (1,2) -- (4,1);
		\draw (2,2) -- (4,1);
		\draw (0,1) -- (1,0);
		\draw (0,1) -- (2,0);
		\draw (0,1) -- (3,0);
		\draw (0,1) -- (4,0);
		\draw (0,1) -- (5,0);
		\draw (1,1) -- (5,0);
		\draw (2,1) -- (5,0);
		\draw (3,1) -- (5,0);
	\end{tikzpicture}
	\caption{The Viro triangulation $V^2_5$.}
\end{figure}

\begin{prop}\label{prop:heredity_triangulation}
	Let $n\geq 2$ and $d\geq 1$. The injective affine map%
	\begin{equation*}
		f^{n-1}_d\colon \P_d^{n-1} \longrightarrow \P_d^{n},
	\end{equation*}%
	 that sends $0$ to $de_1$ and $de_i$ to $de_{i+1}$, for all $1\leq i\leq n-1$, is simplicial with respect to $V_d^{n-1}$ and $V_d^n$.
\end{prop}

\begin{proof}
  The proposition is a consequence of the recursive construction of the Viro triangulations. Let us denote by $K^{n}_d$ the pullback to $\P^n_d$ of $V^{n+1}_d$ by $f^n_d$. By construction, $K^{n}_{d+1}$ equals $K^{n}_{d}+K^{n-1}_{d+1}$, $K^n_1$ is the trivial triangulation, and $K^1_d$ is the unique primitive triangulation of the segment. Therefore, $K_d^n$ is $V_d^n$ for all $d,n\geq 1$.
\end{proof}

\begin{thm}[Real Lefschetz property]\label{thm:rank_max}
	If $K$ is a Viro triangulation of $\P^{n+1}$, then the homological inclusion%
	\begin{equation*}
		i_q\colon H_q(\X;\F_2) \longrightarrow H_q(\R\P^n;\F_2),
	\end{equation*}%
	 is surjective for all $q\leq\left\lfloor\frac{n}{2}\right\rfloor$ and all $\varepsilon\in C^0(K;\F_2)$. That is to say, a projective hypersurface obtained from a primitive patchwork on a Viro triangulation always satisfies the real Lefschetz property.
\end{thm}

\begin{proof}
	Let $d\leq 1$ be an integer. We will prove the following statement by induction on $n\geq 0$:
\begin{center}\begin{minipage}{12cm}
		\emph{For all $\varepsilon \in C^0(V^{2q+1}_d;\F_2)$, the hypersurface $\X$ of\, $\R\P^{2n+1}$ contains a subspace $L_n$ that is homeomorphic to $\R\P^n$ and homologous to a linear subspace in $\R\P^{2n+1}$.}
        \end{minipage}	\end{center}
	
	Since $\X$ is never empty, the initial statement is satisfied. Let us assume that the statement is true for $n\geq 0$. We consider a sign distribution $\varepsilon\in C^0(V^{2n+3};\F_2)$ and the associated T-hypersurface $\X$ of $\R P^{2n+3}$. We consider several faces of $\P^{2n+3}_d$:
\begin{multicols}{2}
	\begin{enumerate}
		\item Let $H$ be the face of $\P^{2n+3}_d$ opposite to the vertex $de_{2n+3}$. It is, by construction, a $(2n+2)$-simplex of size $d$ endowed with the triangulation $V^{2n+2}_d$.
		\item Let $Q$ be the face of $\P^{2n+3}_d$ opposite to the vertex $0$. It is also a $(2n+2)$-simplex of size~$d$.
		\item Let $R$ be the face of $\P^{2n+3}_d$ opposite to the edge $[0;de_{2n+3}]$. It is a $(2n+1)$-simplex of size $d$. It is also the face of $H$ opposite to the vertex $0$. Thus, we can apply Proposition~\ref{prop:heredity_triangulation} to find that it is endowed with the triangulation $V^{2n+1}_d$.
\end{enumerate}
	\begin{figure}[H]\centering
	\begin{tikzpicture}[scale=4]
		\draw (0,0,0) -- (1,0,0);
		\draw (0,0,0) -- (0,1,0);
		\draw (0,0,0) -- (0,0,-1);
		\draw[very thick] (1,0,0) -- (0,0,-1);
		\draw (0,1,0) -- (0,0,-1);
		\draw (0,1,0) -- (1,0,0);
		\draw (0,0,0) node[anchor=north east] {$0$};
		\draw (1,0,0) node[anchor=north west] {$de_1$};
		\draw (0,0,-1) node[anchor=east] {$de_2\;$};
		\draw (0,1,0) node[anchor=south east] {$de_3$};
		\draw (.5,0,0) node[anchor=north]{$H$};
		\draw (.5,.65,0) node[anchor=north]{$Q$};
		\draw (.95,.25,0) node[anchor=north]{$R$};
		\fill[pattern=dots] (1,0,0) -- (0,1,0) -- (0,0,-1);
		\fill[pattern={Lines[angle=0, line width=.5pt]}] (1,0,0) -- (0,0,0) -- (0,0,-1);
	\end{tikzpicture}
	\caption{Positions of the faces $H$, $Q$, and $R$.}
	\label{fig:HQR}
	\end{figure}

	\end{multicols}
	Figure~\ref{fig:HQR} depicts the positions of $H$, $Q$, and $R$ in $\P^{3}_d$. The intersection $\X\cap\R R$ is the T-hypersurface of $\R R$ associated with the sign distribution $\varepsilon|_R$. Since $R$ is endowed with a Viro triangulation, we can apply the recursion hypothesis and find a subspace $L_n\subset \X\cap\R R$ that is homeomorphic to $\R\P^n$, and homologous to a linear subspace in $\R R$. Since $\R R$ is embedded as a linear subspace in $\R\P^{2n+3}$, $L_n$ is also homologous to a linear subspace in $\R\P^{2n+3}$. We recall that, by Definition~\ref{dfn:RP_CW}, the projective space $\R P^{2n+3}$ is obtained as the quotient of $\F_2^{2n+3}\times \P_d^{2n+3}$ by the equivalence relation%
	\begin{equation*} (v;x)\sim(v';x') \quad\textnormal{if and only if}\quad x=x' \textnormal{ and } v-v'\in \Sed(x).
	\end{equation*}
	For a point $x\in \P^{2n+3}_d$, $\Sed(x)$ denotes the reduction modulo $2$ of the lattice contained in the orthogonal\footnote{We implicitly identify $\R^{2n+3}$ with its dual using the integral scalar product $\langle\cdot;\cdot\rangle$ induced by its canonical basis.} of the smallest face that contains $x$. For instance, if $H$ is the smallest face that contains $x$, $\Sed(x)$ is $\F_2(0;\ldots;0;1)$. We define an intermediate quotient using the following equivalence relation:
	\begin{equation*}
		(v;x)\sim_B(v';x') \quad\textnormal{if and only if}\quad x=x' \textnormal{ and } \left\{ \begin{array}{ll} v-v'\in \Sed(x) & \textnormal{if }x\notin Q, \\ v=v' & \textnormal{otherwise}.\end{array}\right. 
	\end{equation*}	
	We denote the quotient $(\F_2^{2n+3}\times \P_d^{2n+3})/\sim_B$ by $B$, and the natural projection by $\pi \colon B\rightarrow \R\P^{2n+3}$. For all $(v;x)\in\F_2^{2n+3}\times \P^{2n+3}_d$, we may also denote $\pi(v;x)$ by $[v;x]$. Let $||\cdot||$ denote the Euclidean norm of $\R^{2n+3}$. The space $B$ is homeomorphic to the unit ball of $\R^{2n+3}$ via
	\begin{equation*}
		\begin{array}{rcl}
			f\colon B & \longrightarrow & B^{2n+3} \\
			~[v;x] & \longmapsto & \displaystyle \frac{\sum_{k=1}^{2n+3} x_i}{||x||}\left((-1)^{v_{k}}x_{k}\right)_{1\leq k\leq 2n+3}.
		\end{array} 
	\end{equation*}
\begin{figure}[H]
	\centering
	\begin{subfigure}[t]{0.45\textwidth}\centering
	\begin{tikzpicture}[scale=2.5]
		\draw[dashed] (1,0,0) --(-1,0,0);
		\draw[dashed] (0,1,0) --(0,-1,0);
		\draw[dashed] (0,0,1) --(0,0,-1);
		\draw (0,1,0) -- (0,0,1);
		\draw (-1,0,0) -- (0,0,1);
		\draw (-1,0,0) -- (0,1,0);
		\draw[dashed] (0,1,0) -- (0,0,-1) -- (-1,0,0) ;
		\draw (0,0,1) -- (1,0,0);
		\draw (1,0,0) -- (0,1,0);
		\draw[dashed] (1,0,0) -- (0,0,-1);
		\draw (-1,0,0) -- (0,-1,0);
		\draw (0,-1,0) -- (0,0,1);
		\draw (0,-1,0) -- (1,0,0) ;
		\draw[dashed] (0,-1,0) -- (0,0,-1);
	\end{tikzpicture}
	\caption{The ball $B$ obtained by gluing $8$ simplices.}
	\label{fig:ball_B}
	\end{subfigure}
	\hfill
	\begin{subfigure}[t]{0.45\textwidth}\centering
	\begin{tikzpicture}[scale=2.5]
		\draw[very thick, dotted] (-.7,-.7,.7) -- (.7,.7,-.7);
		\fill (-.333,-.333,.333) circle (0.035);
		\fill (.333,.333,-.333) circle (0.035);
		\draw (0,1,0) -- (0,0,1);
		\draw (-1,0,0) -- (0,0,1);
		\draw (-1,0,0) -- (0,1,0);
		\draw[dashed] (0,1,0) -- (0,0,-1) -- (-1,0,0) ;
		\draw (0,0,1) -- (1,0,0);
		\draw (1,0,0) -- (0,1,0);
		\draw[dashed] (1,0,0) -- (0,0,-1);
		\draw (-1,0,0) -- (0,-1,0);
		\draw (0,-1,0) -- (0,0,1);
		\draw (0,-1,0) -- (1,0,0) ;
		\draw[dashed] (0,-1,0) -- (0,0,-1);
	\end{tikzpicture}
	\caption{Antipodal identifications of the boundary of $B$.}
	\label{fig:anti_id}
	\end{subfigure}
	\hfill
	\begin{subfigure}[t]{.45\textwidth}\centering
	\begin{tikzpicture}[scale=2.5]
			\draw[very thick] (-1,0,0) -- (0,.5,0) -- (1,0,0);
			\draw[very thick] (1,0,0) -- (0,-.5,0) ;
			\draw[very thick, dashed](0,-.5,0) -- (-1,0,0);
			\draw[very thick] (0,0,1) -- (0,-.5,0);
			\draw[dashed, very thick] (0,0,-1) -- (0,-.5,0);

			\draw[very thick] (0,0,-1) -- (0,.5,0) -- (0,0,1);
			\draw[dashed, very thick] (-1,0,0) -- (0,0,-1); 
			\draw (0,1,0) -- (0,0,1);
			\draw[very thick] (-1,0,0) -- (0,0,1);
			\draw (-1,0,0) -- (0,1,0);
			\draw[dashed] (0,1,0)  -- (0,0,-1) ;
			\draw[very thick] (0,0,1) -- (1,0,0);
			\draw (1,0,0) -- (0,1,0);
			\draw[ very thick] (1,0,0) -- (0,0,-1);
			\draw (-1,0,0) -- (0,-1,0);
			\draw (0,-1,0) -- (0,0,1);
			\draw (0,-1,0) -- (1,0,0) ;
			\draw[dashed] (0,-1,0) -- (0,0,-1);
		\end{tikzpicture}
		\caption{The suspension $S_B$ of the equatorial disc as a subcomplex in $B$.}
		\label{fig:susp_B}
	\end{subfigure}
	\caption{The different polyhedral spaces considered in the proof of Theorem~\ref{thm:rank_max}.}
\end{figure}
	In dimension $3$, $B$ is an octahedron; see Figure~\ref{fig:ball_B}. Let $v_0\in\F_2^{2n+3}$ denote the vector $(1;\ldots;1)$. For every $[v;x]\in B$, we set $-[v;x]$ to be $[v+v_0;x]$. This endows $B$ with an involution for which $\pi$ induces an isomorphism between $\R P^{2n+3}$ and the quotient of $B$, where we identify $[v;x]\in\partial B$ with $-[v;x]$. The homeomorphism $f$ satisfies $f(-[v;x])=-f([v;x])$. Thus, $\pi$ can be thought of as the quotient of $B^{2n+3}$ by antipodal identification of the boundary points; \textit{cf.} Figure~\ref{fig:anti_id}. The inverse image of $\R Q$ by $\pi$ is the boundary $\partial B$. We are in the framework of Lemma~\ref{lem:Cone_over_projective}. Let us denote by $L^0_{n+1}\subset \R P^{2n+3}$ the cone $CL_n$. It is homeomorphic to $\R\P^{n+1}$ and homologous to a linear subspace of $\R\P^{2n+3}$. Our goal now is to move $L_{n+1}^0$ inside $\X$. The inverse image $\pi^{-1}L^0_{n+1}$ is a topological ball contained in the horizontal ball $\pi^{-1}\R H$ (note that $f(\pi^{-1}\R H)$ corresponds to $B^{2n+3}\cap \{x_{2n+3}=0\}$). By the construction of the Viro triangulation, the convex hull of $\{e_{2n+3}\}\cup H$ is the support of subcomplex $S$ of $V^{2n+3}_d$. Abstractly, $S$ is the simplicial join of $e_{2n+3}$ and $V^{2n+2}_d$. The image $S_B$ of $\F_2^{2n+3}\times S$ in $B$ is then a simplicial complex that is isomorphic to the suspension of the triangulation $\pi^{-1}\R V^{2n+1}_d$ of $\pi^{-1}\R H$. Its two apexes are, in the ``coordinates'' $[v;x]$ of~$B$, $[0;e_{2n+3}]$ and $[\overline{e}_{2n+3};e_{2n+3}]$, where $\overline{e}_{2n+3}$ is the reduction modulo $2$ of $e_{2n+3}$. Figure~\ref{fig:susp_B} illustrates this assertion. The intersection $\pi^{-1}\X\cap S_B$ is the dual hypersurface of $\df\eta$, where $\eta$ is the $0$-cochain of $S_B$ defined by $\eta\colon [v;p]\mapsto \varepsilon(p)+\langle p;v\rangle$. We have 
\begin{equation*}
		\eta([0;e_{2n+1}])+\eta([\bar{e}_{2n+3};e_{2n+3}])=\langle \bar{e}_{2n+3};\bar{e}_{2n+3}\rangle =1.
\end{equation*}
Thus, by Proposition~\ref{prop:isotopy_graph}, we can find a continuous map $g\colon [0;1]\times \pi^{-1}\R H \rightarrow |S_B|$ such that
\begin{enumerate}
	\item $x\mapsto g(0;x)$ is the inclusion $\pi^{-1}\R H \subset |S_B|$, 
	\item $x\mapsto g(1;x)$ is a homeomorphism between $\pi^{-1}\R H$ and $\pi^{-1}\X\cap|S_B|$, 
	\item $x\mapsto g(t;x)$ is an embedding for all $t\in[0;1]$, 
	\item $g(t;x)=x$ for every $t\in[0;1]$ and every $x\in \pi^{-1}\X\cap\pi^{-1}\R H$. 
	\end{enumerate}
	For every $t\in [0;1]$, we consider the equivalence relation $\sim_t$ on $\pi^{-1}L^0_{n+1}$ defined by 
	\begin{equation*}
		x\sim_t x' \quad\textnormal{if and only if}\quad \pi\left(g(t;x)\right)=\pi\left(g(t;x')\right). 
	\end{equation*}
	Two points $x$ and $x'$ of $\pi^{-1}L^0_{n+1}$ are $(\sim_t)$-equivalent if and only if $g(t;x)=g(t;x')$, or $g(t;x)$ and $g(t;x')$ both belong to $\partial B=\pi^{-1}\R Q$ and $g(t;x)=\pm g(t;x')$. We note that $\pi^{-1}L^0_{n+1}\cap\pi^{-1}\R Q=\pi^{-1}L_n$ is a subset of $\pi^{-1}\X\cap\pi^{-1}\R H$. Thus, the third and fourth properties of $g$ imply that $x\sim_t x'$ if and only if $x\sim_0 x'$, \textit{i.e.}~$\pi(x)=\pi(x')$. Hence, the composition $\pi\circ g$ induces a homotopy $h\colon [0;1]\times L^0_{n+1} \rightarrow \R\P^{2n+3}$ for which every partial map $x\mapsto h(t;x)$ is an embedding. The second property of $g$ ensures that $L^1_{n+1}:=h(1;L^0_{n+1})$ is included in $\X$. This space is homeomorphic to $\R\P^{n+1}$ by construction. Moreover, it is homotopic to $L^0_{n+1}$ in $\R\P^{2n+3}$; thus it is also homologous to a linear subspace. 
	
	\vspace{5pt}
	
	 Now, we know that every T-hypersurface $\X$ of $\R\P^{2n+1}$ constructed on a Viro triangulation $V^{2n+1}_d$ contains a subspace $L_n$ homeomorphic to $\R\P^{n}$ and homologous to a linear subspace. From this, we deduce the following commutative diagram, where all maps are induced by inclusions:%
	\begin{equation*}
		\begin{tikzcd}
			H_q(\X;\F_2) \ar[dr,"i_q" above] \\
			& H_q(\R\P^{2n+1};\F_2) \\
			H_q(L_n;\F_2)\rlap{.} \ar[ur,"j_q" below] \ar[uu]
		\end{tikzcd}
	\end{equation*}%
	 The map $j_q$ is surjective for all $0\leq q\leq n$, so $\ell(\X)\geq n=\left\lfloor \frac{\dim\X}{2}\right\rfloor$.
	
	\vspace{5pt}
	
	For the case of odd dimensions, let $\X$ be a T-hypersurface of $\R\P^{2n+2}$ constructed on a Viro triangulation $V^{2n+2}_d$. By the construction of the Viro triangulation, the face $H$ of $\P^{2n+2}_d$ opposite to the vertex $de_{2n+2}$ carries the Viro triangulation $V^{2n+1}_d$. Moreover, $\X\cap\R H$ is a T-hypersurface of $\R H$. Thus, we have the following commutative diagram, where all maps are induced by inclusions:
	\begin{equation*}
		\begin{tikzcd}
			H_q(\X;\F_2) \ar[r,"i_q" above] & H_q(\R\P^{2n+2};\F_2)\\
			H_q(\X\cap\R H;\F_2) \ar[u] \ar[r,"j_q" below] & H_q(\R H;\F_2)\rlap{.} \ar[u,"k_p" right]
		\end{tikzcd}
	\end{equation*}
	The maps $j_q$ and $k_q$ are surjective for all $0\leq q\leq n$; thus $\ell(\X)\geq n=\left\lfloor \frac{\dim\X}{2}\right\rfloor$.
\end{proof}

Combining this statement with Corollary~\ref{cor:charact_degeneracy2}, we find the following corollary. 

\begin{cor}\label{cor:dege_viro_even}
	The Renaudineau--Shaw spectral sequences computing the homology and the cohomology of the hypersurface $\X\subset\R\P^n$ constructed from a Viro triangulation $K$ and a sign distribution $\varepsilon\in C^0(K;\F_2)$ degenerates at the second page. 
\end{cor}

A.~Renaudineau and K.~Shaw's conjecture can be rephrased as $r(\X)\leq 2$ or equivalently the inequality ${\ell(\X)\geq\left\lfloor \frac{n}{2}\right\rfloor-1}$. However, as in Theorem~\ref{thm:rank_max}, we believe the stronger inequality ${\ell(\X)\geq\left\lfloor \frac{n-1}{2}\right\rfloor}$ might even be true in full generality, \textit{i.e.}~that every T-hypersurface satisfies the real Lefschetz property.


\newcommand{\etalchar}[1]{$^{#1}$}
\providecommand{\bysame}{\leavevmode\hbox to3em{\hrulefill}\thinspace}


\begin{thebibliography}{BLdMR22+++}

\bibitem[ARS21]{Arn-Ren-Sha_Lef_sec}
C.~Arnal, A.~Renaudineau, and K.~Shaw, \emph{Lefschetz section theorems for
  tropical hypersurfaces}, Ann.\ H.~Lebesgue \textbf{4} (2021),
  1347--1387, \doi{10.5802/ahl.104}.

\bibitem[BFM\etalchar{+}06]{Bih-Fra-Mcc-Ham_eve_tor}
  F.~Bihan, M.~Franz, C.~McCrory,  and J.~van Hamel, \emph{Is Every Toric Variety an M-Variety?},  Manuscripta Math.\ \textbf{120} (2006), no.~2,
  217--232, \doi{10.1007/s00229-006-0004-z}.

\bibitem[BH61]{Bor-Hae_cla_hom}
A.~Borel and A.~Haefliger, \emph{La classe d'homologie fondamentale d'un
  espace analytique}, Bull.\ Soc.\ Math.\ France \textbf{89} (1961), 461--513, \doi{10.24033/bsmf.1571}.

\bibitem[BLdMR22]{Brug-LdM-Rau_Comb_pac}
E.~Brugall{\'e}, L.~L{\'o}pez~de Medrano, and J.~Rau, \emph{Combinatorial
  Patchworking: Back from Tropical Geometry}, preprint \arXiv{2209.14043} (2022).

\bibitem[BM71]{Bru-Man_she_dec}
H.~Bruggesser and P.~Mani, \emph{Shellable Decompositions of Cells and
  Spheres}, Math.\ Scand.\ \textbf{29} (1971), no.~2, 197--205, \doi{10.7146/math.scand.a-11045}.

\bibitem[CE56]{Car-Eil_hom_alg}
H.~Cartan and S.~Eilenberg, \emph{Homological Algebra}, Princeton
Math.\ Ser., vol.~19, Princeton Univ.\ Press, Princeton, NJ, 1956,
\doi{10.1515/9781400883844}.

\bibitem[DK99]{Deg-Kha_top_pro}
A.~Degtyarev and V.~Kharlamov, \emph{{Topological Properties of Real Algebraic
  Varieties: Rokhlin's Way}}, preprint
  \url{https://hal.science/hal-00129721} (1999). 
  
\bibitem[Del71]{Del_the_hod}
P.~Deligne, \emph{Th{\'e}orie de Hodge, II}, Inst.\ Hautes \'Etudes Sci.\ Publ.\ Math.\  \textbf{40} (1971),   5--57, \doi{10.1007/BF02684692}.
  
\bibitem[Ful93]{Ful_tor_var}
W.~Fulton, \emph{Introduction to Toric Varieties}, Ann.\ of Math.\ Stud., vol.~131, William H.\ Roever Lectures Geom., Princeton Univ.\ Press, Princeton, NJ,
1993, \doi{10.1515/9781400882526}.

\bibitem[GKZ94]{Gel-Kap-Zel_dis_res}
I.\,M.~Gelfand, M.\,M.~Kapranov, and A.\,V.~Zelevinsky, \emph{Discriminants,
Resultants, and Multidimensional Determinants}, Math.\ Theory Appl., 
Birkh{\"a}user Boston, Boston, MA, 1994, \doi{10.1007/978-0-8176-4771-1}.

\bibitem[Hat02]{Hat_alg_top}
A.~Hatcher, \emph{Algebraic Topology}, Cambridge Univ. Press, Cambridge,
  2002.
  
\bibitem[Ite97]{Ite_top_rea}
  I.~Itenberg, \emph{Topology of Real Algebraic $T$-surfaces},
in: \emph{Real algebraic and analytic geometry} (Segovia, 1995), pp.~131--152, 
  Rev.\  Mat.\ Univ.\ Complut.\ Madrid \textbf{10} (1997), suppl., Special Issue, \doi{10.5209/rev_REMA.1997.v10.17351}.

\bibitem[IKMZ19]{Ite-Kat-Mik-Zha_tro_hom}
I.~Itenberg, L.~Katzarkov, G.~Mikhalkin, and I.~Zharkov, \emph{Tropical
Homology}, Math.\ Ann.\ \textbf{374} (2019), no.~1-2,
963--1006,  \doi{10.1007/s00208-018-1685-9}.

\bibitem[JRS18]{Jel-Rau-Sha_Lef_11}
P.~Jell, J.~Rau, and K.~Shaw, \emph{Lefschetz $(1,1)$-theorem in tropical
  geometry},
{\'E}pijournal de G{\'e}om.\ Alg{\'e}brique
\textbf{2} (2018),  Art.~11,
\doi{10.46298/epiga.2018.volume2.4126}.

\bibitem[Kal91]{Kal_coh_cha}
I.\,O.~Kalinin, \emph{Cohomological Characteristics of Real Projective
  Hypersurfaces}, Algebra i Analiz \textbf{3} (1991), no.~2, 91--110.

\bibitem[Kal05]{Kal_coh_rea}
\bysame, \emph{Cohomology of Real Algebraic Varieties}, J.~Math.\ Sci.\ (N.Y.) \textbf{131} (2005), no.~1, 5323--5344,
  \doi{10.1007/s10958-005-0405-7}.

\bibitem[KKM\etalchar{+}73]{Kem-Knu-Mum-Sai_tor_emb}
G.~Kempf, F.~Knudsen, D.~Mumford, and B.~Saint-Donat, \emph{Toroidal
Embeddings. I}, Lecture Notes in Math., vol.~339, Springer-Verlag, Berlin-New York, 1973, \doi{10.1007/BFb0070318}.

                          
\bibitem[Kha75]{Kha_add_cong}
V.\,M.~Kharlamov, \emph{Additional Congruences for the Euler Characteristic of
  Real Algebraic Manifolds of Even Dimensions}, Funct.\ Anal.\ Appl.\ \textbf{9} (1975), 134--141, \doi{10.1007/BF01075451}.

\bibitem[Lan02]{Lan_alg}
S.~Lang, \emph{Algebra}, 3rd ed., Grad.\ Texts in Math., vol.~211, Springer-Verlag, New York, 2002, \doi{10.1007/978-1-4613-0041-0}. 

\bibitem[McC01]{McC_use_gui}
J.~McCleary, \emph{A User's Guide to Spectral Sequences}, 2nd ed., Cambridge
Stud.\ Adv.\ Math., vol.~58,
Cambridge Univ.\ Press, Cambridge, 2001.

\bibitem[MR18]{Mik-Rau_tro_geo}
G.~Mikhalkin and J.~Rau, \emph{Tropical Geometry}, preprint  
  \url{https://math.uniandes.edu.co/~j.rau/downloads/main.pdf} (2018).

\bibitem[Mun84]{Munk_ele_alg}
J.\,R.~Munkres, \emph{Elements of Algebraic Topology}, Addison-Wesley
Publishing Company, Menlo Park, CA, 1984.

\bibitem[Nik80]{Nik_int_sym}
V.\,V.~Nikulin, \emph{Integral symmetric bilinear forms and some of their
  applications}, Math.\ USSR-Izv.\ \textbf{14} (1980), no.~1,
103--167,
\doi{10.1070/IM1980v014n01ABEH001060}.

\bibitem[Qui68]{Qui_ass_gra}
D.\,G.~Quillen, \emph{On the Associated Graded Ring of a Group Ring}, J.~Algebra \textbf{10} (1968), no.~4, 411--418,
  \doi{10.1016/0021-8693(68)90069-0}.

\bibitem[RS23]{Ren-Sha_bou_bet}
A.~Renaudineau and K.~Shaw, \emph{Bounding the Betti Numbers of Real
  Hypersurfaces Near the Tropical Limit}, Ann.\ Sci.\ \'Ec.\ Norm.\ Sup\'er.~(4) \textbf{56} (2023), no.~3, 945--980, \doi{10.24033/asens.2547}. 

\bibitem[Ris93]{Ris_con_hyp}
J.\,J.~Risler, \emph{Construction d'hypersurfaces r\'eelles (d'apr\`es
  {Viro)}}, in: \emph{S\'eminaire Bourbaki, Vol.~1992/93, Expos\'es 760-774}, Exp.~No.~763, pp.~69--86, Ast\'erisque \textbf{216} (1993).

\bibitem[She85]{She_cel_des}
A.\,D.~Shepard, \emph{A Cellular Description of the Derived Category of a
Stratified Space}, PhD~thesis, Brown University, 1985, available at 
\url{https://justinmcurry.com/wp-content/uploads/2022/02/shepard.pdf}. 

\bibitem[Vir98]{Vir_mut_pos}
O.~Viro, \emph{Mutual Position of Hypersurfaces in Projective Space},
  in: \emph{Geometry of differential equations}, pp.~161--176, Amer.\ Math.\ Soc.\ Transl.\ Ser.~2, vol.~186, Amer.\ Math.\ Soc., Providence, RI, 1998, \doi{10.1090/trans2/186/06}.
 
\bibitem[Vir06]{Vir_pat_rea}
\bysame, \emph{{Patchworking Real Algebraic Varieties}},
  preprint \arXiv{math/0611382} (2006).

\bibitem[Wel08]{Wel_dif_ana}
R.\,O.~Wells, \emph{{Differential Analysis on Complex Manifolds}}, 3rd ed.,
  Grad.\ Texts in Math., Springer New York, NY, 2008, \doi{10.1007/978-0-387-73892-5}.

\end{thebibliography}
\end{document}